\documentclass[10pt]{articleHJ}

\usepackage[centertags]{amsmath}
\usepackage{amsfonts,amssymb,amsthm}
\usepackage{caption,graphicx}
\usepackage[pagewise]{lineno}%\linenumbers
\usepackage{subcaption}
\usepackage[text={6.0in,8.6in},centering,letterpaper]{geometry}
\usepackage[numbers,comma,square,sort&compress]{natbib}

\usepackage{tabularx}

\newlength{\defbaselineskip}
\newcommand{\setlinespacing}[1]%
           {\setlength{\baselineskip}{#1 \defbaselineskip}}

\newcommand{\singlespacing}{\setlength{\baselineskip}{\defbaselineskip}}

\newtheorem{thm}{Theorem}[section]
\newtheorem{cor}[thm]{Corollary}
\newtheorem{lem}[thm]{Lemma}
\newtheorem{prop}[thm]{Proposition}

\theoremstyle{definition}

\newcommand{\Real}{\mathbb{ R}}

\newcommand{\R}{\mathbb{ R}}
\newcommand{\Z}{\mathbb{ Z}}
\newcommand{\Wholes}{\mathbb{ Z}}

\newcommand{\Complex}{\mathbb{ C}}
\newcommand{\abs}[1]{\left\vert#1\right\vert}			% vector norm
\newcommand{\norm}[1]{\left\Vert#1\right\Vert}		% operator norm
\newcommand{\sref}[1]{(\ref{#1})}                       % equaton references
\numberwithin{equation}{section}

\begin{document}

\begin{frontmatter}
\title{Travelling Corners for Spatially Discrete \\ Reaction-Diffusion Systems }
\journal{...}

\author[LD2]{H. J. Hupkes},
\author[LD1]{L. Morelli \corauthref{coraut}}
\corauth[coraut]{Corresponding author. }

\address[LD1]{
  Mathematisch Instituut - Universiteit Leiden \\
  P.O. Box 9512; 2300 RA Leiden; The Netherlands \\
  Email:  {\normalfont{\texttt{leonardo.morelli@gmail.com}}}
}
\address[LD2]{
  Mathematisch Instituut - Universiteit Leiden \\
  P.O. Box 9512; 2300 RA Leiden; The Netherlands \\ Email:  {\normalfont{\texttt{hhupkes@math.leidenuniv.nl}}}
}

\date{Version of \today}

\begin{abstract}
\singlespacing
We consider reaction-diffusion equations
on the planar square lattice that admit
spectrally stable planar travelling wave solutions.
We show that these solutions can be continued into a branch
of travelling corners. As an example, we consider
the monochromatic and bichromatic Nagumo lattice differential equation
and show that both systems exhibit interior and exterior corners.

Our result is valid in the setting where the group velocity is zero. In this case,
the equations for the corner can be written as a difference equation posed on
an appropriate Hilbert space. Using a non-standard global center manifold reduction,
we recover a two-component difference equation that describes the behaviour
of solutions that bifurcate off the planar travelling wave.
The main technical complication is the lack of regularity caused by the spatial discreteness,
which prevents the symmetry group from being factored out in a standard fashion.
\end{abstract}



\end{frontmatter}

\section{Introduction}

In this paper we construct travelling corner
solutions to a class of planar lattice differential equations
(LDEs) that includes the Nagumo LDE
\begin{equation}
\label{eq:int:nag:lde}
\dot u_{i,j} =u_{i+1,j} +u_{i-1,j} +u_{i,j+1} +u_{i,j-1} -4u_{i,j}
+g_{\mathrm{cub}}(u;\rho)
\end{equation}
posed on the two-dimensional square lattice $(i,j) \in \mathbb{Z}^2$,
in which the nonlinearity is given by the bistable cubic
\begin{equation}
g_{\mathrm{cub}}(u;\rho) =  (u^2 - 1)(\rho - u) ,
\qquad \qquad
-1 < \rho < 1 .
\end{equation}
%with the standard bistable.
Such corners can be seen as interfaces that connect
planar waves travelling in slightly different directions.
In particular, our analysis does not require the use of the comparison principle,
but merely requires a number of spectral and geometric conditions to hold
for the underlying planar travelling waves.
This allows our results to be applied to a wide range of LDEs,
highlighting the important role that anisotropy and topology
play in spatially discrete settings.

\paragraph{Reaction-diffusion systems}
The LDE \sref{eq:int:nag:lde} can be seen as
a nearest-neighbour spatial discretization of the Nagumo PDE
\begin{equation}
\label{eq:int:nag:pde}
u_t = u_{xx} + u_{yy} +  g_{\mathrm{cub}}(u;\rho) .
\end{equation}
In modelling contexts one often uses the two stable
equilibria of the nonlinearity $g$ to represent material phases or
biological species that compete for dominance in a spatial domain.
Indeed, the diffusion term tends to attenuate
high frequency oscillations, while the bistable nonlinearity
promotes these. The
balance between these two
dynamical features leads to interesting pattern forming
behaviour.

As a consequence,
the PDE \sref{eq:int:nag:pde} % with $\Omega = \Real^2$
has served as a prototype system for the understanding
of many basic concepts at the heart of dynamical systems theory, including the
existence and stability of planar travelling waves,
the expansion of localized structures and the study of obstacles.
Multi-component versions of \sref{eq:int:nag:pde} such as the Gray-Scott
model \cite{GRAYSCOTT1983}
play an important role in the formation of patterns,
generating spatially periodic structures from equilibria that destabilize through Turing
bifurcations. Memory devices have been designed
using FitzHugh-Nagumo-type systems with two components \cite{KAMIN2006},
which support stable stationary radially symmetric spot patterns.
Similarly, one can find stable travelling spots \cite{HEIJSTER2014}
for three-component FitzHugh-Nagumo systems,
which have been used to describe gas discharges \cite{ORG1998,SCHENK1997}.

At present, a major effort is underway to understand the impact that
non-local effects can have on reaction-diffusion systems. For example,
many neural field models include infinite-range convolution terms
to describe the dynamics of large networks of neurons
\cite{BRESS2011,BRESS2014,PINTO2001,SNEYD2005},
which interact with each other over long distances.
The description of phase transitions in Ising models \cite{BatesInfRange,BatesDiscConv}
features non-local interactions that
can be both attractive and repulsive
depending on the length scale involved.

It is well-known by now that the
topology of the underlying spatial domain
can have a major impact on the dynamical behaviour exhibited by such non-local
systems. For example, nerve fibers have a myeline coating that admits
gaps at regular intervals \cite{RANVIER1878},
which can block signals
from propagating through the fiber \cite{VL7,HUX1949,LILLIE1925}.
In order to study the growth of plants,
one must take into account that cells
divide and grow in a fashion that is influenced heavily by the spatial configuration
of their neighbours \cite{GRIEN2007}.
Finally, the periodic structure inherent in many meta-materials
strongly influences the phase transitions that can occur
\cite{CAHNNOV1994,CAHNVLECK1999,VAIN2009}
as a consequence of the visco-elastic interactions between their building blocks.

We view the planar LDE \sref{eq:int:nag:lde} as a prototype model
that allows the impact of such non-local spatially-discrete effects
to be explored. Indeed, the
spatial $\Real^2 \to \Wholes^2$ transition
breaks the locality but also the translational and rotational symmetry of \sref{eq:int:nag:pde},
leading to several interesting phenomena and mathematical challenges.

\paragraph{Existence of planar waves}

It is well-known that
the balance between the diffusion and reaction
terms in the PDE \sref{eq:int:nag:pde}
is resolved through the formation of planar travelling
wave solutions
\begin{equation}
\label{eq:int:trv:wave:ansatz}
 u(x,y, t) = \Phi(x \cos \zeta + y \sin \zeta +ct); \qquad \Phi(-\infty) = -1,
   \qquad \Phi(\infty) = 1,
\end{equation}
which connect the two stable equilibria $u = \pm 1$.
When $c \neq 0$, these waves can be thought of as a mechanism by which the fitter species
or more energetically favourable phase invades the spatial domain.
The existence of these waves can be obtained by applying a phase-plane analysis
\cite{Fife1977} to the travelling wave ODE
\begin{equation}
  \label{eq:int:eq:waveProfile:pde}
 c\Phi' = \Phi'' + g_{\mathrm{cub}}(\Phi;\rho),
 \qquad
 \Phi(\pm \infty) = \pm 1,
\end{equation}
which results after substituting \sref{eq:int:trv:wave:ansatz}
into \sref{eq:int:nag:pde}.

On the other hand, substitution of the analogous Ansatz
\begin{equation}
\label{eq:int:eq:waveProfile:lde}
u_{ij}(t) = \Phi(i \cos \zeta + j \sin \zeta +ct); \qquad \Phi(-\infty) = -1,
   \qquad \Phi(\infty) = 1,
\end{equation}
into the LDE \sref{eq:int:nag:lde}
leads to the mixed-type functional differential equation (MFDE)
\begin{equation}
   \label{eq:int:trvWave:mfde}
  \begin{array}{lcl}
  c\Phi'(\xi) &= & \Phi(\xi + \cos \zeta) + \Phi(\xi + \sin \zeta)
    + \Phi(\xi - \cos \zeta) + \Phi(\xi - \sin \zeta) - 4\Phi(\xi)
     + g_{\mathrm{cub}}\big(\Phi(\xi); \rho \big).  % ; \\[0.2cm]
  %& & \qquad \Phi(-\infty) = 0, \qquad \Phi(\infty) = 1.
  \end{array}
\end{equation}
The broken rotational invariance in the transition
from \sref{eq:int:nag:pde} to \sref{eq:int:nag:lde}
is manifested by the
explicit presence of the propagation direction
in \sref{eq:int:trvWave:mfde}. The broken translational
invariance causes the wavespeed $c$ to appear in
\sref{eq:int:trvWave:mfde}
as a singular parameter.

A comprehensive existence theory for solutions to \sref{eq:int:trvWave:mfde}
was obtained in \cite{MPB}. In particular, for every $\zeta \in [0, 2 \pi]$
and $\rho \in [-1,1]$ there exists a unique wavespeed $c = c_{\rho,\zeta}$
for which \sref{eq:int:trvWave:mfde} admits a solution.
However, it is a delicate question to decide whether $c \neq 0$
or $c = 0$. Indeed, a sufficient energy difference between the two
stable equilibrium states is needed for the propagation of
waves \cite{BatesDiscConv,Bell1984, EVV2005AppMath, VL28},
a phenomenon referred to as propagation failure.
In fact, due to the angular dependence in \sref{eq:int:trvWave:mfde},
planar waves can fail to propagate
in certain directions that resonate with the lattice,
whilst travelling freely in others
\cite{CMPVV,HOFFMPcrys,MPCP}.

\paragraph{Linearization}
It is well-known that planar travelling waves
can be used as a skeleton to describe the global dynamics of
the PDE \sref{eq:int:trv:wave:ansatz} \cite{AW78}.
In particular, they have been used as building blocks to construct
other more complicated types of solutions.
A key ingredient in such constructions is to understand
the dynamics of the system that arises after linearizing
\sref{eq:int:nag:pde} around the planar waves
\sref{eq:int:trv:wave:ansatz}.

Performing this linearization for $\zeta =0 $,
we obtain the system
\begin{equation}
\partial_t v(x,y,t) =  \partial_{xx} v(x,y,t) + \partial_{yy}v(x,y,t)
 + g_{\mathrm{cub}}'\big(\Phi(x +ct) ; \rho \big) v(x,y,t),
\end{equation}
which can be transformed to the temporally autonomous system
\begin{equation}
\label{eq:int:pde:linearized:sys:transf}
\partial_t v(x,y,t) = \partial_{xx} v(x,y,t) + \partial_{yy} v(x,y,t)
  - c \partial_x v(x,y,t) + g_{\mathrm{cub}}'\big(\Phi(x); \rho\big) v(x,y,t)
\end{equation}
by the variable transformation $x' = x + c t$.
Since this system is also autonomous with respect to the $y$-coordinate,
which is transverse to the motion of the wave,
it is convenient to apply a Fourier transform in this direction.
Upon introducing the symbol
\begin{equation}
\label{eq:int:pde:l:z:def}
\big[\mathcal{L}_{z} p \big](x)
= \partial_{xx} p(x) + z^2  p(x) - c \partial_x p(x) + g_{\mathrm{cub}}'\big(\Phi(x); \rho \big) p(x),
\end{equation}
we readily find
\begin{equation}
\partial_t \hat{v}_{\omega}(x,t) =
\big[ \mathcal{L}_{i\omega} \hat{v}_{\omega}(\cdot, t)](x) .
\end{equation}

Inspecting \sref{eq:int:pde:l:z:def}, we readily see that
the spectrum of $\mathcal{L}_{z}$
can be obtained by rigidly shifting the spectrum of $\mathcal{L}_0$
by $z^2$. In particular, writing $\lambda_z = z^2$,
we find that
\begin{equation}
\mathcal{L}_{z} \Phi' = \lambda_z \Phi'.
\end{equation}
Noting that $\lambda_{i \omega} = - \omega^2$, we hence
see that %the evolution of
perturbations of the form
$v(x,y, 0) = \theta(y, 0) \Phi'(x)$ %can be evolved under
evolve under
\sref{eq:int:pde:linearized:sys:transf}
%can be determined
according to the heat semiflow $\theta_t = \theta_{yy}$.
These perturbations are important because they correspond
at the linear level with transverse deformations of the planar wave
interface.

On the other hand,
linearizing the LDE \sref{eq:int:nag:lde}
around the spatially-discrete wave \sref{eq:int:eq:waveProfile:lde}
that travels in the horizontal direction $\zeta = 0$,
we obtain  the temporally non-autonomous system
\begin{equation}
 \dot{v}_{ij}(t) = v_{i+1, j}(t) + v_{i, j+1}(t) + v_{i-1, j}(t) + v_{i, j-1}(t)
  - 4 v_{ij}(t) + g_{\mathrm{cub}}'\big(\Phi(i + ct) ; \rho \big) v_{ij}(t).
\end{equation}
Although the time dependence cannot be readily transformed away,
it is still possible to take a discrete Fourier transform in the $j$-direction.
This leads to the system
\begin{equation}
\label{eq:int:time:dep:lin:prob:lde}
\begin{array}{lcl}
\frac{d}{dt}[\hat{v}_{\omega}]_i(t)
& = &
[\hat{v}_{\omega}]_{i+1}(t) + e^{i\omega}[\hat{v}_{\omega}]_{i}(t)
 + [\hat{v}_{\omega}]_{i-1}(t) +
 e^{-i \omega}[\hat{v}_{\omega}]_{i}(t)
  - 4 [\hat{v}_{\omega}]_i(t)
\\[0.2cm]
& & \qquad + g_{\mathrm{cub}}'(\Phi(i + ct) ; \rho \big) [\hat{v}_{\omega}]_{i}(t).
\end{array}
\end{equation}
Looking for solutions of the form
\begin{equation}
 [\hat{v}_{\omega}]_i =
  e^{\lambda t} w_{\omega}( i + ct) ,
\end{equation}
we arrive the eigenvalue problem
\begin{equation}
\mathcal{L}^{\mathrm{hor}}_{i \omega} w_{\omega} = \lambda w_{\omega}
\end{equation}
for the linear operator
\begin{equation}
\begin{array}{lcl}
[\mathcal{L}^{\mathrm{hor}}_{z}p](\xi)
& = &  -c p'(\xi)
  + 2 \cosh(z)  p(\xi)
   + p(\xi +1) + p(\xi -1)
   - 4 p(\xi)
%\\[0.2cm]
%& & \qquad
  + g_{\mathrm{cub}}'\big(\Phi(\xi) ; \rho \big) p(\xi) .
\end{array}
\end{equation}
The theory developed in \cite{HJHNLS,BGV2003} essentially
justifies this formal calculation and confirms that the
spectral properties of
$\mathcal{L}_{z}$ can be used to
understand the dynamics of the time-dependent
problem \sref{eq:int:time:dep:lin:prob:lde}.
Upon writing $\lambda_z = 2 (\cosh(z) - 1)$,
we again have
\begin{equation}
\mathcal{L}^{\mathrm{hor}}_z \Phi' = \lambda_z \Phi' .
\end{equation}
To find the evolution of perturbations of the form
$v_{ij}(0) = \theta_j \Phi'$ under \sref{eq:int:time:dep:lin:prob:lde},
we must now solve the discrete heat equation
\begin{equation}
\label{eq:int:nag:lde:phase:heat:eq}
\dot{\theta}_j = \theta_{j+1} + \theta_{j-1} - 2 \theta_j .
\end{equation}
The situation is hence similar to that encountered
for the PDE \sref{eq:int:nag:pde}.

More material changes arise however when considering
the diagonal direction
$\zeta = \frac{\pi}{4}$. Following a similar procedure as above,
one arrives at the linear operator
\begin{equation}
\begin{array}{lcl}
[\mathcal{L}^{\mathrm{diag}}_{z}p](\xi)
& = &
 -c p'(\xi)
  + 2 \cosh (z )\big[ p(\xi + 1 ) + p(\xi -1) \big]
   - 4 p(\xi)
%\\[0.2cm]
%& & \qquad
  + g_{\mathrm{cub}}'\big(\Phi_*(\xi) ; \rho \big) p(\xi) ,
\end{array}
\end{equation}
which has a spectrum that can no longer be directly
related to that of $\mathcal{L}^{\mathrm{diag}}_0$.
It is hence no longer clear how to formulate
an analogue for  \sref{eq:int:nag:lde:phase:heat:eq}
to describe the linear evolution of interface deformations.
However, it is still the case that $\lambda_z = O(z^2)$
as $z \to 0$ for the curve of eigenvalues
that bifurcates from the zero eigenvalue
$\mathcal{L}^{\mathrm{diag}}_0 \Phi' = 0$.

For general rational angles $\zeta$ this quadratic behaviour
need no longer be true. In fact, we obtain the relation
\begin{equation}
\label{eq:int:def:grp:velocity}
[\partial_z \lambda_z]_{z=0} = \partial_\zeta c_{\rho,\zeta}
\end{equation}
for the quantity that is often referred to as the group velocity.
A similar relation was found in \cite{HARAG2006ANISO}
for planar PDEs with direction-dependent diffusion coefficients.
However, in this case it is always possible to
change the coordinate system in such a way that
$\lambda_z = O(z^2)$ holds again.

Such a transformation is not possible in the spatially discrete
setting \sref{eq:int:nag:lde}, since this would require
the transverse spatial coordinate
to become continuous. However, we do remark here that
the function $\phi \mapsto c_{\rho,\zeta}$ can behave
rather wildly in the critical regime where $\rho$ is small,
allowing the group velocity to vanish at specific values for
$\rho$ even if $\zeta \notin \frac{\pi}{4} \Wholes$.

\paragraph{Stability of planar waves}

The realization that transverse
interfacial deformations are governed by a heat equation
led to the development of two main approaches
to establish the nonlinear stability of the
planar waves \sref{eq:int:trv:wave:ansatz}.
Both approaches exploit the coordinate system
\begin{equation}
u(x,y,t) = \Phi\big( x + ct + \theta(y,t) \big) + v(x, y, t)
\end{equation}
in the neighbourhood of the planar travelling wave
and require the initial perturbations
$\theta(y,0)$ and $v(x,y,0)$ to be localized in a suitable sense.

The first approach was pioneered by Kapitula in \cite{KAP1997},
where he used semigroup methods %pectral methods, Green's functions and
and fixed-point arguments to
%analyze the coupled
%sysstem for $\theta$couple a nonlinear heat equation
%for $\theta$ with an  show
show that $\theta$ tends algebraically to zero, while
$v$ decays exponentially fast.
The advantage of this approach is that only weak
spectral assumptions need to be imposed
on the underlying system. However,
the crude estimates on the nonlinear
terms lead to rather weak estimates for the basin of attraction.

The second approach leverages the comparison principle
to obtain stability for a much larger class of initial perturbations.
By slowing down the natural decay-rate of the fundamental solution
of the heat equation, the authors of the landmark paper
\cite{BHM} were able to construct explicit super and sub-solutions
to \sref{eq:int:nag:pde} that trap perturbations
that can be arbitrarily large (but localized). In fact,
the authors use their construction to show that these planar waves
can pass around large compact obstacles and still eventually recover
their shape.

In \cite{HJHOBST2D,HJHSTB2D} these approaches were generalized
to the discrete setting of \sref{eq:int:nag:lde},
thereby continuing the early work by Bates and Chen \cite{BatesChen2002}
featuring %who studied
a related four-dimensional non-local problem.
In both cases the key technical challenge was the analysis of
troublesome non-selfadjoint terms
spawned by the anisotropy of the lattice, especially in
situation where the group velocity
does not vanish.
These terms have slower decay rates than their PDE counterparts
and hence require special care to close the nonlinear bootstrapping
procedure.
For example, the sub-solutions in \cite{BHM} consist of only two terms,
while 33 terms were required in \cite{HJHOBST2D} to correct for the slower
decay.

\paragraph*{Spreading phenomena}
The classic result \cite[Thm. 5.3]{AW78}
obtained by Weinberger for the PDE
\sref{eq:int:nag:pde}
states that large compact blobs with $u \approx 1$
inside and $u \approx -1$ outside can expand throughout the plane.
The proof of this result relies on the construction of radially expanding
sub- and super-solutions
by glueing together planar travelling waves.

In \cite{HJHOBST2D} a weak
version of this expansion result
was established for the LDE \sref{eq:int:nag:lde}
in the special case that no direction is pinned.
However, the underlying sub- and super-solutions expand at the
speeds $\min_{0 \le \zeta \le 2\pi} c_{\rho, \zeta}$
and $\max_{0 \le \zeta \le 2\pi} c_{\rho,\zeta}$ respectively,
which still leaves a considerable
hole in our knowledge of the expansion process.
Indeed,
the numerical results in \cite{BachBrenda}
provide strong evidence that the limiting
shape of the expanding blob can be found
by applying the Wulff construction \cite{roosen1998wulffman}
to the polar plot of the $\zeta \mapsto c_{\rho,\zeta}$
relation. For a large subset of parameters $\rho$
this limiting shape resembles
a polygon.

The main motivation
behind the current paper is to take a step towards
understanding this expansion process by looking
at the evolution of a single corner. Indeed,
when the expanding blob is sufficiently large,
it would seem to be very reasonable to assume that the corners
of the polygon behave in an almost independent fashion.

\paragraph{Corners for PDEs}
Assuming for concreteness that $\rho <0$,
the horizontal planar wave $(c,\Phi)$ given by
\sref{eq:int:trv:wave:ansatz} with $\phi = 0$ satisfies
$c > 0$, which means that it travels towards the left.
%The results
In \cite{HARAG2006} Haragus and Scheel construct
travelling corner solutions to \sref{eq:int:nag:pde}
by `bending' this planar wave to the left
in the spatial limits $y \to \pm \infty$,
so that the interface resembles a $>$ sign.

In particular, for any small opening angle
$\varphi > 0$, the authors establish the existence and stability
of
solutions of the form
\begin{equation}
\label{eq:int:nag:pde:corner}
u(x,y,t) = \Phi( x + \theta(y) + \frac{c}{\cos\varphi} t)
+ v( x + \frac{c}{\cos \varphi} t , y).
\end{equation}
Here $\norm{v(\cdot, y)}_{H^2} = O( \varphi^2)$
uniformly in $y$, while the phase $\theta$ satisfies the limits
\begin{equation}
\label{eq:int:corner:lims:for:y}
\lim_{y \to \pm \infty} \theta'(y) = \pm \tan \varphi.
\end{equation}
Notice that the horizontal speed $\frac{c}{\cos \varphi}$
of these corners is faster than the original speed
of the planar wave.

The result is obtained by
using the change of variable $x' = x + \tilde{c} t$
to recast
\sref{eq:int:nag:pde} as %in the form
\begin{equation}
u_t = u_{xx} - \tilde{c} u_x
+ u_{yy} + g_{\mathrm{cub}}( u; \rho)
\end{equation}
and subsequently demanding $u_t = 0$.
The resulting system can be written in the
first-order form
\begin{equation}
\label{eq:int:pde:trnsf}
\begin{array}{lcl}
u_y  & = & v ,
\\[0.2cm]
v_y & = & \tilde{c} u_x - u_{xx} - g_{\mathrm{cub}}(u; \rho) ,
\end{array}
\end{equation}
which admits the family of
$y$-independent equilibria
\begin{equation}
\label{eq:int:fam:of:eq:pde}
\tilde{c} = c,
\qquad
\big( u, v\big)(x,y) =
\big( \Phi( x + \vartheta), 0 \big),
\qquad
\qquad
\vartheta \in \Real.
\end{equation}
The linearization of \sref{eq:int:pde:trnsf}
around $(\tilde{c} , u ) = (c, \Phi)$
can be written as
\begin{equation}
\label{eq:int:lin:sys:corner:lin}
\begin{array}{lcl}
p_y  & = & q ,
\\[0.2cm]
q_y & = & \tilde{c} p_x - p_{xx}
- g_{\mathrm{cub}}'(\Phi(\cdot); \rho) p .
\end{array}
\end{equation}
This system admits the $y$-independent
solutions $(\Phi', 0)$ caused by the
translational invariance, together with
the linearly growing solution $(y \Phi', \Phi')$.
In particular, the
desired corner \sref{eq:int:nag:pde:corner}
lives on the two-dimensional global center manifold
associated to the family \sref{eq:int:fam:of:eq:pde}.
The solutions on this manifold can be
represented in the form
\begin{equation}
\begin{array}{lcl}
(u,v) = \Big(\Phi\big(  \theta(y) + x \big), 0\Big)
+ \kappa(y) \Big(0 , \Phi'\big(\theta(y) + x \big) \Big)
+  h^*\big( \kappa(y) , \tilde{c} \big)\big( \theta(y) + x \big)
\end{array}
\end{equation}
for some function
\begin{equation}
h^*: (- \epsilon, \epsilon) \times (c - \epsilon, c+\epsilon) 
\to H^2 \times H^1.
\end{equation}
One can subsequently obtain two skew-coupled
ODEs
to describe the dynamics of the scalar functions
$\kappa$ and $\theta$.
A relatively straightforward analysis
shows that these ODEs have solutions
for which $\theta$ satisfies the limits
\sref{eq:int:corner:lims:for:y},
while $\kappa$ remains small.
This suffices to  establish
the existence of the corners
\sref{eq:int:nag:pde:corner}.

In addition, in \cite{HARAG2006ANISO}
%introduced
anisotropic effects
were introduced
into the problem by allowing
the nonlinearity $g$ to depend on the gradient of $u$
and considering non-diagonal
diffusion coefficients.
In such cases the group velocity $c_g$
defined by the quantities
\sref{eq:int:def:grp:velocity} need not vanish,
but it can be removed by
applying a coordinate transformation
$y' = y - c_g t$ in the transverse direction.

By restricting their attention to small
opening angles $\varphi$
and using center manifold arguments,
Scheel and Haragus were also
able to apply their techniques
to multi-component reaction-diffusion PDEs
such as the FitzHugh-Nagumo and Gray-Scott equations
\cite{HARAG2006ANISO,HARAG2006}
However, it is also possible to consider
large opening angles when considering
equations that admit a comparison principle.
Indeed, in
\cite{BONNETHAMEL1999} explicit
sub- and super-solutions are
used to construct corners for the Nagumo PDE
\sref{eq:int:nag:pde} that can be arbitrarily sharp.

\paragraph{Corners for LDEs}
The crucial point
in the analysis outlined above
for the corners \sref{eq:int:nag:pde:corner}
is that the phase shift $\theta(y)$
can be completely factored out from the system.
This implies that the ODE for $\kappa$
does not depend on $\theta$.
In addition, it allows the center manifold to be constructed
by a standard fixed point argument
analogous to the local case.

This is possible because the
right-hand side of \sref{eq:int:lin:sys:corner:lin}
maps $H^2 \times H^1$ into $H^1 \times L^2$,
which roughly means that its inverse gains an order
of regularity in both components. This precisely
compensates for the loss of regularity
that arises by factoring out the phase-shift.

However, when attempting to mimic
this procedure for the LDE \sref{eq:int:nag:lde}
one runs into a fundamental difficulty.
Indeed, the analog
of \sref{eq:int:lin:sys:corner:lin}
has a right-hand side that
now maps $H^1 \times H^1$ into $H^1 \times L^2$
due to the lack of second derivatives in
the equation.
This forces us to construct a full two-dimensional
global center manifold that takes into the account
the dynamics of $\theta$ and $v$ simultaneously.

A similar situation was encountered
by one of the authors in \cite{HJHBRGSG}, where
modulated travelling
wave solutions were constructed to
a class of non-local systems.
However, the analog of
\sref{eq:int:lin:sys:corner:lin}
is a difference equation rather than a
differential equation. The order
of this difference equation can become
arbitrarily large depending on the
height of the fraction $\tan \zeta$,
which we require to be rational.
Nevertheless, the final step
in our analysis requires us to uncover
%depends on recovering
a first-order difference equation
for the center variables.

The main technical contribution in this paper
is that we adapt the %main
spirit of the approach in \cite{HJHBRGSG}
to construct global center manifolds
for the differential-difference systems
that we encounter here. This approach
uses two intertwined fixed point procedures
to separate the flow problem for the
two center variables from the task
of capturing the shape of the
remainder function $h^*$.
%In order to understand
The underlying linear
problems have non-autonomous
slowly-varying coefficients,
for which we develop appropriate
solution operators.

In this paper we do require
the group velocity \sref{eq:int:def:grp:velocity}
to vanish. Unlike in the spatially continuous setting,
this cannot always be arranged by a simple
variable transformation. Indeed, such a
transformation would force the spatial
variable transverse to the propagation direction
to become continuous, destroying
the difference structure of the system.
This prevents us from exploiting
the $\omega \mapsto \omega +  2\pi$ periodicity
in the Fourier variable.
As a result, resonances start to appear in the spectrum
that are very hard to control. A similar
situation was encountered in \cite{HJHBRGSG},
which forced the authors to add a smoothening term
to the underlying system.

We emphasize that the group velocity
for \sref{eq:int:nag:lde}
vanishes automatically in the directions
$\zeta =0$ and $\zeta = \frac{\pi}{4}$.
In addition, directions where the wavespeed
is minimal (and hence the group velocity is zero)
play an important role in the Wulff construction,
which is the primary motivation for our analysis here.
In any case, the delicate behaviour of the $c_{\rho,\zeta}$
map for the Nagumo LDE \sref{eq:int:nag:lde}
leads to a much richer class of behaviour
than that displayed by its continuous
counterpart \sref{eq:int:nag:pde}.
For example, the latter only features interior corners,
while the former can also admit exterior corners.
The former also allows for so-called
bichromatic corners, which connect
spatially homogeneous equilibria
to checkerboard patterns.

While we are confident that our center manifold construction
will also allow us to establish the (linearized) stability of the corners
constructed here along the lines of the approach in \cite{HARAG2006ANISO},
we do not pursue this in the present paper. The main reason is that
there is no coordinate transformation that can freeze our corners
and also leave the discrete structure of the equation intact.
One would need to generalize the approach developed in \cite{BGV2003,HJHNLS, HJHSTB2D}
to accommodate solutions that vary in two directions instead of just one,
which we expect to be a tedious task.

\paragraph{Organization}
%This paper is organized as follows.
Our main results are formulated in {\S}\ref{sec:mr} and applied to the Nagumo LDE
\sref{eq:int:nag:lde} in {\S}\ref{sec:mr:ex:i}-\ref{sec:mr:ex:ii}.
In {\S}\ref{sec:setup} we derive the differential-difference system that the pair $(\theta, v)$
must satisfy and formulate the global center manifold result. We proceed in {\S}\ref{sec:prlm}
by deriving a representation formula for solutions to the linearized problem with constant phase.
This requires us to compute a convoluted spectral projection operator that arises from the second order pole
that the operator $\mathcal{L}_z^{-1}$ has in $z = 0$. In {\S}\ref{sec:linop:cc}-{\S}\ref{sec:linop:sv}
we combine this representation formula with Fourier analysis to construct a solution operator
for the linearized problem where the phase is allowed to vary slowly. Finally,
we setup the fixed point problems required to build the global center manifold in \S\ref{sec:cm},
appealing at times to the results in \cite{HJHBRGSG} for overlapping parts of the program.

\paragraph{Acknowledgements}
HJH acknowledges support from the Netherlands Organization for Scientific Research (NWO)
(grant 639.032.612). LM acknowledges support from the Netherlands Organization for
Scientific Research (NWO) (grant 613.001.304).

\section{Main Results}
\label{sec:mr}

In this paper we consider the nearest-neighbour lattice
differential equation
\begin{equation}
\label{eq:mr:mainLDE}
\dot{u}_{ij}(t) = f\big(  u_{i+1, j}(t), u_{i, j+1}(t), u_{i-1, j}(t), u_{i, j-1}(t) , u_{ij}(t) \big),
\end{equation}
posed on the planar lattice $(i,j) \in \Wholes^2$,
in which $u$ takes values in $\Real^d$.
For convenience, we introduce the operator $\pi^+_{ij}: \ell^\infty(\Wholes^2; \Real^d) \to (\Real^d)^5$
that acts as
\begin{equation}
\pi^+_{ij} u =
\big( u_{i+1, j}, u_{i, j+1}, u_{i-1, j}, u_{i, j-1} , u_{ij} \big) \in (\Real^d)^5
\end{equation}
for any $(i,j) \in \Wholes^2$,
which allows us to rewrite \sref{eq:mr:mainLDE}
in the condensed form
\begin{equation}
\label{eq:mr:pimainLDE}
\dot{u}_{ij}(t) = f\big( \pi^+_{ij} u(t) \big).
\end{equation}
The plus sign corresponds with the fact that a ''+''-shaped stencil
is used to sample $u$.

The conditions we impose on the nonlinearity $f$ are summarized
in the following assumption.
\begin{itemize}
\item[(Hf)]{
  The nonlinearity $f: (\Real^d)^5 \to \Real^d$ is $C^r$-smooth for some $r \ge 2$ and there exist two points $u_\pm \in \Real^d$ with
  \begin{equation}
    f( u_\pm, u_\pm, u_\pm, u_\pm, u_\pm) = 0.
  \end{equation}
}
\end{itemize}
We emphasize that the two points $u_\pm$ are allowed to be equal.
These two equilibria are required to be connected by a planar travelling wave solution
to \sref{eq:mr:mainLDE}. In particular, we pick an arbitrary rational direction
$(\sigma_A, \sigma_B) \in \Wholes^2$ with $\mathrm{gcd} ( \sigma_A, \sigma_B) = 1$
and impose the following condition.
\begin{itemize}
\item[(H$\Phi$)]{
  There exists a wave speed $c_* \neq 0$ and a wave profile $\Phi_* \in C^{r+1}(\Real, \Real^d)$
  so that the function
  \begin{equation}
    \label{eq:mr:TWansatz}
    u_{ij}(t) = \Phi_*\big( i \sigma_A + j \sigma_B + c_* t \big)
  \end{equation}
  satisfies \sref{eq:mr:mainLDE} for all $t \in \Real$.
  In addition, we have the limits
  \begin{equation}
    \label{eq:mr:wavelim}
     \lim_{\xi \to \pm \infty} \Phi_*(\xi) = u_\pm .
  \end{equation}
}
\end{itemize}
Upon introducing the operator $\tau: C(\Real; \Real^d) \to C(\Real; (\Real^d)^5)$
that acts as
\begin{equation}
\label{eq:mr:defTauWithoutOmega}
[\tau p](\xi) = \Big(p(\xi+\sigma_A),p(\xi+\sigma_B),p(\xi-\sigma_A),p(\xi-\sigma_B),p(\xi) \Big) \in (\Real^d)^5,
\end{equation}
we note that the pair $(c_*, \Phi_*)$ must satisfy the
functional differential equation of mixed type (MFDE)
\begin{equation}
\label{eq:mr:trv:wave:mfde}
\begin{array}{lcl}
c_* \Phi_*'(\xi)
 &= &
f \big( \Phi_*(\xi + \sigma_A), \Phi_*(\xi + \sigma_B), \Phi_*(\xi - \sigma_A),
\Phi_*(\xi - \sigma_B), \Phi_*(\xi) \big)
\\[0.2cm]
& = &  f\big( [\tau\Phi_*](\xi) \big) .
\end{array}
\end{equation}
In particular, the $C^{r+1}$-continuity mentioned in (H$\Phi$) is automatic upon assuming that $\Phi_*$ is merely continuous.

For convenience, we now introduce the new coordinates
\begin{equation}
\begin{array}{lcl}
n & = & \sigma_A i + \sigma_B j ,
\\[0.2cm]
l & = &
  \sigma_A j - \sigma_B i ,
\end{array}
\end{equation}
which are parallel respectively orthogonal to the direction of motion of the wave \sref{eq:mr:TWansatz}.
Upon introducing the notation
\begin{equation}
\pi^\times_{nl} u
 = \big(
        u_{n + \sigma_A , l - \sigma_B } , u_{n + \sigma_B, l + \sigma_A},
        u_{n - \sigma_A , l + \sigma_B } , u_{n - \sigma_B, l - \sigma_A },
        u_{nl}
   \big),
\end{equation}
the LDE \sref{eq:mr:pimainLDE} transforms into
the equivalent problem
\begin{equation}
\label{eq:mr:lde:cross}
\dot{u}_{nl}(t) = f\big( \pi^\times_{nl} u(t) \big) ,
\end{equation}
which admits the travelling wave solution
\begin{equation}
\label{eq:mr:trv:wave:cross}
u_{nl}(t) = \Phi_*( n + c_* t) .
\end{equation}

A standard approach towards establishing the stability of the wave
\sref{eq:mr:trv:wave:cross}
under the nonlinear dynamics of the LDE \sref{eq:mr:lde:cross} is to consider the linear variational problem
\begin{equation}
\dot{v}_{nl}(t) = Df\big( [\tau \Phi_*](n + c_* t) \big) \pi^\times_{nl} v(t).
\end{equation}
Looking for a solution of the form
\begin{equation}
\label{eq:mr:ans:pertb}
v_{nl}(t) = e^{\lambda t} e^{z l } p(n + c_* t),
\end{equation}
we readily find that $p$ must satisfy the eigenvalue problem
\begin{equation}
\mathcal{L}_z p =       \lambda p  .
\end{equation}
Here the linear operator $\mathcal{L}_z: W^{1, \infty}(\Real; \Complex^d) \to L^\infty(\Real; \Complex^d)$
acts as
\begin{equation}
\label{eq:mr:def:lz}
[\mathcal{L}_{z} p](\xi) = - c_* p'(\xi) + \sum_{j=1}^5 A_{z,j}(\xi) p(\xi + r_j),
\end{equation}
with shifts $r_j$ and functions $A_{z,j}$ that are given by
\begin{equation}
\begin{array}{lclclcl}
r_1 & = & \sigma_A, & &
  A_{z,1}(\xi) & = & e^{ -\sigma_B z} D_1 f\big( [\tau \Phi_*](\xi) \big) , \\[0.2cm]
r_2 & = & \sigma_B, & &
  A_{z,2}(\xi) & = & e^{ +  \sigma_A z} D_2 f\big( [\tau \Phi_*](\xi) \big), \\[0.2cm]
r_3 & = & -\sigma_A, & &
  A_{z,3}(\xi) & = & e^{ +  \sigma_B z} D_3 f\big( [\tau \Phi_*](\xi) \big), \\[0.2cm]
r_4 & = & -\sigma_B, & &
  A_{z,4}(\xi) & = & e^{ -  \sigma_A z} D_4 f\big( [\tau \Phi_*](\xi) \big), \\[0.2cm]
r_5 & = & 0, & &
  A_{z,5}(\xi) & = &  D_5 f\big( [\tau \Phi_*](\xi) \big).
\end{array}
\end{equation}

Since $\Phi_*(\xi)$ approaches $u_\pm$ as $\xi \to \pm \infty$,
it is possible to define the characteristic $\Complex^{d \times d}$-valued functions
\begin{equation}
\Delta^\pm_{z}(s) = - c_*  s I + \sum_{j=1}^5 [\lim_{\xi \to \pm \infty} A_{z, j}(\xi)] e^{ s r_j} .
\end{equation}
Our first spectral assumption states that
these characteristic functions %- \lambda
cannot have roots on the imaginary axis
whenever $z$ is purely imaginary.  % if $\Re \lambda \ge 0$.
\begin{itemize}
\item[$\mathrm{(HS1)}$]{
  For all $\omega \in [-\pi, \pi]$ and $\nu \in \Real$
  %and $\lambda \in \Complex$ that have $\Re \lambda \ge 0$,
  we have
  \begin{equation}
    \label{eq:mr:hs1:charEq}
    \det \big[ \Delta^\pm_{i \omega}(i \nu) \big]  %- \lambda I]
        \neq 0 .
  \end{equation}
  %for all $\nu \in \Real$.
  %In addition, for $\omega \neq 0$
  %the operator $\mathcal{L}_{i\omega} $
  %is invertible as a map from $W^{1, \infty}(\Real, \Complex^d)$
  % into $L^\infty(\Real, \Complex^d)$.
}
\end{itemize}

We note that \sref{eq:mr:hs1:charEq}
can be used to rule out kernel elements of $\mathcal{L}_{i \omega}$ that behave
as $e^{ i \nu \xi}$ as $\xi \to \pm \infty$.
In fact, using \cite[Thm. A]{MPA} we see that (HS1)
implies that $\mathcal{L}_{i\omega}$ is a Fredholm operator
for all $\omega \in [-\pi, \pi]$.
Our next condition demands that these operators
are actually invertible for $\omega \neq 0$.

\begin{itemize}
\item[$\mathrm{(HS2)}$]{
  For any $\omega \neq 0$
  the operator $\mathcal{L}_{i\omega} $
  is invertible as a map from $W^{1, \infty}(\Real, \Complex^d)$
  into $L^\infty(\Real, \Complex^d)$.
}
\end{itemize}

Since the Fredholm index varies continuously,
(HS1) and (HS2) together imply that the Fredholm index
of $\mathcal{L}_0$ is zero. The translational invariance of the problem
implies that $\mathcal{L}_0 \Phi_*' = 0$, which means that zero is an eigenvalue
for $\mathcal{L}_0$. Our next assumption
states that this eigenvalue is in fact algebraically simple.

\begin{itemize}
\item[$\mathrm{(HS3)}$]{
  We have the characterization
  \begin{equation}
    \mathrm{Ker}( \mathcal{L}_0 ) = \mathrm{span} \{ \Phi_*' \}
  \end{equation}
  and the algebraic simplicity condition
  \begin{equation}
    \label{eq:mr:not:in:range}
    \Phi_*' \notin \mathrm{Range} (\mathcal{L}_0 ) .
  \end{equation}
}
\end{itemize}

For any $z \in \mathbb{C}$ we now introduce the linear operator
\begin{equation}
\mathcal{L}^{\mathrm{adj}}_z: W^{1, \infty}(\Real, \Complex^d) \to L^{\infty}(\Real, \Complex^d)
% \qquad \omega \in [-\pi, \pi],
\end{equation}
that acts as
\begin{equation}
[\mathcal{L}^{\mathrm{adj}}_z q](\xi) =  c_* q'(\xi) + \sum_{j=1}^5
%A^\dagger_{\omega,j}(\xi) q(\xi - r_j).
A^*_{z,j}(\xi - r_j) q(\xi - r_j).
\end{equation}
An easy computation shows that
\begin{equation}
\int_{-\infty}^{\infty} \langle q(\xi),  [\mathcal{L}_{z} p ](\xi) \rangle_{\Complex^d} \, d \xi
= \int_{-\infty}^{\infty} \langle [\mathcal{L}^{\mathrm{adj}}_z q](\xi), p(\xi) \rangle_{\Complex^d} \, d \xi
\end{equation}
holds for all pairs $p,q \in W^{1, \infty}(\Real, \Complex^d)$.
For these reason, we refer to this operator as the formal adjoint
of $\mathcal{L}_z$.

Using \cite[Thm. A]{MPA} together with $\mathrm{(HS3)}$,
one sees that the kernel of $\mathcal{L}^{\mathrm{adj}}_0$ must also be one-dimensional.
In particular, it is spanned by a function $\psi_* \in W^{1, \infty}(\Real,\Real^d)$
that can be uniquely fixed by the identity
\begin{equation}
\label{eq:mr:normalizatonConditionOnPsi}
\int_{-\infty}^{\infty} \langle \psi_*(\xi),  \Phi_*'(\xi) \rangle \, d \xi = 1
\end{equation}
on account of \sref{eq:mr:not:in:range}.
We note that (HS1) implies that both $\Phi_*'(\xi)$ and $\psi_*(\xi)$ decay exponentially
as $\xi \to \pm \infty$.

We now explore two important consequences of the algebraic simplicity condition (HS3).
The first of these states that the zero eigenvalue can be extended
to a branch of eigenvalues $\lambda_z$ for $\mathcal{L}_z$ when $\abs{z}$ is small.

\begin{lem}[{see \cite[Prop. 2.2]{HJHSTB2D}}]
\label{lem:mr:def:branch:lambda:z}
Assume that (Hf), $(H\Phi)$ and (HS1)-(HS3) are all satisfied.
Then there exists a constant $\delta_z > 0$ together
with pairs
\begin{equation}
  (\lambda_z, \phi_z ) \in \Complex \times W^{1, \infty}(\Real, \Complex^d) ,
\end{equation}
defined for each
$z \in \mathbb{C}$ with $\abs{z}< \delta_z$,
such that the following hold true.
\begin{itemize}
\item[(i)]{
  The characterization
  \begin{equation}
    \mathrm{Ker} ( \mathcal{L}_z - \lambda_z) = \mathrm{span} \{ \phi_z \}
  \end{equation}
  together with the algebraic simplicity condition
  \begin{equation}
    \phi_z \notin \mathrm{Range}( \mathcal{L}_z - \lambda_z)
  \end{equation}
  hold for each $z \in \mathbb{C}$ with $\abs{z} < \delta_z$.
}
\item[(ii)]{
  We have $\lambda_0 = 0$, $\phi_0 = \Phi_*'$ and
  the maps $z \mapsto \lambda_z$ and $z \mapsto \phi_z$ are analytic.
}
\item[(iii)]{
 The normalization condition
 \begin{equation}
   \label{eq:mr:norm:cnd:on:phi:z}
   \langle \psi_*, \phi_z \rangle_{L^2} = 1
 \end{equation}
 holds for every $z \in \mathbb{C}$ with $\abs{z} < \delta_z$ .
}
\end{itemize}
\end{lem}

The second consequence is that wave $(c_*,\Phi_*)$
travelling in the rational direction $(\sigma_A, \sigma_B)$ can be perturbed
to yield waves travelling in nearby directions. In particular,
we introduce the constants $(\zeta_*, \sigma_*)$ by writing
\begin{equation}
\sigma_* = \sqrt{\sigma_A^2 + \sigma_B^2},
\qquad (\sigma_A, \sigma_B) = \sigma_*( \cos \zeta_*, \sin \zeta_*) .
\end{equation}
Looking for solutions
to the LDE \sref{eq:mr:lde:cross}
of the form
\begin{equation}
\label{eq:mr:ansatz:u:nl:varphi}
u_{nl}(t)
 = \Phi_{\varphi}( n \cos \varphi + l \sin \varphi + c_\varphi t ),
\end{equation}
a short computation
shows that the pair $(c_\varphi, \Phi_\varphi)$
must satisfy the MFDE
\begin{equation}
\label{eq:mr:trv:wave:diff:angle}
c_\varphi \Phi_{\varphi}'(\xi) = f\big( \tau_\varphi \Phi_\varphi \big)(\xi),
\end{equation}
in which we have introduced the notation
\begin{equation}
\begin{array}{lcl}
[\tau_{\varphi} p](\xi)
& = &
\Big(
  p\big(\xi + \sigma_* \cos(\varphi+\zeta_*)\big)
    ,
  p\big(\xi + \sigma_* \sin(\varphi + \zeta_*)\big) ,
\\[0.2cm]
& & \qquad
  p\big(\xi - \sigma_* \cos(\varphi+ \zeta_*) \big),
  p\big(\xi - \sigma_* \sin(\varphi + \zeta_*) \big),
  p\big(\xi)
\Big) .
\end{array}
\end{equation}
In order to translate these waves back to the original coordinates,
we remark that any solution to \sref{eq:mr:trv:wave:diff:angle}
yields a solution to the original LDE \sref{eq:mr:mainLDE}
by writing
\begin{equation}
u_{ij}(t) = \tilde{\Phi}_{\tilde{\zeta}}\big(i \cos \tilde{\zeta} + j \sin \tilde{\zeta}
 + \tilde{c}_{\tilde{\zeta}} t \big)
\end{equation}
with the rescaled quantities
\begin{equation}
\label{eq:mr:rescalings}
\tilde{\zeta} = \zeta_* + \varphi,
\qquad
\tilde{c}_{\tilde{\zeta}} = \sigma_*^{-1} c_{\varphi},
\qquad
\tilde{\Phi}_{\tilde{\zeta}}(\xi) = \Phi_{\varphi}(\sigma_* \xi) .
\end{equation}

\begin{lem}[{see {\S}\ref{sec:prlm}}]
\label{lem:mr:angl:dep}
Assume that (Hf), $(H\Phi)$ and (HS1)-(HS3) are all satisfied.
Then there exists a constant $\delta_\varphi > 0$ together
with pairs
\begin{equation}
  (c_{\varphi}, \Phi_{\varphi} ) \in \Real \times W^{1, \infty}(\Real, \Complex^d) ,
\end{equation}
defined for each $\varphi \in (-\delta_\varphi, \delta_\varphi)$,
such that the following hold true.
\begin{itemize}
\item[(i)]{
  For every $\varphi \in (-\delta_\varphi, \delta_\varphi)$,
  the pair $(c_{\varphi} , \Phi_{\varphi})$
  satisfies the MFDE \sref{eq:mr:trv:wave:diff:angle},
  while the function
  \begin{equation}
    u_{nl}(t) = \Phi_{\varphi}\big(n \cos \varphi + l \sin \varphi + c_{\varphi} t \big)
  \end{equation}
  satisfies the LDE \sref{eq:mr:lde:cross} for all $t \in \Real$.
}
\item[(ii)]{
  The maps $\varphi \mapsto c_{\varphi}$ and $\varphi \mapsto \Phi_{\varphi}$
  are $C^{r-1}$-smooth. In addition, we have
  $c_0 = c_*$, together with $\Phi_0 = \Phi_*$.
}
\item[(iii)]{
 The normalization condition
 \begin{equation}
      \langle \psi_*, \Phi_{\varphi} \rangle = \langle \psi_*, \Phi_* \rangle
 \end{equation}
 holds for every $\varphi \in (-\delta_\varphi, \delta_\varphi)$.
}
\item[(iv)]{
  We have the identities
  \begin{equation}
  \label{eq:mr:grp:vel:vs:ang}
  [\partial_\varphi c_{\varphi}]_{\varphi = 0} = [\partial_z \lambda_z]_{z =0},
  \qquad
  [\partial_\varphi \Phi_{\varphi}]_{\varphi = 0} = [\partial_z \phi_z]_{z =0} .
  \end{equation}
}
\end{itemize}
\end{lem}

We remark that the first quantities in \sref{eq:mr:grp:vel:vs:ang} can be interpreted as a
so-called \textit{group velocity}, which %can be seen
represents
%as
the speed at which long-amplitude perturbations travel in the transverse direction.
Indeed, expanding \sref{eq:mr:ans:pertb} with $p = \phi_z$ and $\lambda = \lambda_z$
we find
\begin{equation}
v_{nl}(t) = \mathrm{exp}\big[ z( l + [\partial_z \lambda_z]_{z =0} t) + O (z^2 t) \big] \phi_z(n + c_* t).
\end{equation}
Our final condition requires $\lambda_z$ to depend quadratically on $z$,
which means that this group velocity has to vanish.
We emphasize that the inequality $[\partial^2_z \lambda_z]_{z = 0}  > 0$ was required
in \cite{HJHSTB2D} to obtain the nonlinear stability of the planar wave $(c_*, \Phi_*)$.

\begin{itemize}
\item[$\mathrm{(HM)}$]{
  We have the identities
  \begin{equation}
    \label{eq:mr:grp:vel:zero}
    [\partial_\varphi c_{\varphi}]_{\varphi = 0} = [\partial_z \lambda_z]_{z =0} = 0,
  \end{equation}
  together with the inequality
  \begin{equation}
    [\partial^2_z \lambda_z]_{z = 0} \neq 0 .
  \end{equation}
}
\end{itemize}

As a final preparation, we introduce the directional dispersion
\begin{equation}
d_\varphi = \frac{c_\varphi}{\cos\varphi }.
\end{equation}
Assuming that the original wave travels in the horizontal direction $\zeta_* = 0$,
the quantity $d_\varphi$ represents the speed at which level-sets
of the wave $(c_{\varphi}, \Phi_{\varphi})$ travel along the horizontal axis;
see Figure \ref{fig:mr:corn:int:ext}.
An easy calculation using \sref{eq:mr:grp:vel:zero}
shows that
\begin{equation}
d_0 = c_*, \qquad [\partial_\varphi d_\varphi]_{\varphi =0} = 0,
\qquad
[\partial_\varphi^2 d_\varphi]_{\varphi =0}
= [\partial_\varphi^2 c_\varphi]_{\varphi=0} + c_* .
\end{equation}
Our main result establishes the existence of travelling corners
in the setting where $[\partial_\varphi^2 d_\varphi]_{\varphi =0}  \neq 0$.
Assuming again that $\phi_* =0$ and that $[\partial_z^2 \lambda_{z}]_{z=0} $
and $[\partial_\varphi^2 d_\varphi]_{\varphi =0}$ are both strictly positive,
the level-sets resemble a $>$ sign.
In particular, when $c_* >0$ this resembles an interior corner
travelling to the left.

\begin{figure}[t]
\centering
\includegraphics[width=\textwidth]{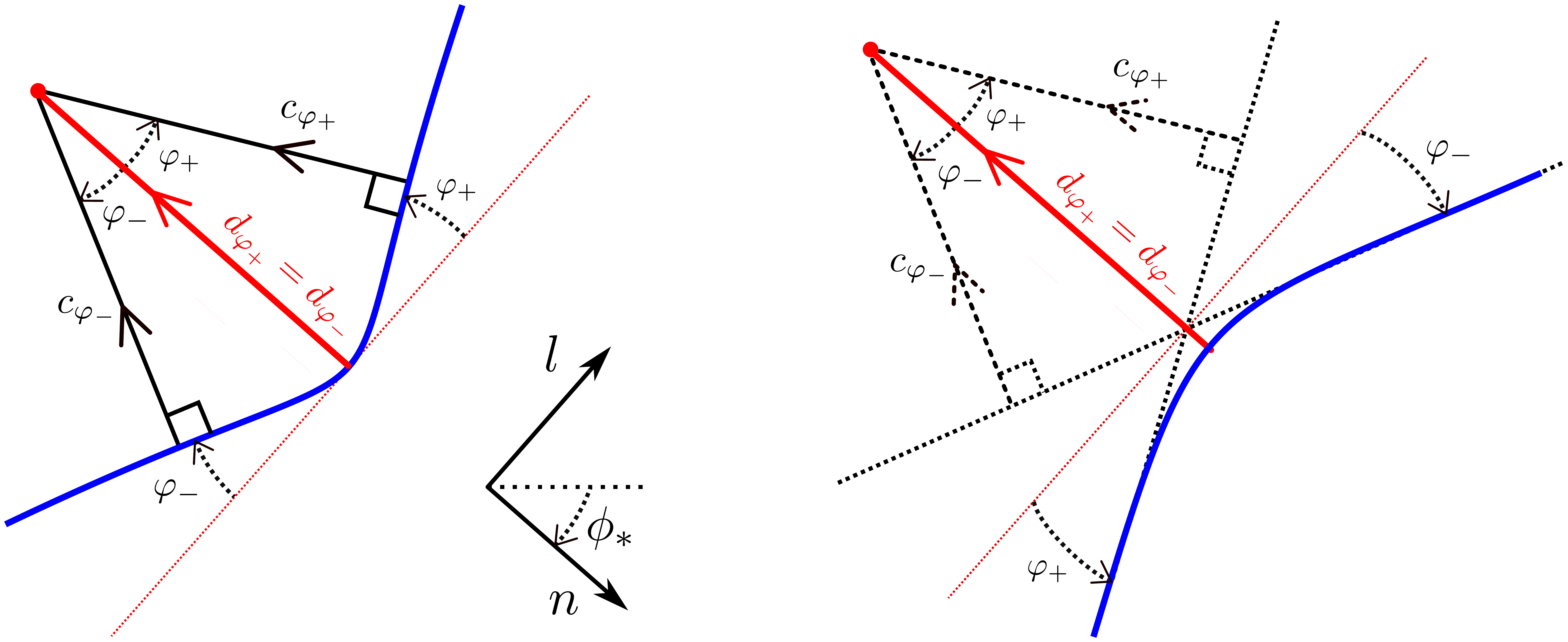}
\caption{The blue curves in the left and right panels
depict the interface of an interior respectively
exterior corner. Both corners travel
at the speed $d_{\varphi_-} = d_{\varphi_+}$
and share the
coordinate system $(n,l)$ depicted in the center.
Angles are positive
when oriented counter-clockwise and negative otherwise.
All speeds are positive.}
\label{fig:mr:corn:int:ext}
\end{figure}

\begin{thm}[{see {\S}\ref{sec:setup}}]
\label{thm:mr:ex:corner}
Assume that (Hf), $(H\Phi)$, (HS1)-(HS3) and (HM) are all satisfied.
Assume furthermore that $[\partial_{\varphi}^2 d_\varphi]_{\varphi=0} \neq 0$
and pick a sufficiently large $K > 0$.
Then for any $c \neq c_*$ sufficiently close to $c_*$ with
\begin{equation}
\mathrm{sign} \big( c - c_*  \big) =
 \mathrm{sign} \Big( [\partial_{\varphi}^2 d_\varphi]_{\varphi=0} \Big),
\end{equation}
there exist two sequences
\begin{equation}
(\theta, v):  \Wholes \to \Real \times H^1(\Real; \Real^d)
\end{equation}
together with two angles $\varphi_- < 0 < \varphi_+$
that satisfy the following properties.
\begin{itemize}
\item[(i)]{
  The function
  \begin{equation}
    u_{nl}(t) = \Phi_*(n + ct + \theta_l) + v_l(n + ct)
  \end{equation}
  satisfies the LDE \sref{eq:mr:lde:cross} for all $t \in \Real$.
}
\item[(ii)]{
  We have $\langle \psi_*( \cdot + \theta_l) , v_l \rangle_{L^2} = 0$
  for all $l \in \Wholes$ .
}
\item[(iii)]{
  We have the identities
  \begin{equation}
    d(\varphi_\pm) = c.
  \end{equation}
}
\item[(iv)]{
  If  $[\partial_\varphi^2 d_\varphi]_{\varphi =0}$
  and $[\partial_z^2 \lambda_z]_{z=0}$ have the same sign,
  %  $d''(0) \lambda''(0) > 0$
  then we have the limits
  \begin{equation}
  \theta_{l+1}-\theta_l \to \tan( \varphi_\pm) ,
  \qquad
  \qquad
  l \to \pm \infty.
  \end{equation}
  %\begin{equation}
%\begin{split}
%\theta_{l+1}-\theta_l &\rightarrow \tan( \varphi_-)  <0 \quad l\rightarrow -\infty \\
%\theta_{l+1}-\theta_l &\rightarrow \tan(\varphi_+)  >0 \quad l\rightarrow +\infty .
%\end{split}
%\end{equation}
  On the other hand, if these quantities have opposing signs,
  then we have the limits
  \begin{equation}
  \theta_{l+1}-\theta_l \to \tan( \varphi_\mp) ,
  \qquad
  \qquad
  l \to \pm \infty.
  \end{equation}
  %
  %  \begin{equation}
%\begin{split}
%\theta_{l+1}-\theta_l &\rightarrow \tan( \varphi_+)  >0 \quad l\rightarrow -\infty \\
%\theta_{l+1}-\theta_l &\rightarrow \tan(\varphi_-)  <0 \quad l\rightarrow +\infty .
%\end{split}
%\end{equation}
}
\item[(v)]{
 For every $l \in \Wholes$ we have the bound
 \begin{equation}
   \abs{ \varphi_\pm}^2 + \norm{ v_l }_{H^1} \le K \abs{ c - c_* }.
 \end{equation}
}
\end{itemize}
\end{thm}

\subsection{The Nagumo LDE}
\label{sec:mr:ex:i}

As an example, we return to the Nagumo LDE
\begin{equation}
\label{eq:mr:nag:lde}
\dot{u}_{ij} = u_{i+1, j} + u_{i, j+1} + u_{i-1, j}  + u_{i, j-1} - 4u_{ij} + g(u_{ij} ; \rho)
\end{equation}
in which the nonlinearity is given by the scaled cubic
\begin{equation}
g(u ; \rho) = \frac{5}{2} g_{\mathrm{cub}}(u) = \frac{5}{2} (u^2 - 1)( \rho - u)
\end{equation}
for some detuning parameter $\rho \in (-1, 1)$.
In the terminology of \sref{eq:mr:mainLDE},
we hence have
\begin{equation}
f\big(
 v_1, \ldots, v_5)
= v_1 + v_2 + v_3 + v_4 - 4 v_5 + g(v_5; \rho),
\end{equation}
which shows that (Hf) is satisfied upon picking $u_\pm = \pm 1$.

Turning to $(H\Phi)$, we note that the results in \cite{MPB} show
that for each $\zeta \in [0, 2 \pi]$
and $\rho \in (-1, 1)$ there is a unique wavespeed
$c = c_{\rho,\zeta}$ for which the
system
\begin{equation}
\label{eq:mr:ex:mfde:for:phi}
c \Phi'(\xi) = \Phi(\xi + \cos \zeta)
+ \Phi(\xi - \cos \zeta)
+ \Phi(\xi + \sin \zeta)
+ \Phi(\xi - \sin \zeta)
-4 \Phi(\xi) + g\big( \Phi(\xi) ;\rho \big)
\end{equation}
admits a monotonic solution $\Phi = \Phi_{\rho,\zeta}$ that
also satisfies the limits \sref{eq:mr:wavelim}. Figure
\ref{fig:mr:ex:i:speeds:mc} contains polar plots of the
$\zeta \mapsto c_{\rho,\zeta}$ relation, which can be very delicate
whenever $\abs{\rho}$ is small.

By symmetry, we have
$c_{\rho,\zeta} = - c_{-\rho,\zeta}$
and hence $c_{0,\zeta} = 0$
for all angles $\zeta \in  [0,2 \pi]$.
Upon writing
\begin{equation}
\rho_*(\zeta) = \sup\{ \rho: c_{\rho, \zeta} = 0 \},
\end{equation}
the results in \cite{MPB}
show that $0 \le \rho_*(\zeta) < 1$
for all $\zeta \in [0, 2 \pi]$.
In particular, this means that
$(H\Phi)$ is satisfied whenever
$\tan \zeta$ is rational (or infinite)
and $\rho_*(\zeta) < \abs{\rho} < 1$.
Under the same conditions,
the discussion in \cite[{\S}6]{HJHSTB2D}
uses arguments based on the comparison principle
%and the work in \cite{MPB}
to show that also (HS1)-(HS3)
are valid.

The verification of the conditions in (HM) is
much more subtle. In order to make the angular
dependence fully explicit, we first pick
\begin{equation}
\label{eq:mr:ex:nag:tilde:sigma}
(\tilde{\sigma}_A, \tilde{\sigma}_B) = (\cos \zeta, \sin \zeta)
\end{equation}
and consider the operators
\begin{equation}
\tilde{\mathcal{L}}_z: H^1(\Real;\Real) \to L^2(\Real;\Real)
\end{equation}
that act as
\begin{equation}
\label{eq:mr:ex:nag:l:z}
\begin{array}{lcl}
[\tilde{\mathcal{L}}_z p](\xi)
& = & -c_* p'(\xi)
  + e^{ - \tilde{\sigma}_B z} p(\xi + \tilde{\sigma}_A )
  + e^{ \tilde{\sigma}_A z}  p(\xi + \tilde{\sigma}_B )
  + e^{ \tilde{\sigma}_B z} p(\xi - \tilde{\sigma}_A )
  + e^{-\tilde{\sigma}_A z} p(\xi - \tilde{\sigma}_B )
  -4 p(\xi)
\\[0.2cm]
& & \qquad
  + g'\big(\Phi_*(\xi) ; \rho \big) p(\xi) .
\end{array}
\end{equation}
Writing $\tilde{\lambda}_{z}$ for the branch of eigenvalues
for $\tilde{\mathcal{L}}_z$
bifurcating from $\tilde{\lambda}_0 = 0$
and comparing this to the branch $\lambda_z$
defined in Lemma \ref{lem:mr:def:branch:lambda:z},
we have
\begin{equation}
\tilde{\lambda}_{z} = \lambda_{\sigma_*^{-1} z} ,
\end{equation}
with $\sigma_* > 0$ the smallest number so that
\begin{equation}
\sigma_* \big(\tilde{\sigma}_A, \tilde{\sigma}_B )
 = (\sigma_A, \sigma_B) \in \mathbb{Z}^2.
\end{equation}
In particular, we have
\begin{equation}
[\partial_{z}\tilde{\lambda}_{z} ]_{z = 0}
= \sigma_*^{-1} [\partial_z \lambda_z]_{z=0}.
\end{equation}
In view of the similar rescalings \sref{eq:mr:rescalings}
and the fact that the statements in Theorem
\ref{thm:mr:ex:corner} merely depend on the signs
of the quantities $[\partial^2_\varphi d_{\varphi}]_{\varphi=0}$
and $[\partial_z^2 \lambda_z]_{z=0}$,
we simply write\footnote{
  The terms involving  $\pm$ and $\mp$
  should be evaluated twice, once with the top signs
  and once with the bottom signs. For example,
  $\pm p(\xi \mp 1) = p(\xi -1 ) - p (\xi +1)$.
}
\begin{equation}
\begin{array}{lcl}
[\mathcal{L}_{z;\rho,\zeta} p](\xi)
& = & -c_{\rho, \zeta} p'(\xi)
+ e^{ \mp z \sin \zeta } p(\xi \pm \cos \zeta )
+ e^{ \pm z \cos \zeta} p(\xi \pm \sin \zeta)
%  + e^{ - z \sin \zeta } p(\xi + \cos \zeta )
%  + e^{  z \cos \zeta }  p(\xi + \sin \zeta )
%  + e^{ z \sin \zeta} p(\xi - \cos \zeta )
%  + e^{-z \cos \zeta} p(\xi - \sin \zeta )
  -4 p(\xi)
\\[0.2cm]
& & \qquad
  + g'\big(\Phi_{\rho,\zeta}(\xi) ; \rho \big) p(\xi)
\end{array}
\end{equation}
and focus on the eigenvalues $\lambda_{z; \rho,\zeta}$
and eigenfunctions $\phi_{z; \rho,\zeta}$
bifurcating from $(0, \Phi'_{\rho,\zeta} )$
for $\rho_*(\zeta) < \abs{\rho} < 1$.
We write $\psi_{\rho, \zeta}$
for the solution to the adjoint
equation
\begin{equation}
\label{eq:mr:ex:mfde:for:psi}
\begin{array}{lcl}
-c_{\rho, \zeta} \, \psi'(\xi)
& = &
 \psi(\xi \pm \cos \zeta )
+  \psi(\xi \pm \sin \zeta)
%  + e^{ - z \sin \zeta } p(\xi + \cos \zeta )
%  + e^{  z \cos \zeta }  p(\xi + \sin \zeta )
%  + e^{ z \sin \zeta} p(\xi - \cos \zeta )
%  + e^{-z \cos \zeta} p(\xi - \sin \zeta )
  -4 \psi(\xi)
%\\[0.2cm]
%& & \qquad
  + g'\big(\Phi_{\rho,\zeta}(\xi) ; \rho \big) \psi(\xi)
\end{array}
\end{equation}
that is normalized to have
$\langle \psi_{\rho,\zeta} , \Phi_{\rho,\zeta}' \rangle = 1$.

In our context, the operators defined in
\sref{eq:prlm:def:A:k}
and \sref{eq:prlm:def:A:k}
act as
\begin{equation}
\label{eq:mr:ex:nag:defs:a:i:ii:b:i}
\begin{array}{lcl}
{[}A_1 p](\xi) & = & \mp \sin( \zeta ) p(\xi \pm \cos \zeta)
  \pm \cos (\zeta) p(\xi \pm \sin \zeta),
\\[0.2cm]
{[}A_2 p](\xi) & = & \sin( \zeta)^2  p(\xi \pm \cos \zeta)
  + \cos( \zeta)^2 p(\xi \pm \sin \zeta),
\\[0.2cm]
{[}B_1 p](\xi) & = &
   \mp \cos( \zeta)  p(\xi \pm \cos \zeta)
   \mp \sin( \zeta)  p(\xi \pm \sin \zeta) .
\end{array}
\end{equation}
In particular, Lemma
\ref{lem:prlms:ids:for:lambda:derivs}
allows us to compute
\begin{equation}
\begin{array}{lcl}
{[}\partial_z \lambda_{z;\rho,\zeta}]_{z=0}
& = &
 \mp \sin\zeta \langle \psi_{\rho,\zeta}, \Phi_{\rho,\zeta}'(\cdot \pm \cos \zeta) \rangle
  \pm \cos \zeta \langle \psi_{\rho,\zeta}, \Phi_{\rho,\zeta}'(\cdot \pm \sin \zeta)
 \rangle ,
\\[0.2cm]
\end{array}
\end{equation}
which in turn allows us to find
$[\partial_z \phi_{z;\rho,\zeta}]_{z=0}$
by solving the MFDE
\begin{equation}
\label{eq:mr:ex:mfde:for:phi:1}
\begin{array}{lcl}
\mathcal{L}_{0;\rho, \zeta}
[\partial_z \phi_{z; \rho, \zeta}]_{z=0}
& = &
 \pm \sin(\zeta)  \Phi_{\rho,\zeta}'(\cdot \pm \cos \zeta)
  \mp \cos (\zeta)  \Phi_{\rho,\zeta}'(\cdot \pm \sin \zeta)
 + [\partial_z \lambda_{z;\rho,\zeta}]_{z=0} \Phi_{\rho, \zeta}'.
\end{array}
\end{equation}
In addition, item (iv) of Lemma \ref{lem:mr:angl:dep}
shows that
\begin{equation}
[\partial_\zeta c_{\rho,\zeta}] =
 [\partial_z \lambda_{z;\rho,\zeta}]_{z =0},
  \qquad
  [\partial_\zeta \Phi_{\rho,\zeta}]
    = [\partial_z \phi_{z;\rho,\zeta}]_{z =0} .
\end{equation}

Turning to the second derivatives,
we again use
Lemma \ref{lem:prlms:ids:for:lambda:derivs}
to compute
\begin{equation}
\begin{array}{lcl}
{[}\partial_z^2 \lambda_{z;\rho,\zeta}]_{z=0}
& = & \sin(\zeta)^2  \langle \psi_{\rho,\zeta},
     \Phi_{\rho,\zeta}'(\cdot \pm \cos \zeta) \rangle
  + \cos(\zeta)^2 \langle \psi_{\rho,\zeta} , \Phi_{\rho,\zeta}'(\cdot \pm \sin \zeta) \rangle
\\[0.2cm]
& & \qquad
\mp 2 \sin \zeta \langle \psi_{\rho,\zeta},
    [\partial_z \phi_{z;\rho,\zeta}]_{z=0}(\cdot \pm \cos \zeta)
    \rangle
\\[0.2cm]
& & \qquad
    \pm 2 \cos \zeta \langle \psi_{\rho,\zeta} ,
       [\partial_z \phi_{z;\rho,\zeta}]_{z=0}(\cdot \pm \sin \zeta)
  \rangle
\\[0.2cm]
& & \qquad
  - 2 [\partial_z \lambda_{z;\rho,\zeta}]_{z=0}
    \langle \psi_{\rho,\zeta}, [\partial_z \phi_{z;\rho,\zeta}]_{z=0} \rangle .
\end{array}
\end{equation}
We remark that the last line vanishes in principle
if the normalization \sref{eq:mr:norm:cnd:on:phi:z} is imposed.
However, numerically it is convenient to be free to utilize a
different normalization, in which case
this term should be included.

Finally, in view of the fact
that that $D^2f$ acts only on its fifth argument,
we can  use
Lemma \ref{lem:prlms:ids:for:Phi:phi:derivs}
to obtain
\begin{equation}
\begin{array}{lcl}
[\partial_\zeta^2 c_{\rho,\zeta}]
& = &
 \langle \psi_{\rho,\zeta},  g''(\Phi_{\rho,\zeta};\rho)
\big[\partial_\zeta \Phi_{\rho, \zeta} ]^2 \rangle
\\[0.2cm]
& & \qquad
 + \sin(\zeta)^2 \langle \psi_{\rho,\zeta},  \Phi_{\rho,\zeta}''(\cdot \pm \cos \zeta) \rangle
 + \cos(\zeta)^2 \langle \psi_{\rho,\zeta}, \Phi_{\rho,\zeta}''(\cdot \pm \sin \zeta) \rangle
\\[0.2cm]
& & \qquad
   \mp \cos \zeta \langle \psi_{\rho,\zeta}, \Phi_{\rho,\zeta}'(\cdot \pm \cos \zeta) \rangle
   \mp \sin \zeta \langle \psi_{\rho,\zeta}, \Phi_{\rho,\zeta}'(\cdot \pm \sin \zeta) \rangle
\\[0.2cm]
& & \qquad
\mp  2 \sin\zeta \langle \psi_{\rho,\zeta},
       [\partial_\zeta \Phi_{\rho,\zeta}'](\cdot \pm \cos \zeta)
    \rangle
%\\[0.2cm]
%& & \qquad
   \pm 2 \cos \zeta \langle \psi_{\rho,\zeta},
       [\partial_\zeta \Phi_{\rho,\zeta}'](\cdot \pm \sin \zeta)
  \rangle
\\[0.2cm]
& & \qquad
 - 2 [\partial_\zeta c_{\rho,\zeta}]
   \langle \psi_{\rho,\zeta},
     [\partial_\zeta \Phi_{\rho,\zeta}']
   \rangle .
\end{array}
\end{equation}
The last line can be ignored if indeed
$[\partial_\zeta c_{\rho,\zeta}] = 0$.

\begin{figure}[t]
\centering
\includegraphics[width=\textwidth]{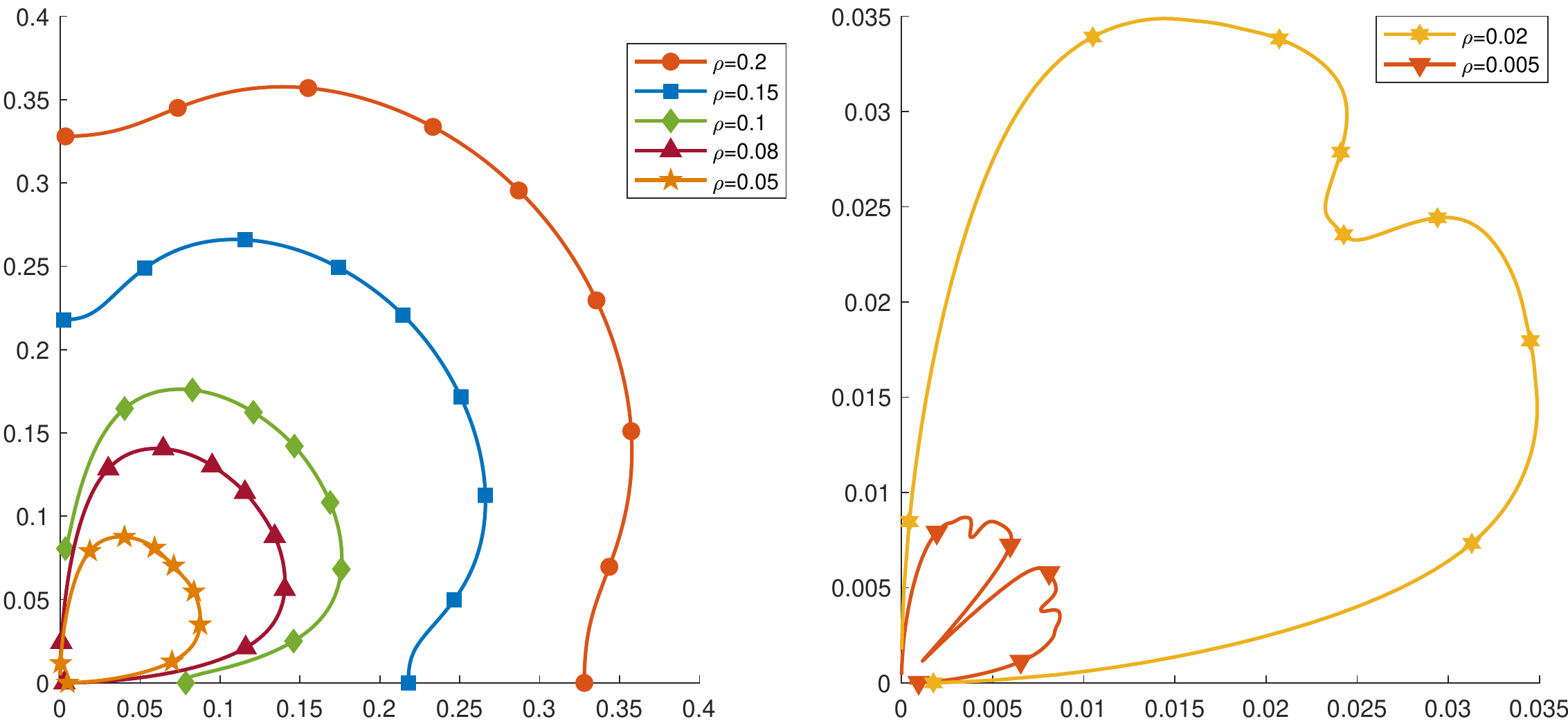}
\caption{Both panels contain polar plots of the $\zeta \mapsto c_{\rho,\zeta}$
relation, for various values of $\rho > 0$. Since $c \le 0 $ in this setting,
we have actually plotted the points $-c_{\rho,\zeta}(\cos\zeta,\sin\zeta)$
for $0 \le \zeta \le \frac{\pi}{2}$. Notice the extra minima that start to form in the directions
$\tan \zeta = 1$ and subsequently $\tan \zeta = \frac{2}{3}$ as $\rho$ is decreased.
}
\label{fig:mr:ex:i:speeds:mc}
\end{figure}

In the special cases
where $\zeta = k \frac{\pi}{4}$ for some $k \in \Wholes$,
we have $A_1 = 0$ and hence
\begin{equation}
[\partial_z \lambda_{z; \rho,k \frac{\pi}{4} }]_{z=0}
= [\partial_\zeta c_{\rho,\zeta}]_{\zeta= k \frac{\pi}{4} }
= 0,
\qquad
[\partial_z \phi_{z; \rho,k \frac{\pi}{4}}]_{z=0}
= [\partial_\zeta \Phi_{\rho,\zeta}]_{\zeta= k \frac{\pi}{4} }
= 0.
\end{equation}
For $\zeta = 0$ this allows us to write
\begin{equation}
\label{eq:mr:ex:hor:snd:derivs}
\begin{array}{lcl}
[\partial_z^2 \lambda_{z; \rho, 0 }]_{z=0}
& = & 2 \langle  \psi_{\rho,\zeta}, \Phi_{\rho,\zeta}' \rangle = 2,
\\[0.2cm]
[\partial_\zeta^2 c_{\rho,\zeta}]_{\zeta = 0}
& = & \langle \psi_{\rho,\zeta}, 2 \Phi_{\rho,\zeta}''
+ \Phi_{\rho,\zeta}'(\cdot - 1)
   - \Phi_{\rho,\zeta}'(\cdot + 1) \rangle .
\end{array}
\end{equation}
On the other hand, for $\zeta = \frac{\pi}{4}$
we have
\begin{equation}
\label{eq:mr:ex:diag:snd:derivs}
\begin{array}{lcl}
[\partial_z^2 \lambda_{z; \rho, \frac{\pi}{4}}]_{z=0}
& = & \langle  \psi_{\rho,\zeta}, \Phi_{\rho,\zeta}'(\cdot + \frac{1}{2} \sqrt{2})
   + \Phi_{\rho,\zeta}'(\cdot - \frac{1}{2} \sqrt{2}) \rangle ,
\\[0.2cm]
[\partial_\zeta^2 c_{\rho,\zeta}]_{\zeta = \frac{\pi}{4} }
& = & \langle \psi_{\rho,\zeta},
   \Phi_{\rho,\zeta}''(\cdot + \frac{1}{2} \sqrt{2})
 + \Phi_{\rho, \zeta}''(\cdot - \frac{1}{2} \sqrt{2} )
   + \sqrt{2} \Phi_{\rho,\zeta}'(\cdot - \frac{1}{2} \sqrt{2} )
     - \sqrt{2} \Phi_{\rho,\zeta}'(\cdot + \frac{1}{2} \sqrt{2} ) \rangle .
\end{array}
\end{equation}
Since $\psi_{\rho,\zeta}$ and $\Phi_{\rho,\zeta}'$ are strictly positive
we hence see that
\begin{equation}
  [\partial_z^2 \lambda_{z; \rho, k\frac{\pi}{4}}]_{z=0} > 0
\end{equation}
for all $k \in \Wholes$. The numerical results
in \cite[{\S}6]{HJHSTB2D} suggest that this
inequality extends to a wide range of $(\rho, \zeta)$
and we take this for granted for the remainder of our discussion.
However, even for the straightforward expressions
\sref{eq:mr:ex:hor:snd:derivs}-\sref{eq:mr:ex:diag:snd:derivs},
it is not clear
whether the quantity
$c_{\rho, \zeta} + \partial_\zeta^2 c_{\rho, \zeta}$
has a sign.

For any fixed $\zeta_*$, we now introduce the notation
\begin{equation}
\label{eq:mr:ex:i:defs:cg:kappa:d}
c_g(\rho) = [\partial_z \lambda_{z;\rho, \zeta}]_{z=0},
\qquad
\kappa_{d}(\rho) = c_{\rho, \zeta_*} + \big[\partial_\zeta^2 c_{\rho, \zeta} \big]_{\zeta = \zeta_*}
\end{equation}
for the group velocity and second derivative of the directional dispersion
that play a role in Theorem \ref{thm:mr:ex:corner}.
In particular, to apply this result
we need $c_g(\rho) = 0$ and $\kappa_d(\rho) \neq 0$.
Since $c_{\rho,\zeta} \le 0$ whenever $\rho \ge 0$,
we have an interior corner for $\kappa_d(\rho) < 0$
and an exterior corner for $\kappa_d(\rho ) > 0$. In both cases the corner
travels in the rightward direction (provided $\abs{\zeta} < \frac{\pi}{2}$).

\begin{figure}[t]
\centering
\includegraphics[width=\textwidth]{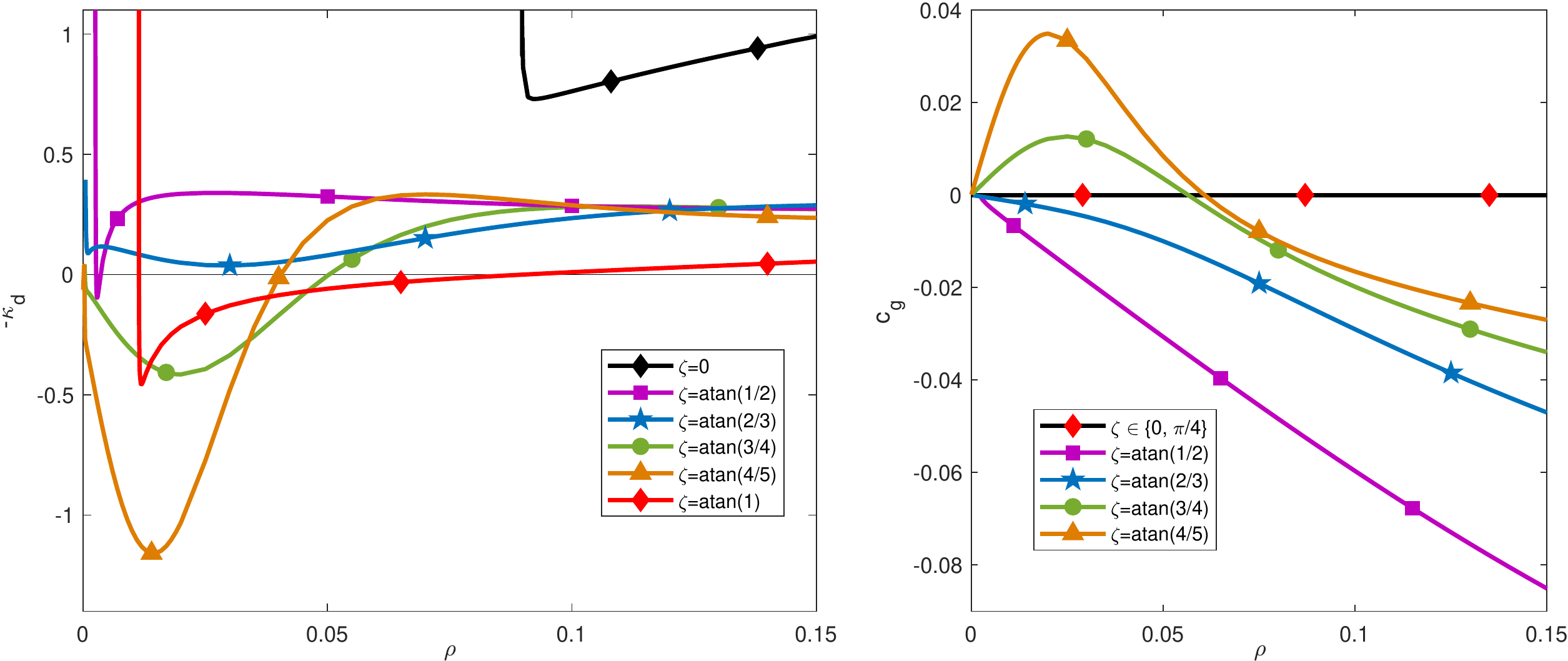}
\caption{The left panel contains numerically computed values for $-\kappa_d(\rho)$. The sharp spikes
occur at the critical value $\rho_*(\zeta)$ where pinning sets in.
We note that sign changes appear for $\zeta = \frac{\pi}{2}$ but not for $\zeta = 0$. % and $\zeta_* = \arctan(\frac{1}{2})$.
In particular, the identity $c_g \equiv 0$ for these directions implies that interior and exterior
corners can both occur for $\zeta = \frac{\pi}{2}$, while the horizontal direction $\zeta = 0$ features interior corners only.
The right panel contains numerically computed values for $c_g(\rho)$. Notice
the zero-crossings for $\tan \zeta = \frac{3}{4}$ and $\tan \zeta = \frac{4}{5}$,
which indicates the presence of interior corners at these two critical values for $\rho$.
}
\label{fig:mr:ex:i:cg:kappa:d}
\end{figure}

In Figure \ref{fig:mr:ex:i:cg:kappa:d} we have
numerically computed the quantities \sref{eq:mr:ex:i:defs:cg:kappa:d}
for a range of rational directions. In all cases,
we also confirmed numerically that $[\partial_z^2 \lambda_{z,\rho,\zeta}]_{z=0} > 0$.
In particular, the results predict interior corners travelling in the horizontal direction $\zeta_* = 0$,
while both types of corners can travel in the diagonal direction $\zeta_* = \frac{\pi}{4}$.
In addition, for two critical values of $\rho > 0$ there are interior corners that travel
in the direction $\zeta_* = \arctan(3/4)$ respectively $\zeta_* = \arctan(4/5)$.
% at two critical values for $\rho$.
%In addition, and
%exterior corners travelling in the directions $\zeta = \arctan(3/4)$ and $\zeta = \arctan(4/5)$
%and both types of corners travelling in the diagonal direction $\zeta = \frac{\pi}{4}$.
To obtain these results,
we simultaneously solved the systems
\sref{eq:mr:ex:mfde:for:phi},
\sref{eq:mr:ex:mfde:for:psi} and
\sref{eq:mr:ex:mfde:for:phi:1}.
For well-posedness reasons,
we added the extra terms $\gamma \Phi''$,
$\gamma \psi''$
respectively $\gamma [\partial_z \phi_{z;\rho,\zeta}]''$
to the right-hand side of each equation,
taking $\gamma = 10^{-6}$. % some small $\gamma > 0$.
For the
precise procedure, we refer
to \cite[{\S}6]{HJHSTB2D}.

%Crystallographic pinning conjecture.
%Indicates discontinuities in pinning region
%boundary . Analog of saying that minima
%occur naturally
%at rational directions?

\subsection{Bichromatic Nagumo LDE}
\label{sec:mr:ex:ii}

We here reconsider the Nagumo LDE
\begin{equation}
\label{eq:mr:nag:lde:ex:ii}
\dot{u}_{ij} = \alpha \big[
 u_{i+1, j} + u_{i, j+1} + u_{i-1, j}  + u_{i, j-1} - 4u_{ij}
\big] + g_{\mathrm{cub}}(u_{ij} ; \rho),
\end{equation}
but are now
interested
in so-called bichromatic planar travelling wave solutions.
Such solutions  can be written in the form
\begin{equation}
\label{eq:mr:ex:ii:bichrom:wave:ansatz}
u_{ij}(t)
 = \left\{ \begin{array}{lcl}
   \Phi^{(u)}( i \cos \phi + j \sin \phi + c t) & & \hbox{if } i+j \hbox{ is even}, \\[0.2cm]
   \Phi^{(v)}( i \cos \phi + j \sin \phi + c t) & & \hbox{if } i+j \hbox{ is odd}, \\[0.2cm]
   \end{array}
\right.
\end{equation}
for some wavespeed $c \in \Real$
and $\Real^2$-valued waveprofile
\begin{equation}
\Phi = (\Phi^{(u)}, \Phi^{(v)}): \Real \to \Real^2.
\end{equation}
These waves fit into the framework
of this paper, since they can be seen
as travelling wave
solutions for the `doubled' LDE
\begin{equation}
\label{eq:mr:nag:ex:ii:doubled:lde}
\begin{array}{lcl}
\dot{u}_{ij} & = & \alpha \big[ v_{i+1,j} + v_{i, j+1} + v_{i-1, j} + v_{i, j-1} - 4 u_{ij} \big]
 + g_{\mathrm{cub}}\big( u_{ij} ; \rho\big) ,
\\[0.2cm]
\dot{v}_{ij} & = & \alpha \big[ u_{i+1,j} + u_{i, j+1} + u_{i-1, j} + u_{i, j-1} - 4 v_{ij} \big]
 + g_{\mathrm{cub}}\big( v_{ij} ; \rho\big) .
\\[0.2cm]
\end{array}
\end{equation}

\begin{figure}[t]
\centering
\includegraphics[width=\textwidth]{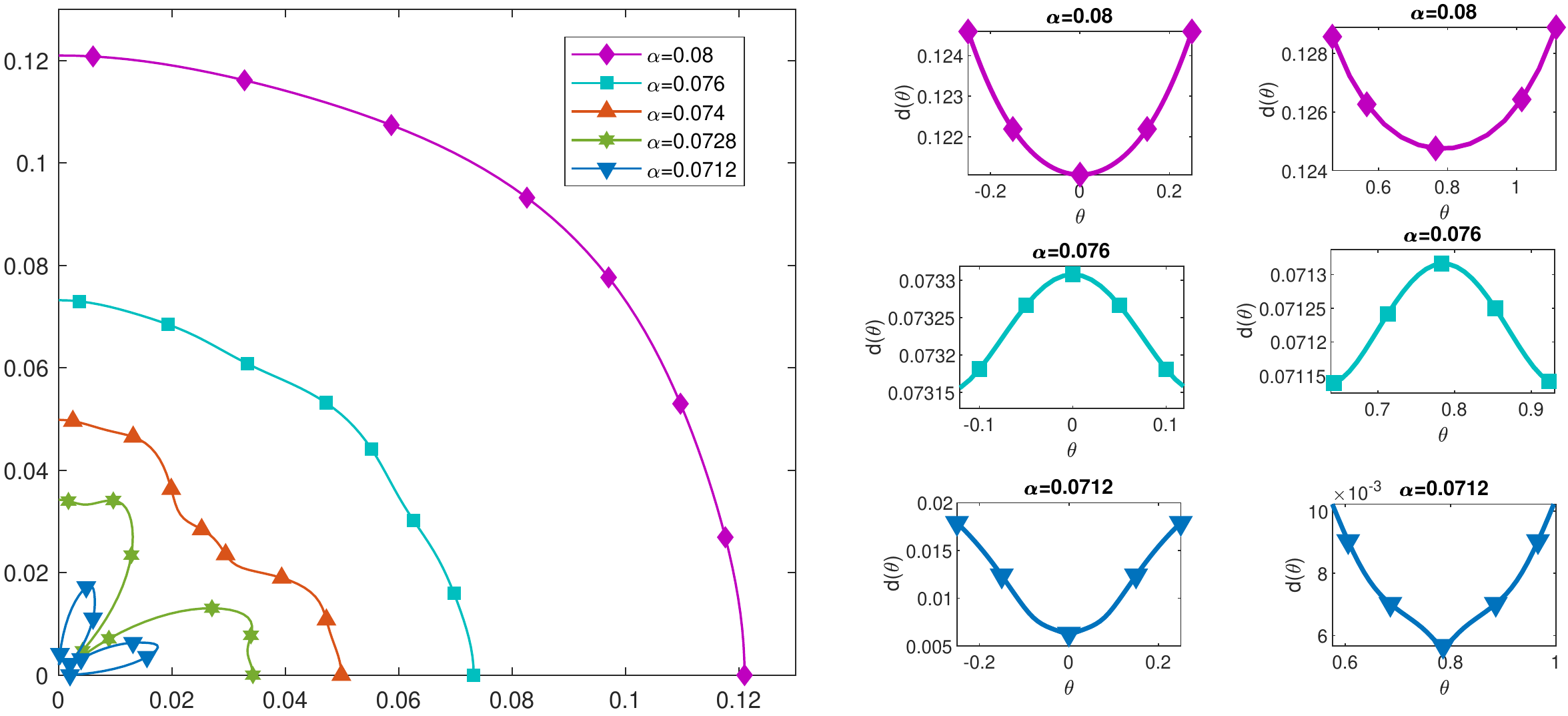}
\caption{The left panel contains polar plots of the $\zeta \mapsto c_{\rho, \alpha,\zeta}$
relation, with fixed $\rho =0$. In particular, the curves consist of the points $c_{\rho,\alpha,\zeta}(\cos\zeta,\sin\zeta)$.
The right panel depicts the directional dispersion $d(\zeta) = \frac{c_{\rho,\alpha, \zeta}}{\cos (\zeta - \zeta_*) }$,
with $\zeta_* = 0$ for the left column and $\zeta_* = \frac{\pi}{4}$ for the right column, again with $\rho = 0$.
These results strongly suggest that $[\partial_{\zeta}^2 d(\zeta) ]_{\zeta = \zeta_*}$ can take both signs as the diffusion
coefficient $\alpha$ is varied. In particular, both the horizontal and diagonal directions can
have interior and exterior corners.
}
\label{fig:mr:ex:ii:bic:sp}
\end{figure}

We now introduce the notation
\begin{equation}
G(u, v; \rho, \alpha) =
\left(
  \begin{array}{ccc}
     4\alpha (v - u ) + g_{\mathrm{cub}}(u;\rho) \\[0.2cm]
     4\alpha (u - v ) + g_{\mathrm{cub}}(v;\rho) \\[0.2cm]
  \end{array}
  \right),
\end{equation}
together with the matrix
\begin{equation}
\mathcal{J} =
\left( \begin{array}{cc} 0 & 1 \\ 1 & 0 \end{array} \right)
\end{equation}
and the operator
\begin{equation}
\Delta_{z, \zeta} p =
 e^{ \mp z \sin \zeta } p(\xi \pm \cos \zeta )
+ e^{ \pm z \cos \zeta} p(\xi \pm \sin \zeta)
  -4 p(\xi) .
\end{equation}
Substitution of the bichromatic wave Ansatz \sref{eq:mr:ex:ii:bichrom:wave:ansatz}
into the LDE \sref{eq:mr:nag:lde:ex:ii}
leads to the travelling wave MFDE
\begin{equation}
\label{eq:mr:nag:bc:trv:wave}
 c \Phi' = \alpha \mathcal{J} \Delta_{0, \zeta} \Phi
   + G(\Phi ; \rho, \alpha),
\end{equation}
while the spatially homogeneous equilibria $(u_{ij}, v_{ij}) = (u,v)$
to
the doubled LDE \sref{eq:mr:nag:ex:ii:doubled:lde}
satisfy the system
\begin{equation}
\label{eq:mr:ex:ii:bichrom:eqlb}
G(u, v; \rho,\alpha) = 0.
\end{equation}

The results in \cite{HJHBICHROM} show that there exists an open set of values $(\rho, \alpha)$
with $\rho \in (-1, 1)$ and $\alpha > 0$
for which \sref{eq:mr:ex:ii:bichrom:eqlb} admits solutions
$(u_{\mathrm{bc}}, v_{\mathrm{bc}}) \in (0, 1)^2$
that are stable spatially homogeneous equilibria for \sref{eq:mr:nag:ex:ii:doubled:lde}.
By applying the theory in \cite{HJHNEGDIF}, we
hence obtain the existence of solutions to \sref{eq:mr:nag:bc:trv:wave}
that satisfy the limits
\begin{equation}
\label{eq:mr:nag:bc:trv:wave:lims}
\lim_{\xi \to - \infty} ( \Phi^{(u)}, \Phi^{(v)} ) = (0, 0),
\qquad
\lim_{\xi \to + \infty} ( \Phi^{(u)}, \Phi^{(v)} ) =
(u_{\mathrm{bc}}, v_{\mathrm{bc}}) ,
\end{equation}
together with similar solutions that connect
$(u_{\mathrm{bc}}, v_{\mathrm{bc}})$ with $(1,1)$.

We remark that the existence theory in \cite{HJHNEGDIF} does not prescribe whether $c \neq 0$ or $c = 0$.
However, in the special case $\zeta = \frac{\pi}{4}$ the travelling wave
system can be written as
\begin{equation}
\begin{array}{lcl}
  c [\Phi^{(u)}]'(\xi) & = &
    2 \alpha \big[ \Phi^{(v)}(\xi + \frac{1}{2}\sqrt{2} ) + \Phi^{(v)}(\xi - \frac{1}{2} \sqrt{2} ) - 2 \Phi^{(u)}(\xi)]
      + g_{\mathrm{cub}}\big( \Phi^{(u)}(\xi) ; \rho \big),
\\[0.2cm]
  c [\Phi^{(v)}]'(\xi) & = &
   2 \alpha \big[ \Phi^{(u)}(\xi + \frac{1}{2}\sqrt{2} ) + \Phi^{(u)}(\xi - \frac{1}{2} \sqrt{2} ) - 2 \Phi^{(v)}(\xi)]
       + g_{\mathrm{cub}}\big( \Phi^{(v)}(\xi) ; \rho \big) .
\end{array}
\end{equation}
After a spatial rescaling, this corresponds precisely with the bichromatic travelling wave
MFDE \cite[Eq. (2.5)]{HJHBICHROM} encountered in the one-dimensional spatial setting.
We remark here that this does \textbf{not} hold for the horizontal direction $\zeta = 0$.

In any case, this observation allows us to apply \cite[Thm. 2.3]{HJHBICHROM}.
As a consequence, for $\zeta = \frac{\pi}{4}$ there exists
an open set of values $(\rho, \alpha)$ for which the wavespeed does not vanish,
allowing us to verify $(H\Phi)$. By continuity in $\zeta$, this hence also holds
for nearby angles. In addition, our numerical results suggest
that the diagonal direction is the first to become pinned as $\alpha$ is decreased;
see Figure \ref{fig:mr:ex:ii:bic:sp}.
This contrasts the situation encountered in the monochromatic case,
where the diagonal waves satisfy the same travelling wave MFDE as the horizontal waves,
but with a doubled diffusion coefficient. As a result, the
monochromatic horizontal waves pin earlier than their diagonal counterparts.

Since the spectral conditions (HS1)-(HS3) can be verified
with techniques similar to those used for the monochromatic case,
we now turn our attention to (HM). In particular,
writing $(c_{\rho,\alpha,\zeta}, \Phi_{\rho,\alpha,\zeta})$
for the solution to \sref{eq:mr:nag:bc:trv:wave}
that satisfies the limits \sref{eq:mr:nag:bc:trv:wave:lims},
we introduce the operator
\begin{equation}
\mathcal{L}_{z;\rho,\alpha, \zeta} p =
 - c_{\rho, \alpha, \zeta} p'
 + \alpha \mathcal{J} \Delta_{z, \zeta} p
   + DG(\Phi_{\rho, \alpha, \zeta} ; \rho, \alpha) p
\end{equation}
for any $p = ( p^{(u)}, p^{(v)} ) \in H^1(\Real;\Real)
  \times H^1(\Real;\Real)$.
We now introduce the notation
 $A_1^{\mathrm{mc}}$, $A_2^{\mathrm{mc}}$ and $B_1^{\mathrm{mc}}$
for the operators \sref{eq:mr:ex:nag:defs:a:i:ii:b:i}
defined for the monochromatic equation.
In addition, we write
$A_1^{\mathrm{bc}}$, $A_2^{\mathrm{bc}}$ and $B_1^{\mathrm{bc}}$
for the operators defined in \sref{eq:prlm:def:A:k}
and \sref{eq:prlm:def:A:k}
associated to
the bichromatic problem \sref{eq:mr:nag:bc:trv:wave}.
It is not hard to verify the relations
\begin{equation}
A_1^{\mathrm{bc}} = \alpha \mathcal{J} A_1^{\mathrm{mc}}  ,
\qquad
A_2^{\mathrm{bc}} = \alpha \mathcal{J} A_2^{\mathrm{mc}} ,
\qquad
B_1^{\mathrm{bc}} = \alpha \mathcal{J} B_1^{\mathrm{mc}}.
\end{equation}

We write $(\lambda_{z; \rho, \alpha , \zeta} , \phi_{z; \rho,\alpha,\zeta})$
and $\psi_{\rho,\alpha, \zeta}$ for the analogs
of the similar named expressions defined in {\S}\ref{sec:mr:ex:i}.
Since $A_1^{\mathrm{bc}} =0 $
for $\zeta = 0$ and $\zeta = \frac{\pi}{4}$,
we again have
\begin{equation}
[\partial_z \lambda_{z; \rho, \alpha, k \frac{\pi}{4} }]_{z=0}
= [\partial_\zeta c_{\rho,\alpha,\zeta}]_{\zeta= k \frac{\pi}{4} }
= 0,
\qquad
[\partial_z \phi_{z; \rho,\alpha,k \frac{\pi}{4}}]_{z=0}
= [\partial_\zeta \Phi_{\rho,\alpha,\zeta}]_{\zeta= k \frac{\pi}{4} }
= 0
\end{equation}
for all $k \in \Wholes$.

For $\zeta = 0$ this allows us to write
\begin{equation}
\label{eq:mr:ex:bichrom:hor:snd:derivs}
\begin{array}{lcl}
[\partial_z^2 \lambda_{z; \rho, \alpha, 0 }]_{z=0}
& = & 2 \alpha \langle  \psi_{\rho,\alpha,\zeta}, \mathcal{J} \Phi_{\rho,\alpha,\zeta}' \rangle ,
\\[0.2cm]
[\partial_\zeta^2 c_{\rho,\alpha,\zeta}]_{\zeta = 0}
& = & \alpha \langle \psi_{\rho,\alpha,\zeta},
 \mathcal{J}\big[
    2 \Phi_{\rho,\alpha,\zeta}''
    + \Phi_{\rho,\alpha,\zeta}'(\cdot- 1)
     - \Phi_{\rho,\alpha,\zeta}'(\cdot + 1)
    \big]
  \rangle .
\end{array}
\end{equation}
On the other hand, for $\zeta = \frac{\pi}{4}$
we have
\begin{equation}
\label{eq:mr:ex:bichrom:diag:snd:derivs}
\begin{array}{lcl}
[\partial_z^2 \lambda_{z; \rho,\alpha, \frac{\pi}{4}}]_{z=0}
& = & \alpha \langle  \psi_{\rho,\alpha,\zeta},
  \mathcal{J} \big[
    \Phi_{\rho,\alpha,\zeta}'(\cdot + \frac{1}{2} \sqrt{2})
   + \Phi_{\rho,\alpha,\zeta}'(\cdot - \frac{1}{2} \sqrt{2})
  \big]
  \rangle ,
\\[0.2cm]
[\partial_\zeta^2 c_{\rho,\alpha,\zeta}]_{\zeta = \frac{\pi}{4} }
& = & \alpha \langle \psi_{\rho,\alpha,\zeta},
  \mathcal{J} \big[\Phi_{\rho,\alpha,\zeta}''(\cdot + \frac{1}{2} \sqrt{2})
 + \Phi_{\rho,\alpha, \zeta}''(\cdot - \frac{1}{2} \sqrt{2} )
    \big] \rangle
\\[0.2cm]
& & \qquad
  + \sqrt{2} \alpha \langle \psi_{\rho,\alpha,\zeta},
    \mathcal{J} \big[
     \Phi_{\rho,\alpha,\zeta}'(\cdot - \frac{1}{2} \sqrt{2} )
     -  \Phi_{\rho,\alpha,\zeta}'(\cdot + \frac{1}{2} \sqrt{2} )
  \big]
  \rangle .
\end{array}
\end{equation}
Since both components of $\psi_{\rho,\alpha,\zeta}$ and $\Phi_{\rho,\alpha,\zeta}'$ are strictly positive,
we again see that
\begin{equation}
  [\partial_z^2 \lambda_{z; \rho, \alpha, k\frac{\pi}{4}}]_{z=0} > 0
\end{equation}
for all $k \in \Wholes$.

As before, it is unclear if the derivatives of $c$ have a sign.
As a consequence of the increasing number of components in the MFDEs,
our numerical method is at present unable to compute the desired
derivatives in the same fashion as above.
Instead, we simply compute the directional dispersion relation directly
and determine by inspection whether it is concave or convex;
see Figure \ref{fig:mr:ex:ii:bic:sp}. Interestingly enough,
we find that this characterization flips at least twice as $\alpha$ is decreased,
both for the horizontal direction $\zeta = 0$ and the diagonal direction $\zeta = \frac{\pi}{4}$.
In contrast to the monochromatic case,
we hence see that interior and exterior bichromatic corners
can both travel in the horizontal direction.

\section{Problem setup}
\label{sec:setup}

%The actual LDE in $(n,l)$ coordinates is given by
%\begin{equation}
%\dot{u}_{nl}(t) = u_{n \pm \sigma_A , l \pm \sigma_B} + u_{n \pm \sigma_B , l \mp \sigma_A }
% - 4 u_{nl} + f\big(u_{nl} \big)
%\end{equation}
In this section we setup a differential-algebraic equation
to describe
solutions to the LDE
\begin{equation}
\label{eq:set:lde}
\dot{u}_{nl}(t) = f\big( \pi^\times_{nl} u(t) \big)
\end{equation}
that can be written in the form
\begin{equation}
\label{eq:set:ansatz:u:with:xi}
u_{nl}(t) = \Xi_l( n + ct)
\end{equation}
for some sequence
\begin{equation}
\label{eq:set:seq:Xi}
\Xi: \Wholes \to W^{1, \infty}(\Real; \Real^d) .
\end{equation}
The elements $\Xi_l$ will be required to lie in the orbital vicinity of the waveprofile $\Phi_*$.
In particular, we formulate a global center manifold reduction that allows us to find an equivalent
two component difference equation of skew-product form.

For any sequence $\Xi$ of the form \sref{eq:set:seq:Xi}, we introduce the
notation
\begin{equation}
  \Xi^\diamond: \Wholes \to W^{1, \infty}(\Real; (\Real^d)^5)
\end{equation}
to refer to the expanded sequence
\begin{equation}
\Xi^\diamond_l =
\Big( \Xi_{l - \sigma_B } - \Xi_{l} , \Xi_{l + \sigma_A} - \Xi_l,
  \Xi_{l + \sigma_B} - \Xi_l , \Xi_{l - \sigma_A} - \Xi_l, 0 \Big) .
\end{equation}
In addition,
for any
\begin{equation}
\overline{p} = (p_1, \ldots, p_5)
\in C(\Real; (\Real^d)^5 )
\end{equation}
we introduce the function $\overline{\tau} \overline{p}$
that is given by
\begin{equation}
\begin{array}{lcl}
[\overline{\tau} \overline{p} ](\xi)
& = & \big(p_1(\xi + \sigma_A) , p_2(\xi + \sigma_B), p_3(\xi - \sigma_A),
  p_4 (\xi - \sigma_B) , p_5(\xi) \big) \in (\Real^d)^5.
\end{array}
\end{equation}
This allows us to write
\begin{equation}
\pi^{\times}_{nl} u (t) =
\tau \Xi_l(n + ct ) +
\overline{\tau} \Xi^\diamond_{l} (n + ct)
\end{equation}
for the function $u$
defined in \sref{eq:set:ansatz:u:with:xi}.
In particular,
this function $u$ %defined in \sref{eq:set:ansatz:u:with:xi}
satisfies the LDE
\sref{eq:set:lde}
if and only if
\begin{equation}
\label{eq:set:sys:sc:for:xi}
c \Xi_l'(\xi) = f \big( [\tau \Xi_l](\xi)
+ [\overline{\tau} \Xi^\diamond_l](\xi) \big)
\end{equation}
holds for all $l \in \Wholes$ and $\xi \in \Real$.

From now on, we often drop the explicit $\xi$-dependence.
For instance, we simply write
\begin{equation}
\label{eq:set:sys:sc:for:xi:sh}
c \Xi_l' = f \big( [\tau \Xi_l]
+ [\overline{\tau} \Xi^\diamond_l] \big)
\end{equation}
instead of the longer form \sref{eq:set:sys:sc:for:xi}.

For technical reasons it is advantageous
to recast \sref{eq:set:sys:sc:for:xi:sh}
as a $2d$-component system.
To this end, we introduce
the first differences
\begin{equation}
\Upsilon_l %(\xi)
 = \Xi_l %(\xi)
    - \Xi_{l-1}. %(\xi).
\end{equation}
In addition, we introduce the notation
\begin{equation}
s^\diamond[\Upsilon]: \Wholes
  \to C\big(\Real ; (\Real^d)^5 \big)
\end{equation}
for the summed sequence
\begin{equation}
\begin{array}{lcl}
s^\diamond[\Upsilon]_l
& = &
\Big(
  - \sum_{j=0}^{ \sigma_B- 1} \Upsilon_{l - j}, \,
  \sum_{j=1}^{\sigma_A} \Upsilon_{l + j},  \,
  \sum_{j=1}^{\sigma_B} \Upsilon_{l+j},   \,
  - \sum_{j=0}^{\sigma_A -1 } \Upsilon_{l - j}, \,
  0
\Big),
\end{array}
\end{equation}
in which we make the convention that sums
where the lower index is strictly larger
than the upper index are set to zero.
For example, in the special
case $(\sigma_A, \sigma_B) = (1, 0)$
we have
\begin{equation}
\begin{array}{lcl}
s^\diamond[\Upsilon]_l
& = &
\Big(
  0,
   \Upsilon_{l + 1},
  0  ,
  -  \Upsilon_{l },
  0
\Big).
\end{array}
\end{equation}
These definitions allow us to observe
that
\begin{equation}
\Xi^\diamond = s^\diamond[\Upsilon].
\end{equation}
In particular, the system \sref{eq:set:sys:sc:for:xi:sh}
can now be rewritten in the equivalent form
\begin{equation}
\label{eq:mr:sys:xi:upsilon}
\begin{array}{lcl}
\Xi_{l+1} - \Xi_{l} & = & \Upsilon_{l+1} ,
\\[0.2cm]
c \Xi'_l & = &
f\big( \tau\Xi_l + \overline{\tau} s^\diamond[\Upsilon]_l \big) .
\end{array}
\end{equation}

In the special case $c = c_*$,
the travelling wave solution
\sref{eq:mr:trv:wave:cross}
gives rise to
$l$-independent solutions
to \sref{eq:mr:sys:xi:upsilon}
of the form
\begin{equation}
\label{eq:set:trv:wave:xi:ups}
\big(\Xi_l, \Upsilon_l \big)
= \big(
  T_{\vartheta} \Phi_*,   0  \big)  ,
\end{equation}
in which $\vartheta \in \Real$
can be chosen arbitrarily.
Here we have introduced the left-shift operator
\begin{equation}
[T_{\vartheta} p](\xi) = p(\xi + \vartheta)
\end{equation}
for any $p \in C(\Real; \Real^d)$.

We now look for a branch of solutions
to \sref{eq:mr:sys:xi:upsilon}
that bifurcates from the travelling
waves \sref{eq:set:trv:wave:xi:ups}
for $c \neq c_*$. In particular,
we consider the Ansatz
\begin{equation}
\label{eq:set:ansatz:xi:ups}
\begin{array}{lcl}
\Xi_l
 & = & T_{\theta_l} \Phi_*
  + v_l  ,
\\[0.2cm]
\Upsilon_l
 & = & w_l ,
\end{array}
\end{equation}
for three sequences
\begin{equation}
(\theta, v, w): \Wholes \to \Real \times H^1(\Real;\Real^d) \times H^1(\Real;\Real^d)
\end{equation}
that much hence satisfy
\begin{equation}
\label{eq:setup:pb:eq:i}
\begin{array}{lcl}
v_{j+1} -v_j
 & = &
 w_{j+1} -T_{\theta_{j+1}}\Phi_*
 + T_{\theta_j} \Phi_* ,
\\[0.2cm]
c \big[ T_{\theta_j}\Phi' + v' \big]
& = &
 f\big( T_{\theta_j} \tau \Phi_* + \tau v
   + \overline{\tau} s^\diamond[w]_l \big) .
\end{array}
\end{equation}
In order to close the system,
we augment \sref{eq:setup:pb:eq:i}
by demanding that
\begin{equation}
\label{eq:set:trnsv:def:eq}
\langle T_{\theta_l} \psi_*,  v_l \rangle = 0
\end{equation}
for all $l \in \Wholes$.

We now set out to isolate the linear
and nonlinear parts
of the system \sref{eq:setup:pb:eq:i}.
For any $\tilde{v} \in C(\Real; (\Real^d)^5 )$ and
$\vartheta \in \Real$
we therefore introduce the function
$\mathcal{N}_{f; \vartheta}(\tilde{v}) \in C(\Real; \Real^d)$
that is given by
\begin{equation}
\mathcal{N}_{f;\vartheta}(\tilde{v})
%(\xi)
= f(   \tau  T_{\vartheta} \Phi_*  +   \tilde{v} \big)
  - f\big( \tau T_{\vartheta} \Phi_*  \big)
 - Df\big( \tau T_{\vartheta} \Phi_*  \big) \tilde{v} .
\end{equation}
%\begin{equation}
%\mathcal{N}_{f;\vartheta}(\tilde{v})
%%(\xi)
%= f(   \big[ \tau  T_{\vartheta} \Phi_* ](\xi) +   \tilde{v}(\xi) \big)
%  - f\big( [\tau T_{\vartheta} \Phi_* ](\xi) \big)
% - Df\big( [\tau T_{\vartheta} \Phi_* ](\xi) \big) \tilde{v}(\xi)
%\end{equation}
In addition,
for any phase $\vartheta \in \Real$ and
difference $\theta_d \in \Real$
we introduce the function
$\mathcal{N}_{\Phi_*;\vartheta}(\theta_{d}) \in C(\Real; \Real^d)$
by writing
\begin{equation}
\mathcal{N}_{\Phi_*;\vartheta}(\theta_{d})
  = T_{\vartheta + \theta_d}\Phi_*
    - T_{\vartheta}\Phi_*
    - T_{\vartheta} \Phi_*' \theta_{d}.
\end{equation}

Using these functions,
the system \sref{eq:setup:pb:eq:i}
can be recast as
\begin{equation}
\label{eq:setup:pb:eq:ii}
\begin{array}{lcl}
v_{j+1} -v_j
 & = &
 w_{j+1} - T_{\theta_j} \Phi_*' (\theta_{j+1} - \theta_j)
  -  \mathcal{N}_{\Phi;\theta_j}(\theta_{j+1} - \theta_j) ,
\\[0.2cm]
(c - c_*) \big[ T_{\theta_j} \Phi_*' + v' \big]
& = &
\mathcal{L}^{(\theta_j)}_* v_j
+ Df( T_{\theta_j} \tau \Phi_* )\overline{\tau} s^\diamond[w]_j
+ \mathcal{N}_{f; \theta_j}( \tau v + \overline{\tau} s^\diamond[w]_j ) ,
\end{array}
\end{equation}
in which we have
\begin{equation}
\mathcal{L}^{(\vartheta)}_* v
= - c_* v' + Df(T_{\vartheta} \tau \Phi_*) \tau v .
\end{equation}

For any $v \in L^2(\Real;\Real^d)$, we now introduce the notation
\begin{equation}
Q_{\vartheta} v = \langle T_{\vartheta} \psi_*, v \rangle_{L^2},
\qquad
P_{\vartheta} v =  [Q_{\vartheta} v ] T_{\vartheta} \Phi_*' .
\end{equation}
Applying a difference to \sref{eq:set:trnsv:def:eq},
we obtain
\begin{equation}
\begin{array}{lcl}
0 & = &
  Q_{\theta_{j+1}} v_{j+1}
  - Q_{\theta_j} v_j
\\[0.2cm]
& = &
  Q_{\theta_j} (v_{j+1} -v_j)
   + \big[
       Q_{\theta_{j+1} } - Q_{\theta_j}
     \big]
        v_{j+1} .
\\[0.2cm]
\end{array}
\end{equation}
Substituting the equation for $v$,
we arrive at
\begin{equation}
\label{eq:set:eq:theta:diff:pre}
\begin{array}{lcl}
 \big[
   Q_{\theta_j} - Q_{\theta_{j+1}}
 \big] v_{j+1}
& = &
  Q_{\theta_j} w_{j+1}
  - (\theta_{j+1} - \theta_j)
  - Q_{\theta_j}
      \mathcal{N}_{\Phi;\theta_j}(\theta_{j+1} - \theta_j)  .
\end{array}
\end{equation}
In order to
formulate this in a more compact fashion,
we write
\begin{equation}
\sigma_* = \max\{ \abs{\sigma_A}, \abs{\sigma_B} \}
\end{equation}
together with
\begin{equation}
\mathrm{ev}_l p =
\Big( p_{l -\sigma_* + 1} , \ldots , p_{l + \sigma_* - 1}, p_{l + \sigma_*}
\Big)
\end{equation}
for any sequence $p$. %$p: \Wholes \to \mathcal{H}$.
In addition, we introduce the shorthand notation
\begin{equation}
L^2 = L^2(\Real;\Real^d),
\qquad
H^1 = H^1(\Real;\Real^d),
\qquad
\mathbf{L}^2 = (L^2)^{2  \sigma_*},
\qquad
\mathbf{H}^1 = (H^1)^{2 \sigma_*}.
\end{equation}

With a slight abuse of notation,
we introduce the function $\mathcal{S}: \Real^{2 \sigma_*} \times \mathbf{H}^1 \times \mathbf{H}^1 \to \Real$
that acts as
\begin{equation}
\label{eq:set:defn:nl:s}
\begin{array}{lcl}
\mathcal{S}\big( \mathrm{ev}_l( \theta, v, w) \big)
 & = &
  \big[Q_{\theta_{l+1}} - Q_{\theta_l} \big]
      v_{l+1}
    - Q_{\theta_l}
      \mathcal{N}_{\Phi; \theta_l}
        (\theta_{l+1} - \theta_l) .
\end{array}
\end{equation}
This allows us to rewrite \sref{eq:set:eq:theta:diff:pre}
in the form
\begin{equation}
\label{eq:set:sys:for:theta:diff}
\begin{array}{lcl}
\theta_{l+1} - \theta_l
& = &
  Q_{\theta_l} w_{l+1}
  + \mathcal{S}\big( \mathrm{ev}_{l}( \theta, v, w) \big) .
\end{array}
\end{equation}
Substituting this back into \sref{eq:setup:pb:eq:ii},
% our coupled system
%for $(v,w)$,
we obtain
\begin{equation}
\label{eq:set:eq:for:v:diff:ii}
\begin{array}{lcl}
v_{l+1} - v_l
 & = &
   (1 -  P_{\theta_l} ) w_{l+1}
   -  \mathcal{S}\big( \mathrm{ev}_l(\theta, v, w) \big) T_{\theta_l} \Phi_*'
\\[0.2cm]
& & \qquad
  -  \mathcal{N}_{\Phi; \theta_l}(\theta_{l+1} - \theta_l)  .
\end{array}
\end{equation}

For any triplet $(\theta, v, w) : \Wholes \to \Real \times H^1 \times H^1$ ,
we now introduce the sequences
\begin{equation}
\mathcal{T}(\theta)[v,w]: \Wholes \to H^1 \times L^2,
\qquad
\mathcal{D}(\theta)[v,w]: \Wholes \to H^1 \times H^1,
\end{equation}
defined by
\begin{equation}
\begin{array}{lcl}
\mathcal{T}(\theta)[v, w]_l
 & = & (w_{l+1} , - \mathcal{L}^{(\theta_l)}_* v_l ) ,
\\[0.2cm]
\mathcal{D}(\theta)[ v,w]_l
& = & \big(v_{l+1} - v_l ,
  Df( T_{\theta_l} \tau \Phi_*) \overline{\tau} s^\diamond[w]_l
  \big).
\end{array}
\end{equation}
In addition, we introduce
the nonlinear function
\begin{equation}
\mathcal{R}: \Real^{2 \sigma_*} \times \mathbf{H}^1 \times \mathbf{H}^1
\to H^1 \times L^2
\end{equation}
that acts as
\begin{equation}
\label{eq:set:defn:nl:r}
\begin{array}{lcl}
\mathcal{R}\big( \mathrm{ev}_l (\theta, v, w) \big)
& = &
 ( c - c_*)(0, v_l')
  - \mathcal{S}\big( \mathrm{ev}_l(\theta, v, w) \big)
  \big( T_{\theta_l} \Phi_*' , 0 \big)
\\[0.2cm]
& & \qquad
  - \Big( \mathcal{N}_{\Phi; \theta_l}(\theta_{l+1} - \theta_l),
  \mathcal{N}_{f; \theta_l}
    \big( \tau v_l + \overline{\tau} s^\diamond[w]_l \big) \Big) .
\end{array}
\end{equation}
Finally, we introduce the operators
\begin{equation}
\label{eq:set:defns:q:i:ii}
Q^{(1)}_{\vartheta} [v, w] = Q_{\vartheta} v,
\qquad \qquad
Q^{(2)}_{\vartheta} [v, w] = Q_{\vartheta} w
\end{equation}
and the associated projections
\begin{equation}
\label{eq:set:defns:p:i:ii}
P^{(1)}_{\vartheta} [v, w]
   = (P_\vartheta v , 0) ,
\qquad
P^{(2)}_{\vartheta}(v, w) = (0 , P_{\vartheta} w) .
\end{equation}
This allows us to represent the full problem as
\begin{equation}
\label{eq:set:red:sys}
\begin{array}{lcl}
\theta_{l+1} - \theta_l
 & = &
    Q_{\theta_l} w_{l+1}
    + \mathcal{S}\big(\mathrm{ev}_l(\theta, v, w) \big) ,
\\[0.2cm]
\mathrm{pev}_l \mathcal{D}(\theta)[v, w]
 & = &
   (I - P^{(1)}_{\theta_l} )
  \mathrm{pev}_l \mathcal{T}(\theta)[v, w]
  + (c - c_*) (0, T_{\theta_l} \Phi_*')
  +     \mathcal{R}\big( \mathrm{ev}_l (\theta, v, w) \big),
\end{array}
\end{equation}
in which we have introduced
% the projection
%\begin{equation}
%\begin{array}{lcl}
%P^{(1)}_{\vartheta} (v, w)
% & = & (P_\vartheta v , 0) ,
%\end{array}
%\end{equation}
%together with
the pointwise evaluation operator
\begin{equation}
\mathrm{pev}_l H = H_l
\end{equation}
that acts on sequences $H$.

\begin{lem}
\label{lem:set:formulations:eqv}
Assume that (Hf), $(H\Phi)$ and $(HS1)$-$(HS3)$ are satisfied.
Pick a constant $c \in \Real$
together with three sequences
\begin{equation}
(\theta, v, w): \Wholes \to \Real \times H^1 \times H^1
\end{equation}
and consider the pair $(\Xi, \Upsilon)$
defined by
\sref{eq:set:ansatz:xi:ups}.
Then the differential-algebraic system \sref{eq:mr:sys:xi:upsilon}
and the identity
\begin{equation}
\label{eq:set:lem:eqv:q:proj:cnst}
Q_{\theta_l} v_l = Q_{\theta_0} v_0
\end{equation}
are both satisfied
for all $l \in \Wholes$,
if and only if
\sref{eq:set:red:sys}
holds for all $l \in \Wholes$.
\end{lem}
\begin{proof}
The computations above
shows that indeed \sref{eq:set:red:sys}
holds whenever
\sref{eq:mr:sys:xi:upsilon}  and
\sref{eq:set:lem:eqv:q:proj:cnst} are satisfied.
The converse implication
can be checked by using \sref{eq:set:eq:for:v:diff:ii}
to compute
\begin{equation}
\begin{array}{lcl}
Q_{l} [v_{l+1} - v_l]
& = & - \mathcal{S}\big(\mathrm{ev}_l(\theta, v, w) \big)
 - Q_{\theta_l} \mathcal{N}_{\Phi;\theta_l}
     (\theta_{l+1} - \theta_l)
\\[0.2cm]
& = & -[Q_{\theta_{l+1}} - Q_{\theta_l} ] v_{l+1},
\end{array}
\end{equation}
which gives \sref{eq:set:lem:eqv:q:proj:cnst}.
The identity
\sref{eq:mr:sys:xi:upsilon}
then follows readily.
\end{proof}

We now proceed to obtain estimates on the nonlinear terms. These are mostly standard
bounds for quadratic nonlinearities that will be used in {\S}\ref{sec:cm}
for the center manifold construction. Notice however that any dependencies on the phase $\theta$
always involve differences in $\theta$. % rather than the actual value for $\theta$.

\begin{lem}
\label{lem:set:bnd:n:f}
Assume that (Hf) and $(H\Phi)$ hold
and pick a sufficiently large constant $K > 0$.
Then for any $\tilde{v} \in H^1\big(\Real;(\Real^d)^5\big)$ with $\norm{\tilde{v}}_{H^1} \le 1$
%$\abs{\tilde{v}} \le 1$
we have the bound
\begin{equation}
\label{eq:set:bnd:n:f:i}
\norm{ \mathcal{N}_{f; \vartheta}(\tilde{v}) }_{L^2}
\le K \norm{\tilde{v}}_{H^1} \norm{\tilde{v}}_{L^2}.
\end{equation}
In addition, for any
pair $(\vartheta^A, \vartheta^B) \in \Real^2$
and any pair
\begin{equation}
(\tilde{v}^A, \tilde{v}^B) \in H^1\big(\Real;(\Real^d)^5\big) \times H^1\big(\Real;(\Real^d)^5\big)
\end{equation}
with $\norm{\tilde{v}^A}_{H^1} \le 1$
and $\norm{\tilde{v}^B}_{H^1} \le 1$,
we have the Lipschitz estimate
\begin{equation}
\label{eq:set:bnd:n:f:ii}
\norm{ \mathcal{N}_{f; \vartheta^A}(\tilde{v}^A)
- \mathcal{N}_{f; \vartheta^B}(\tilde{v}^B)  }_{L^2}
\le K \big[ \norm{\tilde{v}^A}_{H^1} + \norm{\tilde{v}^B}_{H^1}  + \abs{\vartheta^A -  \vartheta^B} \big]
     \Big[
       \norm{ \tilde{v}^A - \tilde{v}^B }_{L^2}
       +  \abs{\vartheta^A -  \vartheta^B}
     \Big] .
\end{equation}
\end{lem}
\begin{proof}
Upon writing
\begin{equation}
\mathcal{M}_{f; \tilde{\Phi}}(\tilde{v})
= f(\tau \tilde{\Phi} + \tilde{v} ) -
  Df( \tau \tilde{\Phi} ) \tilde{v}
- f(\tau \tilde{\Phi}  ),
\end{equation}
we readily see that
\begin{equation}
\mathcal{N}_{f;\vartheta}(\tilde{v})
= \mathcal{M}_{f; T_{\vartheta} \Phi_*}(\tilde{v}) .
\end{equation}
Since $f$ is at least $C^2$-smooth,
there exists $C_1 > 0$
so that the pointwise bound
\begin{equation}
\abs{ \mathcal{M}_{f; \tilde{\Phi}}(\tilde{v})(\xi) }
\le C_1 \abs{ \tilde{v}(\xi) }^2
\end{equation}
holds whenever
\begin{equation}
\label{eq:set:bnds:t:phi:v}
\norm{\tilde{\Phi}}_{H^1} \le \norm{\Phi_*}_{H^1} + 1,
\qquad
\norm{\tilde{v}}_{H^1} \le 1.
\end{equation}
This yields
\begin{equation}
\norm{ \mathcal{M}_{f; \tilde{\Phi}}(\tilde{v}) }_{L^2}
\le C_2 \norm{\tilde{v}}_{H^1} \norm{\tilde{v}}_{L^2}
\end{equation}
for some $C_2 > 0$, from which \sref{eq:set:bnd:n:f:i}
follows.

Upon writing
\begin{equation}
\Delta_{\mathcal{M}}
=
\mathcal{M}_{f;\tilde{\Phi}^A}(\tilde{v}^A)
  - \mathcal{M}_{f; \tilde{\Phi}^B}(\tilde{v}^B),
\end{equation}
a short computation shows that
\begin{equation}
\begin{array}{lcl}
\Delta_{\mathcal{M}}
 & = & \mathcal{M}_{f; \tilde{\Phi}^B + \tilde{v}^B}\big(\tilde{\Phi}^A + \tilde{v}^A - \tilde{\Phi}^B - \tilde{v}^B \big)
   - \mathcal{M}_{f; \tilde{\Phi}^B}\big(\tilde{\Phi}^A - \tilde{\Phi}^B)
\\[0.2cm]
& & \qquad
  + \big( Df(\tau \tilde{\Phi}^B + \tilde{v}^B) - Df(\tau \tilde{\Phi}^B) \big)[\tilde{\Phi}^A + \tilde{v}^A - \tilde{\Phi}^B - \tilde{v}^B ]
\\[0.2cm]
& & \qquad
 + \big( Df(\tau\tilde{\Phi}^B ) - Df(\tau\tilde{\Phi}^A) \big) \tilde{v}^A .
\end{array}
\end{equation}
Under the assumption that
\sref{eq:set:bnds:t:phi:v}
holds for both $(\tilde{\Phi}^A, v^A)$
and  $(\tilde{\Phi}^B, v^B)$,
we hence find
\begin{equation}
\begin{array}{lcl}
\norm{\Delta_{\mathcal{M}}}_{L^2}
\le
C_3 \big[
 \norm{\tilde{\Phi}^A - \tilde{\Phi}^B}_{H^1}
 + \norm{\tilde{v}^A }_{H^1} + \norm{\tilde{v}^B}_{H^1}
\big] \big[
   \norm{\tilde{v}^A - \tilde{v}^B}_{L^2}
   + \norm{\tilde{\Phi}^A - \tilde{\Phi}^B}_{L^2}
\big]
\end{array}
\end{equation}
for some $C_3 > 0$.
The bound \sref{eq:set:bnd:n:f:ii}
now follows from the fact that
\begin{equation}
\norm{ T_{\vartheta^B} \Phi_* - T_{\vartheta^A} \Phi_* }_{H^1}
\le C_4 \abs{\vartheta^B - \vartheta^A}
\end{equation}
for some $C_4 > 0$.
\end{proof}

\begin{lem}
\label{lem:set:bnd:n:phi}
Assume that (Hf), $(H\Phi)$ and (HS1) hold
and pick a sufficiently large constant $K > 0$.
Then for any pair $\vartheta \in \Real$
and $\theta_d \in \Real$
we have the bound
\begin{equation}
\label{eq:set:bnd:n:phi}
\norm{ \mathcal{N}_{\Phi_*; \vartheta}(\theta_{d}) }_{H^1}
\le K \abs{\theta_d}^2 .
\end{equation}
In addition, for any pair $(\vartheta^A, \vartheta^B) \in \Real^2$
and any pair $(\theta_d^A, \theta_d^B) \in \Real^2$
we have the bound
\begin{equation}
\label{eq:set:bnd:n:phi:delta}
\norm{ \mathcal{N}_{\Phi_*; \vartheta^A}(\theta^A_{d})
 - \mathcal{N}_{\Phi_*; \vartheta^B}(\theta^B_{d})
}_{H^1}
\le K \big[ \abs{\vartheta^A - \vartheta^B} + \abs{ \theta_d^A} + \abs{\theta_d^B} \big]
  \big[ \abs{\vartheta^A - \vartheta^B} +  \abs{ \theta^A_d - \theta^B_d } \big].
\end{equation}
\end{lem}
\begin{proof}
Assumptions (Hf), $(H\Phi)$ and (HS1)
imply that $\Phi_*'''$ is a continuous
function that decays exponentially,
from which \sref{eq:set:bnd:n:phi} follows.
Writing
\begin{equation}
\Delta_{\mathcal{N}}
=  \mathcal{N}_{\Phi_*; \vartheta^A}(\theta^A_{d})
 - \mathcal{N}_{\Phi_*; \vartheta^B}(\theta^B_{d}) ,
\end{equation}
one can compute
\begin{equation}
\begin{array}{lcl}
\Delta_{\mathcal{N}}
 & = & \mathcal{N}_{\Phi_*; \vartheta^B + \theta^B_d}
   \big(\vartheta^A + \theta_d^A - \vartheta^B - \theta_d^B \big)
   - \mathcal{N}_{\Phi_*; \vartheta^B}\big(\vartheta^A - \vartheta^B)
\\[0.2cm]
& & \qquad
  + \big( T_{\vartheta^B + \theta_d^B} \Phi_*' - T_{\vartheta^B} \Phi_*' \big)
     [\vartheta^A + \theta_d^A - \vartheta^B - \theta_d^B ]
\\[0.2cm]
& & \qquad
 + \big( T_{\vartheta^B}\Phi_*' - T_{\vartheta^A} \Phi_*'  \big) \theta_d^A .
\end{array}
\end{equation}
The inequality \sref{eq:set:bnd:n:phi:delta}
follows directly from this representation.
\end{proof}

\begin{cor}
\label{cor:set:bnds:full:r:s}
Assume that (Hf), $(H\Phi)$ and (HS1) hold
and pick a sufficiently large constant $K > 0$.
Then for any triplet of sequences
\begin{equation}
(\theta, v, w) : \Wholes \to \Real \times H^1 \times H^1
\end{equation}
that has
\begin{equation}
  \norm{v_l}_{H^1} + \norm{w_l}_{H^1} \le 1
\end{equation}
for all $l \in \Wholes$,
we have the bounds
\begin{equation}
\begin{array}{lcl}
\norm{ \mathcal{R}\big(\mathrm{ev}_l( \theta , v, w) \big) }_{H^1 \times L^2}
& \le &
K \big[ \abs{\theta_{l+1} - \theta_l} + \norm{\mathrm{ev}_l (v,w)}_{\mathbf{H}^1 \times \mathbf{H}^1}
   %  + \norm{\mathrm{ev}_l w}_{H^1}
   \big]
       \big[ \abs{\theta_{l+1} - \theta_l}  + \norm{\mathrm{ev}_l (v,w)}_{\mathbf{L}^2 \times \mathbf{L}^2}
        % + \norm{\mathrm{ev}_l w}_{L^2}
        \big]

 %K  \abs{\theta_{l+1} - \theta_l}^2
\\[0.2cm]
& & \qquad
 +   \abs{c - c_*} \norm{v_l}_{H^1}  ,
\\[0.2cm]
\abs{ \mathcal{S}\big( \mathrm{ev}_l (\theta, v, w) \big)}
& \le &
 K \abs{\theta_{l+1} - \theta_l}
   \big[ \norm{v_{l+1}}_{L^2}
          + \abs{\theta_{l+1} - \theta_l} \big]
\end{array}
\end{equation}
for all $l \in \Wholes$.
In addition, for any pair of triplets
\begin{equation}
(\theta^A, v^A, w^A) : \Wholes \to \Real \times H^1 \times H^1,
\qquad
(\theta^B, v^B, w^B) : \Wholes \to \Real \times H^1 \times H^1
\end{equation}
that has
\begin{equation}
  \norm{v^A_l}_{H^1} + \norm{w^A_l}_{H^1} \le 1,
  \qquad \qquad
  \norm{v^B_l}_{H^1} + \norm{w^B_l}_{H^1} \le 1
\end{equation}
for all $l \in \Wholes$,
the quantities
\begin{equation}
\begin{array}{lcl}
\Delta_{\mathcal{R}} & = & \mathcal{R}\big(\mathrm{ev}_l( \theta^A , v^A, w^A) \big)
- \mathcal{R}\big(\mathrm{ev}_l( \theta^B , v^B, w^B) \big) ,
\\[0.2cm]
\Delta_{\mathcal{S}} & = &
    \mathcal{S}\big( \mathrm{ev}_l (\theta^A, v^A, w^A) \big)
- \mathcal{S}\big( \mathrm{ev}_l (\theta^B, v^B, w^B) \big)
\end{array}
\end{equation}
satisfy the bounds
\begin{equation}
\begin{array}{lcl}
\norm{ \Delta_{\mathcal{R}}
 }_{H^1 \times L^2}
& \le &
 K \big[
   \abs{\theta^A_{l+1} - \theta^A_l}
   + \abs{\theta^B_{l+1} - \theta^B_l}
   + \norm{\mathrm{ev}_l (v^A, w^A)}_{\mathbf{H}^1 \times \mathbf{H}^1}
   + \norm{\mathrm{ev}_l (v^B, w^B)}_{\mathbf{H}^1 \times \mathbf{H}^1}
   %+  \norm{\mathrm{ev}_l w^A}_{H^1}
   %+ \norm{\mathrm{ev}_l w^B}_{H^1}
 \big]
\\[0.2cm]
& & \qquad \qquad
  \times \big[
   \norm{ \mathrm{ev}_l (v^A -  v^B) }_{\mathbf{L}^2}
   + \norm{ \mathrm{ev}_l (w^A -  w^B) }_{\mathbf{L}^2}
   + \abs{\mathrm{ev}_l (\theta^A - \theta^B) }
 \big]
\\[0.2cm]
& & \qquad
  + \abs{c - c_*} \norm{v^A_l - v^B_l}_{H^1},
\\[0.2cm]
\abs{ \Delta_{\mathcal{S}} }
& \le &
 \big[
   \abs{\theta^A_{l+1} - \theta^A_l}
   + \abs{\theta^B_{l+1} - \theta^B_l}
   + \norm{\mathrm{ev}_l v^A}_{\mathbf{L}^2}
   + \norm{\mathrm{ev}_l v^B}_{\mathbf{L}^2}
 \big]
\\[0.2cm]
& & \qquad \qquad
  \times \big[
   \norm{ \mathrm{ev}_l (v^A -  v^B) }_{\mathbf{L}^2}
   + \abs{\mathrm{ev}_l (\theta^A - \theta^B) }
 \big] .
\end{array}
\end{equation}
\end{cor}
\begin{proof}
This follows from Lemma's \ref{lem:set:bnd:n:f}-\ref{lem:set:bnd:n:phi}
upon inspecting the definitions of $\mathcal{S}$ and $\mathcal{R}$.
\end{proof}

We are now in a position to state our main center manifold result.
The fact that this manifold is two dimensional
is related to the observation
that the linear problem
\begin{equation}
\mathcal{D}(0)[v,w]
=  (I - P_0^{(1)}) \mathcal{T}(0)[v,w]
\end{equation}
has the constant solutions $(\Phi_*', 0)$
and $( [\partial_z \phi_z]_{z=0}, \Phi_*')$,
as we will see in \S\ref{sec:linop:cc}.

\begin{prop}[{see {\S}\ref{sec:cm}}]
\label{prp:set:cm}
Assume that (Hf), $(H\Phi)$, (HS1)-(HS3) and (HM)
all hold, recall the integer $r$ defined in (Hf)
%Consider the system \eqref{NEW_3_vect_10}.
and pick sufficiently small constants $\delta_c > 0$
and $\delta > 0$.
%For any small $\delta_\theta$
%there exist a constant $\delta>0$ and
Then there exists a function
\begin{equation}
h: \R^2 \times (c_* - \delta_c , c_* + \delta_c)
  \mapsto H^1 \times H^1
\end{equation}
together with functions
\begin{equation}
f_\theta, f_\kappa: \Real \times (c_* - \delta_c , c_* + \delta_c)  \to \Real
\end{equation}
such that the following properties are satisfied.
\begin{enumerate}
\item[(i)]{
 The function $h$ is $C^{r-1}$ smooth.
 %where $r$ is the regularity of $f$.
 In addition, we have the behaviour
\begin{equation}
h(\vartheta,\kappa ,c)=O(|\kappa|^2 +|c-c_*|)
\end{equation}
for $\kappa \rightarrow 0$ and $c\rightarrow c_*$,
uniformly for $\vartheta \in \Real$.
}
\item[(ii)]{
The functions $f_\theta$ and $f_\kappa$ are $C^{r-1}$-smooth.
In addition, we have the behaviour
%Also the expansions:
%\begin{equation}
%\begin{array}{lcl}
%f_\theta(\kappa, c) & = & O( \abs{c - c_*}) + \abs{\kappa}^2 )
%\\[0.2cm]
\begin{equation}
\begin{array}{lcl}
f_\kappa(\kappa, c) & = &
2 [\partial_z^2 \lambda_z]_{z=0}^{-1}
  \big[ c - c_*
    - \frac{1}{2} [\partial_z^2 d_\varphi]_{\varphi=0}  \kappa^2
  \big]
  + O\big( (c - c_*)^2 + (c - c_*) \kappa + \kappa^3 \big)
\end{array}
\end{equation}
as $\kappa \rightarrow 0$ and $c \rightarrow c_*$.
}

\item[(iii)]{
  For each small $\phi \ge 0$ there is a unique $\kappa_\phi \in [0, \delta]$
  for which $f_{\kappa}(\kappa_\phi, d_\phi ) = 0$.
  Similarly, whenever $-\phi \ge 0$ is small there is a unique $\kappa_\phi \in [-\delta, 0]$
  for which $f_{\kappa}(\kappa_\phi, d_\phi ) =0 $.
  In both cases we have $f_\theta( \kappa_\phi, d_\phi ) = \tan \phi$.
}

\item[(iv)]{
 Pick a $c \in (c_* - \delta ,c_* + \delta)$ and consider a triplet of sequences
 \begin{equation}
   (\theta, v, w) : \Wholes \to \Real \times H^1 \times H^1
 \end{equation}
 that satisfies \sref{eq:set:red:sys} and admits the bound
 \begin{equation}
 \norm{v_l}_{H^1} + \norm{w_l}_{H^1} \le \delta
 \end{equation}
 for all $l \in \Wholes$. Then upon writing
 \begin{equation}
\kappa_l = Q_{\theta_l} w_l = \left< T_{\theta_l} \psi_* ,w_l\right>_{L^2},
\end{equation}
  the identity
  \begin{equation}
  \label{eq:setup:cm:projs}
   (v,w)_l = \kappa_l T_{\theta_l} ( [\partial_z \phi_z]_{z=0} , \Phi_*')
   + h(\theta_l,\kappa_l, c)
\end{equation}
  together with the difference equation
  \begin{equation} \label{eq:setup:diff:eq:kappa:theta}
\begin{split}
\theta_{l+1} -\theta_l &= f_\theta(\kappa_l,c), \\
\kappa_{l+1} -\kappa_l &= f_\kappa(\kappa_l,c)
\end{split}
\end{equation}
  are both satisfied for all $l \in \Wholes$.
}

\item[(v)]{
 Pick a speed $c \in (c_* - \delta ,c_* + \delta)$ and pair of sequences
 \begin{equation}
   (\theta, \kappa) : \Wholes \to \Real \times \Real
 \end{equation}
 that satisfies \sref{eq:setup:diff:eq:kappa:theta} and admits the bound
 \begin{equation}
 \norm{\kappa_l} \le \delta
 \end{equation}
 for all $l \in \Wholes$. Then the triplet $(\theta, v, w)$
 obtained by applying the identity \sref{eq:setup:cm:projs}
 satisfies \sref{eq:set:red:sys}.
}
\end{enumerate}
\end{prop}

\begin{proof}[Proof of Theorem \ref{thm:mr:ex:corner}]
For explicitness, we assume that
$[\partial_\varphi^2 d_\varphi]_{\varphi=0} > 0$
and $[\partial_z^2 \lambda_z]_{z=0} > 0$.
Whenever $c - c_* > 0$ is sufficiently small,
the identity $[\partial_{\varphi} d_{\varphi} ]_{\varphi=0} = 0$
allows us to use a Taylor expansion to show that
there exist
$\varphi_- < 0 < \varphi_+$
for which $d_{\varphi_-} = d_{\varphi_+} = c$.
Noting that
\begin{equation}
\varphi_\pm = O( \sqrt{  c - c_*  } ),
\end{equation}
we use (iii) of Proposition \ref{prp:set:cm}
to define two quantities
$\kappa_\pm = \kappa_{\varphi_\pm}$
for which we have the identity
\begin{equation}
f_\kappa(\kappa_{\varphi_\pm} , c) = 0.
\end{equation}
In addition, possibly after further restricting
the size of $c-c_*$, we can ensure that the inequalities
\begin{equation}
f_{\kappa}( \kappa , c) > 0,
\qquad
\qquad
\abs{D_1 f_{\kappa}(\kappa , c) } < \frac{1}{2}
\end{equation}
hold for all $\kappa \in ( \kappa_- , \kappa_+)$.
As a consequence, for any such $\kappa$ we have
the bounds
\begin{equation}
\begin{array}{lcl}
\kappa - \kappa_-
& \le  &
f_{\kappa}(\kappa , c) - f_{\kappa}(\kappa_-, c)
\\[0.2cm]
& = &
  f_{\kappa}(\kappa , c)
\\[0.2cm]
& = &
f_{\kappa}(\kappa , c) -
  f_{\kappa}(\kappa_+ , c)
\\[0.2cm]
& \le &
   \kappa_+ - \kappa .
\end{array}
\end{equation}
In particular, for any $\kappa \in (\kappa_-, \kappa_+)$
one can apply the contraction mapping principle
to the fixed point problem
\begin{equation}
\tilde{\kappa}  = \kappa - f_{\kappa}(\tilde{\kappa} , c)
\end{equation}
and obtain a unique solution in $\tilde{\kappa} \in  (\kappa_-, \kappa)$.

For any choice of $\tilde{\kappa}_0 \in (\kappa_-, \kappa_+)$,
the problem
\begin{equation}
\kappa_{l+1} - \kappa_{l} = f_{\kappa}(\kappa, c) ,
\qquad
\qquad
\kappa_0 = \tilde{\kappa}_0
\end{equation}
can therefore be iterated backwards and forwards with respect to $l$
to yield a solution $\kappa: \Wholes \to (\kappa_-, \kappa_+)$.
This solution is strictly increasing and satisfies
the limits
\begin{equation}
\lim_{l \to \pm \infty} \kappa_l = \kappa_\pm .
\end{equation}
By applying the representation
\sref{eq:setup:cm:projs} one can now construct
the desired solution $(\theta, v)$.
\end{proof}

\section{Preliminaries}
\label{sec:prlm}

In this section we obtain a number
of preliminary results related
to the constant-coefficient
linear system
\begin{equation}
\label{eq:lin:repr:inhom:probl}
\mathcal{D}(0) [V] - \mathcal{T}(0)[V] = H .
\end{equation}
In particular, we study the Fourier symbol $\Delta(z)$
associated to this system and obtain
a representation formula for solutions that are allowed to
grow at a small exponential rate.

As a preparation, we
introduce the notation
\begin{equation}
\mathcal{L}_* = \mathcal{L}_0 = \mathcal{L}_*^{(0)},
\qquad
Q_* = Q_0 = \langle \psi_*, \cdot \rangle
\end{equation}
for the linearization of the travelling wave MFDE \sref{eq:mr:trv:wave:mfde}
around $\Phi_*$ and the corresponding projection onto the kernel element $\Phi_*'$.
In addition,
for any $z \in \mathbb{C}$ we define the vector
\begin{equation}
\begin{array}{lcl}
s^\diamond_z
& = &
\Big(
  - \sum_{j=0}^{ \sigma_B- 1} e^{ - z j} ,
  \sum_{j=1}^{\sigma_A} e^{z j},
  \sum_{j=1}^{\sigma_B} e^{z j },
  - \sum_{j=0}^{\sigma_A -1 } e^{- z j } ,
  0
\Big) \in \Complex^5,
\end{array}
\end{equation}
recalling the convention that
sums where the lower index is strictly larger
than the upper index are set to zero.
By construction, this allows us to write
\begin{equation}
s^\diamond [ e^{z \cdot} w] = e^{z \cdot} s^\diamond_z w
\end{equation}
for any $w \in H^1$.

%In particular,
Upon introducing
the linear operators
$\Delta(z): H^1 \times H^1 \to H^1 \times L^2$
that act as
\begin{equation}
\label{eq:prlm:def:delta:z}
\begin{array}{lcl}
\Delta(z) & = &
    \left(
  \begin{array}{cc}
     e^{z} - 1 & - e^{z} \\[0.2cm]
     \mathcal{L}_* &  Df (\tau \Phi_*) \overline{\tau} s^\diamond_z \\[0.2cm]
  \end{array}
  \right) ,
\end{array}
\end{equation}
we hence have the identity
\begin{equation}
\mathcal{D}(0)[ e^{z \cdot} V ] - \mathcal{T}(0)[ e^{z \cdot} V]
= e^{z \cdot} \Delta(z) V
\end{equation}
for any $V \in H^1 \times H^1$.
Our first main result shows that these operators
are invertible along vertical lines that are close to
the imaginary axis.

\begin{prop}
\label{prp:hom:delta:inv}
Assume that (Hf), $(H\Phi)$, (HS1)-(HS3) and (HM) are all satisfied
and pick a sufficiently small $\eta_{\max} > 0$. Then there exists a constant $K > 0$
so that for every $\eta$ with $0 < \abs{\eta} < \eta_{\max}$
and every $\omega \in [0,  2 \pi]$, the operator
\begin{equation}
\Delta(\eta + i \omega): H^1 \times H^1 \to H^1 \times L^2
\end{equation}
is invertible and satisfies the bound
\begin{equation}
\norm{ \Delta(\eta + i \omega)^{-1} }_{
  \mathcal{L}\big( H^1 \times L^2 ;  H^1 \times H^1  \big)
} \le K \eta^{-2}.
\end{equation}
\end{prop}

In order to gain insight regarding solutions to the homogeneous linear system
\begin{equation}
\label{eq:lin:repr:prob:hom}
\mathcal{D}(0)[V] - \mathcal{T}(0)[V] = 0,
\end{equation}
we briefly discuss the maximal Jordan chain associated to
$\Delta(z)$ at $z = 0$. In particular, we set out
to construct an analytic function
$z \mapsto \mathcal{J}(z) \in H^1 \times H^1$
with $\mathcal{J}(0) \neq 0$ so that
\begin{equation}
\label{eq:prlm:jordan:chain:defn}
\Delta(z) \mathcal{J}(z) = O ( z^m )
\end{equation}
for the largest possible value of $m$.

As a preparation, we note that
\begin{equation}
\begin{array}{lcl}
s^\diamond_z
& = &
\frac{e^z}{1 - e^z}
\Big(  e^{ - \sigma_B  z } - 1 ,
  e^{ \sigma_A z} - 1 , e^{\sigma_B z} - 1, e^{- \sigma_A z } - 1
  , 0 \Big)
\end{array}
\end{equation}
whenever $z \neq 0$.
Recalling the definition \sref{eq:mr:def:lz},
we hence see that
\begin{equation}
\label{eq:id:for:mcl:z:i}
\mathcal{L}_z = \mathcal{L}_*  + e^{-z} (e^z - 1) Df(\tau \Phi_*) \overline{\tau} s^\diamond_z .
\end{equation}
Upon introducing the notation
\begin{equation}
\label{eq:prlm:def:A:k}
A_k = \big[ \partial^k_z \mathcal{L}_z \big]_{z = 0}
\end{equation}
for any integer $k \ge 1$,
we may differentiate \sref{eq:id:for:mcl:z:i}
to find
\begin{equation}
\label{eq:prlm:ids:ai:all}
\begin{array}{lcl}
Df(\tau \Phi_*) \overline{\tau} s^\diamond_0
& = &
  A_1,
\\[0.2cm]
Df(\tau \Phi_*) \overline{\tau} [\partial_z s^\diamond_z]_{z=0}
& = &
 \frac{1}{2} (A_2 + A_1).
\\[0.2cm]
\end{array}
\end{equation}

In particular, we may write
\begin{equation}
\Delta(0) =
\left(
  \begin{array}{cc}
     0 & - 1 \\[0.2cm]
     \mathcal{L}_* & A_1  \\[0.2cm]
  \end{array}
  \right),
\qquad
\Delta'(0) =
\left(
  \begin{array}{cc}
     1 & - 1 \\[0.2cm]
     0 & \frac{1}{2}(A_1 + A_2) \\[0.2cm]
  \end{array}
  \right).
\end{equation}
Using (HS3) we immediately see
\begin{equation}
\mathrm{Ker} \big(\Delta(0)\big) = \mathrm{span} \{ (\Phi_*' , 0)^T  \},
\end{equation}
which allows us to pick $\mathcal{J}(0) = (\Phi_*', 0)$.
This chain can be extended by exploiting
the following preliminary identities.

\begin{lem}
\label{lem:prlms:ids:for:lambda:derivs}
Assume that (Hf), $(H\Phi)$ and (HS1)-(HS3) are satisfied.
Then we have the identities
\begin{equation}
\label{eq:prlm:diffs:of:l}
\begin{array}{lcl}
\mathcal{L}_* [\partial_z \phi_z]_{z=0} ]
  &= &  - A_1 \Phi_*' + [\partial_z \lambda_z]_{z=0} \Phi_*' ,
  \\[0.2cm]
\mathcal{L}_* [\partial_z^2 \phi_z]_{z=0} ]
  &= &  - A_2 \Phi_*' - 2 A_1 [\partial_z \phi_z]_{z=0}
  + [\partial_z^2 \lambda_z]_{z=0} \Phi_*'
  + 2 [\partial_z \lambda_z]_{z=0} [\partial_z \phi_z]_{z=0},
\\[0.2cm]
\mathcal{L}_* [\partial_z^3 \phi_z]_{z=0} ]
 & = &  - (A_3 \Phi_*' + 3 A_2 [\partial_z \phi_z]_{z=0} + 3 A_1 [\partial_z^2 \phi_z]_{z=0} )
\\[0.2cm]
& & \qquad
   + [\partial_z^3 \lambda_z]_{z=0} \Phi_*' + 3 [\partial_z^2 \lambda_z]_{z=0} [\partial_z \phi_z]_{z=0}
   + 3 [\partial_z \lambda_z]_{z=0} [\partial_z^2 \phi_z]_{z=0} ,
\end{array}
\end{equation}
together with
\begin{equation}
\label{eq:prlm:ids:for:lambda:derivs}
\begin{array}{lcl}
%\lambda'(0)
[\partial_z \lambda_z]_{z=0} & = &
  \langle \psi_*, A_1 \Phi_*' \rangle   ,
\\[0.2cm]
[\partial_z^2 \lambda_z]_{z=0} & = &
  \langle \psi_* , A_2 \Phi_*' + 2 A_1 [\partial_z \phi_z]_{z=0} \rangle ,
\\[0.2cm]
[\partial_z^3 \lambda_z]_{z=0} & = &
  \langle \psi_*, A_3 \Phi_*' + 3 A_2 [\partial_z \phi_z]_{z=0}
    + 3 A_1 [\partial_z^2 \phi_z]_{z=0} \rangle .
\end{array}
\end{equation}
\end{lem}
\begin{proof}
Differentiating the definition
$\mathcal{L}_z \phi_z = \lambda_z \phi_z$,
we obtain the identities
\begin{equation}
\begin{array}{lcl}
[\partial_z \mathcal{L}_z] \phi_z
  + \mathcal{L}_z [\partial_{z} \phi_z]
& = &
[\partial_z \lambda_z] \phi_z
    + \lambda_z [\partial_z \phi_z ] ,
\\[0.2cm]
[\partial_z^2 \mathcal{L}_z] \phi_z +
 2 [\partial_z \mathcal{L}_z ] \partial_z \phi_z
 + \mathcal{L}_z [\partial_{z}^2 \phi_z]
 & =&
  [\partial_z^2 \lambda_z] \phi_z +
  2 [\partial_z \lambda_z] \partial_z \phi_z
  + \lambda_z [\partial_z^2 \phi_z ]  ,
\\[0.2cm]
[\partial_z^3 \mathcal{L}_z] \phi_z +
 3 [\partial_z^2 \mathcal{L}_z] \partial_z \phi_z
 + 3 [\partial_z \mathcal{L}_z] \partial_{z}^2 \phi_z
 + \mathcal{L}_z [\partial_{z}^3 \phi_z]
 & = &
  [\partial_z^3 \lambda_z] \phi_z +
  3 [\partial_z^2 \lambda_z] \partial_z \phi_z
  + 3[\partial_z \lambda_z] \partial_{z}^2 \phi_z
\\[0.2cm]
& & \qquad
  + \lambda_z [\partial_{z}^3 \phi_z].
\end{array}
\end{equation}
Evaluating these expressions
at $z = 0$ we find \sref{eq:prlm:diffs:of:l}.
The deriatives \sref{eq:prlm:ids:for:lambda:derivs}
can then be obtained by
recalling the normalization
$\langle \psi_*, \phi_z \rangle = 1$
and using the fact
that $\langle \psi_*, \mathcal{L}_* y \rangle = 0$
for all $y \in H^1$.
\end{proof}

Indeed, we now write
\begin{equation}
\mathcal{J}(z)
= (\Phi_*' , 0)
+ z ( [\partial_z \phi_z]_{z=0} , \Phi_*')
+ z^2 \big( v(z), w(z) \big)
\end{equation}
for a pair of analytic functions
$z \mapsto \big( v(z), w(z) \big) \in H^1 \times H^1$.
Using the identity
\begin{equation}
\Delta''(0) =
\left(
  \begin{array}{cc}
     1 & - 1 \\[0.2cm]
     0 & Df(\tau \Phi_*) \overline{\tau}
        [\partial_z^2 s^\diamond_z]_{z=0}
     \\[0.2cm]
  \end{array}
  \right) ,
\end{equation}
we may exploit \sref{eq:prlm:diffs:of:l}
to compute
\begin{equation}
\label{eq:prlm:comp:jord:chain}
\begin{array}{lcl}
\Delta(z) \mathcal{J}(z)
& = &
  \Delta(0) (\Phi_*', 0)^T
  + z \Big( \Delta'(0) (\Phi_*', 0)^T + \Delta(0) ([\partial_z \phi_z]_{z=0}, \Phi_*')^T
  \Big)
\\[0.2cm]
& & \qquad
  + z^2 \Big( \frac{1}{2}\Delta''(0) (\Phi_*', 0)^T
    + \Delta'(0)([\partial_z \phi_z]_{z=0}, \Phi_*')^T
    + \Delta(0) V(0) \Big)
+ O(z^3)
\\[0.2cm]
& = &
  z^2 \Big(
     \big(   [\partial_z \phi_z]_{z=0} -\frac{1}{2} \Phi_*',
    \frac{1}{2}(A_1 + A_2) \Phi_*' \big)^T
    + \Delta(0) \big(v(0), w(0)\big)^T
  \Big) + O (z^3) .
\end{array}
\end{equation}
In particular, we have achieved $m = 2$
in \sref{eq:prlm:jordan:chain:defn}.
This corresponds with the
presence of two solutions
\begin{equation}
\label{eq:lin:repr:def:two:hom:sols}
[V_{\mathrm{hom}}^A]_l =  (\Phi_*', 0),
\qquad
\qquad
[V_{\mathrm{hom}}^B]_l = ( l \Phi_*', 0) +  ( [\partial_z \phi_z]_{z=0} , \Phi_*')
\end{equation}
to the linear homogeneous problem
\sref{eq:lin:repr:prob:hom}.
However, it is not possible
to achieve $m = 3$. Indeed,
setting the $O(z^2)$ term in \sref{eq:prlm:comp:jord:chain}
to zero, we obtain
\begin{equation}
w(0) = \frac{1}{2} \Phi_*' - [\partial_z \phi_z]_{z=0}
\end{equation}
and hence
\begin{equation}
\mathcal{L}_* v(0)
= \frac{1}{2} A_2 \Phi_*' + A_1 [\partial_z \phi_z]_{z=0}.
\end{equation}
Taking the inner product with $\psi_*$, we
may use (HM) to obtain the contradiction
\begin{equation}
0 = \langle \psi_*, \frac{1}{2} A_2 \Phi_*' + A_1 [\partial_z \phi_z]_{z=0} \rangle
= \frac{1}{2} [\partial_z^2 \lambda_z]_{z=0} \neq 0 .
\end{equation}

Our second main result confirms that there are no other
linearly independent solutions to \sref{eq:lin:repr:prob:hom}
that are bounded by $e^{\eta_{\max} \abs{l} }$.
In addition, it provides a representation formula
for solutions to the inhomogeneous system \sref{eq:lin:repr:inhom:probl}
that share such an exponential bound.

In order to formulate this conveniently,
we introduce the family of sequence spaces
\begin{equation}
BX_{\mu, \nu} (\mathcal{H}) :=
 \left\{V :\Z \mapsto \mathcal{H} :
 \sup_{l<0} e^{- \mu l} \norm{V_l}_{\mathcal{H}}
    +\sup_{l\geq 0} e^{-\nu l} \norm{V_l}_{\mathcal{H}}  <\infty   \right\}
\end{equation}
for any $\mu,\nu \in \R$ and any Hilbert space $\mathcal{H}$.
In addition, for any $H \in BX_{\mu, \nu}$
we introduce the forward discrete Laplace transform
\begin{equation}
\mathcal{Z}_+[H](z) = \sum_{n=0}^\infty e^{-zn}H_n,
\qquad
\qquad \Re z > \nu
\end{equation}
together with the backward discrete Laplace transform
\begin{equation}
\mathcal{Z}_-[H](z) = \sum_{n=1}^\infty e^{zn}H_{-n},
\qquad \qquad
\Re z < \mu .
\end{equation}

\begin{prop}
\label{prp:prlm:repr:frm}
Assume that (Hf), $(H\Phi)$, (HS1)-(HS3) and (HM) are all satisfied.
Fix four constants $\gamma_- < \eta_- < \eta_+ < \gamma_+$
for which $0 < \abs{\eta_\pm} < \eta_{\max}$
and $0 < \abs{\gamma_\pm} < \eta_{\max}$ all hold,
together with
\begin{equation}
\mathrm{sign} (\gamma_-) = \mathrm{sign} ( \eta_-)
,\qquad
\mathrm{sign} (\gamma_+) = \mathrm{sign} ( \eta_+) .
\end{equation}
Consider any
$V \in BX_{\eta_-, \eta_+}(H^1 \times H^1)$
and write
\begin{equation}
H = \mathcal{D}(0) [ V] - \mathcal{T}(0)[V].
\end{equation}
Then we have the representation formula
\begin{equation}
\label{eq:prlm:main:repr:formula}
\begin{array}{lcl}
V_l & = & \frac{1}{2 \pi i} \int_{\gamma_+ - \pi i}^{\gamma_+ + \pi i}
  e^{z l} \Delta(z)^{-1} \mathcal{Z}_+[H](z) \, d z
   + \frac{1}{2 \pi i} \int_{\gamma_- - \pi i }^{\gamma_- + \pi i}  e^{z l}
       \Delta(z)^{-1} \mathcal{Z}_-[H](z) \, dz
\\[0.2cm]
& & \qquad
+ [V_{\mathrm{hom}}^A]_l
  \tilde{P}_A \big[ \mathrm{ev}_0 V \big]
  + [V_{\mathrm{hom}}^B]_l \big[\tilde{P}_B \mathrm{ev}_0 V \big]
\end{array}
\end{equation}
for some pair of bounded linear maps
\begin{equation}
(\tilde{P}_A, \tilde{P}_B):
\mathbf{H}^1 \times \mathbf{H}^1 \to \Real \times \Real
\end{equation}
that satisfy
\begin{equation}
\tilde{P}_{A} = \tilde{P}_B = 0
\end{equation}
when
$\mathrm{sign}(\eta_-) = \mathrm{sign}(\eta_+)$
and
\begin{equation}
\tilde{P}_{A} \big[ \mathrm{ev}_0 V_{\mathrm{hom}}^A \big]
= \tilde{P}_{B} \big[ \mathrm{ev}_0 V_{\mathrm{hom}}^B \big] = 1,
\qquad \qquad
\tilde{P}_{A} \big[ \mathrm{ev}_0 V_{\mathrm{hom}}^B \big]
= \tilde{P}_{B} \big[ \mathrm{ev}_0 V_{\mathrm{hom}}^A \big] = 0
\end{equation}
otherwise.
\end{prop}

Whenever $\mathcal{L}_z: H^1 \to L^2$ is invertible,
a short calculation shows that
the same is true for $\Delta(z)$
with
\begin{equation}
\label{eq:prlm:def:delta:z:inv}
\begin{array}{lcl}
\Delta(z)^{-1}
& = &
  \left(
  \begin{array}{cc}
     \mathcal{L}_z^{-1} Df(\tau \Phi_*) \overline{\tau} s^\diamond_z  e^{-z} &
      \mathcal{L}_z^{-1}  \\[0.2cm]
     - e^{-z} + e^{-z} (e^{z}-1) \mathcal{L}_z^{-1} Df(\tau \Phi_*) \overline{\tau} s^\diamond_z e^{-z}
       &   e^{-z} (e^{z} - 1 )   \mathcal{L}_z^{-1}  \\[0.2cm]
  \end{array}
  \right)
   .
\end{array}
\end{equation}
It is hence crucial to understand the behaviour of
$\mathcal{L}_z^{-1}$ for small $\abs{z} > 0$,
which we set out to do by exploiting
the Fredholm properties of $\mathcal{L}_*$.

As a preparation, %To this end,
we introduce the notation
\begin{equation}
\mathcal{L}^{\mathrm{qinv}}_* f = v
\end{equation}
for the unique $v \in H^1$
that has $\langle \psi_*, v \rangle = 0$
and satisfies the problem
\begin{equation}
\mathcal{L}_* v = f - \langle \psi_*, f \rangle \Phi_*' .
\end{equation}
This allows us to rephrase
the identities \sref{eq:prlm:diffs:of:l}
in a more explicit form.

\begin{cor}
\label{cor:prlms:ids:for:phi:derivs}
Assume that (Hf), $(H\Phi)$ and (HS1)-(HS3) are satisfied.
Then we have the identities
\begin{equation}
\begin{array}{lcl}
[\partial_z \phi_z]_{z=0} & = &
  - \mathcal{L}^{\mathrm{qinv}} A_1 \Phi_*' ,
\\[0.2cm]
[\partial_z^2 \phi_z]_{z=0} & = &
  - \mathcal{L}^{\mathrm{qinv}} \Big[ A_2 \Phi_*' + 2 A_1 [\partial_z \phi_z]_{z=0}
    - 2 [\partial_z \lambda_z]_{z=0} [\partial_z \phi_z]_{z=0}  \Big] .
\\[0.2cm]
\end{array}
\end{equation}
\end{cor}
\begin{proof}
These expressions
follow from
\sref{eq:prlm:diffs:of:l},
noting that $\mathcal{L}^{\mathrm{qinv}} \Phi_*' = 0$.
%recalling the normalization
%$\langle \psi_*, \phi_z \rangle  = 1$
%and using the fact that
%$\langle \psi_*, \mathcal{L}_* y \rangle = 0$ for all $y \in H^1$.
\end{proof}

At this point, it is natural
to briefly turn our attention to the
angular dependence of the waves $(c_\varphi, \Phi_\varphi)$,
which can be analyzed using techniques that are similar to those used above.
Unfortunately, the expressions for the second derivatives
are somewhat more involved. In order to accomodate this,
we introduce the notation
\begin{equation}
B_1 p = Df(\tau \Phi_*)(-\sigma_A , -\sigma_B, \sigma_A, \sigma_B, 0)
  [\overline{\tau} p ]
\end{equation}
for any $p \in H^1$.

\begin{lem}
\label{lem:prlms:ids:for:Phi:phi:derivs}
Assume that (Hf), $(H\Phi)$ and (HS1)-(HS3) are satisfied.
Then the statements in Lemma \ref{lem:mr:angl:dep}
hold true. In addition,
we have the identities
\begin{equation}
\label{eq:prlm:diffs:of:Phi:phi}
\begin{array}{lcl}
\mathcal{L}_* [\partial_{\varphi} \Phi_{\varphi}]_{\varphi=0}
  &= &  - A_1 \Phi_*' + [\partial_{\varphi} c_{\varphi}]_{\varphi=0} \Phi_*' ,
  \\[0.2cm]
\mathcal{L}_* [\partial_{\varphi}^2 \Phi_{\varphi}]_{\varphi=0}
  &= &
  [\partial_\varphi^2 c_\varphi]_{\varphi = 0} \Phi_*'
  + 2 [\partial_\varphi c_\varphi]_{\varphi = 0} [\partial_\varphi \Phi']_{\varphi =0 }
\\[0.2cm]
& & \qquad
- D^2 f(\tau \Phi_*) \Big[ s^\diamond_0 \tau \,  \Phi_*'
  +  \tau [\partial_\varphi \Phi_{\varphi}]_{\varphi =0} ,
  s^\diamond_0 \tau \,  \Phi_*'
  +  \tau [\partial_\varphi \Phi_{\varphi}]_{\varphi =0}
   \Big]
\\[0.2cm]
& & \qquad
- A_2 \Phi_*''
- B_1 \Phi_*'
- 2 A_1 [\partial_\varphi \Phi_\varphi']_{\varphi =0 } ,
\\[0.2cm]
\end{array}
\end{equation}
together with
\begin{equation}
\label{eq:prlm:ids:for:Phi:phi:derivs}
\begin{array}{lcl}
%\lambda'(0)
[\partial_{\varphi} c_{\varphi}]_{\varphi=0} & = &
  \langle \psi_*, A_1 \Phi_*' \rangle ,
\\[0.2cm]
[\partial_{\varphi}^2 c_{\varphi}]_{\varphi = 0} & = &
 \langle \psi_*,
   D^2 f(\tau \Phi_*)
   \big[   s^\diamond_0 \tau \, \Phi_*'
  +  \tau [\partial_\varphi \Phi_{\varphi}]_{\varphi =0} ,
  s^\diamond_0 \tau \,  \Phi_*
  +  \tau [\partial_\varphi \Phi_{\varphi}]_{\varphi =0} \big]
  \rangle
\\[0.2cm]
& & \qquad + \langle \psi_*, A_2 \Phi_*'' \rangle
+ \langle \psi_*, B_1 \Phi_*' \rangle
+ 2 \langle \psi_* , A_1 [\partial_\varphi \Phi_\varphi']_{\varphi = 0} \rangle
\\[0.2cm]
& & \qquad
- 2 [\partial_\varphi c_\varphi]_{\varphi = 0}
    \langle \psi_*, [\partial_\varphi \Phi_\varphi']_{\varphi = 0} \rangle .
\\[0.2cm]
\end{array}
\end{equation}
\end{lem}
\begin{proof}
Items (i)-(iii) of Lemma \ref{lem:mr:angl:dep} can be established as in
\cite[Prop 3.7]{HJHOBST2D}.
Differentiating the travelling wave MFDE
\sref{eq:mr:trv:wave:diff:angle}, we find
\begin{equation}
\label{eq:prlm:deriv:phi:id:i}
\begin{array}{lcl}
[\partial_\varphi c_\varphi] \Phi_\varphi'
& = &
-  c_\varphi [\partial_\varphi \Phi_\varphi']
+ Df ( \tau_\varphi \Phi_\varphi) \big[ \tau_\varphi [\partial_\varphi \Phi_\varphi] + [\partial_\varphi \tau_\varphi] \Phi_\varphi \big]  ,
\end{array}
\end{equation}
together with
\begin{equation}
\label{eq:prlm:deriv:phi:id:ii}
\begin{array}{lcl}
[\partial_\varphi^2 c_\varphi] \Phi'_\varphi
& = &
- c_{\varphi} [\partial_\varphi^2 \Phi'_\varphi] -  2 [\partial_\varphi c_\varphi]
   [ \partial_\varphi \Phi'_\varphi]
\\[0.2cm]
& & \qquad
+ Df(\tau_\varphi \Phi_\varphi)
     \big[
        \tau_\varphi [\partial_\varphi^2 \Phi_\varphi]
         + 2 [\partial_\varphi \tau_\varphi] [\partial_\varphi \Phi_\varphi]
     + [\partial_\varphi^2 \tau_\varphi ]\Phi_{\varphi}
     \big]
\\[0.2cm]
& & \qquad
+
D^2 f(\tau_\varphi \Phi_\varphi)
\big[ [\partial_\varphi \tau_\varphi ]\Phi_{\varphi} + \tau_\varphi [\partial_\varphi \Phi_\varphi],
  [\partial_\varphi \tau_\varphi ]\Phi_{\varphi} + \tau_\varphi [\partial_\varphi \Phi_\varphi]
\big] .
\\[0.2cm]
\end{array}
\end{equation}

For any $p \in H^1$,
we can compute
\begin{equation}
\begin{array}{lcl}
[\partial_\varphi \tau_\varphi p]_{\varphi = 0}
& = & \sigma_* \overline{\tau} \big( -\sin\phi_*, \cos \phi_*, \sin \phi_*, - \cos \phi_* \big)  p'
\\[0.2cm]
& = &  \overline{\tau}\big( - \sigma_B , \sigma_A, \sigma_B, -\sigma_A \big)
   p'
\\[0.2cm]
& = & \overline{\tau} s^\diamond_0 \, p' ,
\end{array}
\end{equation}
together with
\begin{equation}
\begin{array}{lcl}
[\partial_\varphi^2 \tau_\varphi p]_{\varphi = 0}
& = & \sigma_*^2 \overline{\tau} \big( \sin(\phi_*)^2, \cos(\phi_*)^2,  \sin(\phi_*)^2, \cos(\phi_*)^2 \big)
    p''
\\[0.2cm]
& &
\qquad
  + \sigma_* \overline{\tau } \big( - \cos(\phi_*), - \sin(\phi_*), \cos(\phi_*), \sin(\phi_*) \big)
         p'
\\[0.2cm]
& = &   \overline{\tau}\big( \sigma_B^2, \sigma_A^2, \sigma_B^2, \sigma_A^2 \big)
   p''
   + \overline{\tau}\big( - \sigma_A, - \sigma_B, \sigma_A, \sigma_B, 0\big)
         p'
\\[0.2cm]
& = &
  \overline{\tau}(2 [\partial_z s^\diamond_z]_{z=0}
       -s^\diamond_0)  p''
  + \overline{\tau}\big( - \sigma_A, - \sigma_B, \sigma_A, \sigma_B, 0 \big)
     p' .
\end{array}
\end{equation}
Based on \sref{eq:prlm:ids:ai:all},
we may hence write
\begin{equation}
\begin{array}{lcl}
Df(\tau \Phi_*)
[\partial_\varphi \tau_\varphi p]_{\varphi = 0}
& = &
  A_1 p' ,
\\[0.2cm]
Df(\tau \Phi_*)
[\partial_\varphi^2 \tau_\varphi p]_{\varphi = 0}
& = & A_2 p'' + B_1 p' .
\end{array}
\end{equation}

Using these identities to
evaluate the expressions
\sref{eq:prlm:deriv:phi:id:i}-\sref{eq:prlm:deriv:phi:id:ii}
at $\varphi = 0$
readily leads to the identities
\sref{eq:prlm:diffs:of:Phi:phi}.
The deriatives \sref{eq:prlm:ids:for:Phi:phi:derivs}
can then be obtained by
using the fact
that $\langle \psi_*, \mathcal{L}_* y \rangle = 0$
for all $y \in H^1$. Item (iv) of
Lemma \ref{lem:mr:angl:dep} follows directly by comparing
the first identities in
\sref{eq:prlm:diffs:of:l}
and \sref{eq:prlm:ids:for:lambda:derivs}
with those in
\sref{eq:prlm:diffs:of:Phi:phi}
and \sref{eq:prlm:ids:for:Phi:phi:derivs}.
\end{proof}

We now construct a preliminary inverse for
$\mathcal{L}_z$ that behaves as $z^{-2}$
as $z \mapsto 0$.
As a preparation,
we implicitly define the remainder
expressions $R_{\mathcal{L}; i}$
by writing
\begin{equation}
\begin{array}{lcl}
\mathcal{L}_z
  &= &
     \mathcal{L}_*  + z R_{\mathcal{L};1}(z)
\\[0.2cm]
& = &
   \mathcal{L}_*  + z A_1 + z^2 R_{\mathcal{L};2}(z)
\\[0.2cm]
& = &
  \mathcal{L}_*  + z A_1 + \frac{1}{2} z^2 A_2 + z^3 R_{\mathcal{L};3}(z) .
\end{array}
\end{equation}

\begin{lem}
\label{lem:prlm:inverse:b2}
Assume that (Hf), $(H\Phi)$, (HS1)-(HS3) and (HM) are satisfied.
Pick a sufficiently large $K > 0 $ together with
a sufficiently small $\delta_z > 0$.
Then there exists an analytic map
\begin{equation}
\{ z \in \mathbb{C} : \abs{z} < \delta_z \}
\ni z \mapsto B_2(z) \in \mathcal{L}(L^2; H^1)
\end{equation}
so that $v = B_2(z) h$ is the
unique $v \in H^1$ that satisfies
  \begin{equation}
    \label{eq:prlm:eq:for:b2:to:solve}
    \mathcal{L}_z v = z^2 h
  \end{equation}
whenever $h \in L^2$ and $0 < \abs{z} < \delta_z$.
\end{lem}
\begin{proof}
We set out to seek a solution a solution to
\sref{eq:prlm:eq:for:b2:to:solve}
%$\mathcal{L}_z v = z^2 h$
of the form
\begin{equation}
v = \kappa
 \big[
   \Phi_*' + z [\partial_z \phi_z]_{z=0}
 \big] + z^2 w
\end{equation}
for some $\kappa \in \Real$ and $w \in H^1$ that satisfies
$\langle \psi_*, w \rangle = 0$.
Writing
\begin{equation}
S(z) = R_{\mathcal{L};3}(z) \Phi_*'
+ R_{\mathcal{L};2}(z) [\partial_z \phi_z]_{z=0} ,
\end{equation}
we may use \sref{eq:prlm:diffs:of:l} to
compute
\begin{equation}
\begin{array}{lcl}
z^{-2} \mathcal{L}_z v
& = & \kappa
\big(\frac{1}{2} A_2 \Phi_*' + A_1 [\partial_z \phi_z]_{z=0}
  + z S(z) \big)
%\\[0.2cm]
%& & \qquad
 + \mathcal{L}_* w + z R_{\mathcal{L};1}(z) w .
\end{array}
\end{equation}
The identity \sref{eq:prlm:eq:for:b2:to:solve}
is hence equivalent to the system
\begin{equation}
\label{eq:prlm:b2:first:line}
\begin{array}{lcl}
Q_* h  & = & z Q_*  R_{\mathcal{L};1}(z) w + \kappa
 \Big( \frac{1}{2}[\partial_z^2 \lambda_z]_{z=0}
   + z Q_* S(z)  \Big) ,
\\[0.2cm]
w & = & \mathcal{L}^{\mathrm{qinv}}
\big[ h - z R_{\mathcal{L};1}(z) w \big]
+ \frac{1}{2} \kappa [\partial_z^2 \phi_z]_{z=0}
 -  z \kappa \mathcal{L}^{\mathrm{qinv}} S(z) .
\end{array}
\end{equation}
Whenever $\abs{z}$ is sufficiently small,
we may use the quantity
\begin{equation}
%\begin{array}{lcl}
\nu(z) =
\Big[ \frac{1}{2} [\partial_z^2 \lambda_z]_{z=0} + z Q_* S(z)
\Big]^{-1}
\end{equation}
to rewrite the
first line of \sref{eq:prlm:b2:first:line}
in the form
\begin{equation}
\kappa = \nu(z) Q_* h  - z \nu(z) Q_* R_{\mathcal{L};1}(z) w.
\end{equation}
Substituting this into the second line of
\sref{eq:prlm:b2:first:line},
we find
\begin{equation}
\begin{array}{lcl}
w + z M(z) w
& = & \mathcal{L}^{\mathrm{qinv}} h
+ \frac{1}{2} \nu(z) [Q_* h ] [\partial_z^2 \phi_z]_{z=0}
- z \nu(z) [Q_* h] \mathcal{L}^{\mathrm{qinv}} S(z),
\end{array}
\end{equation}
in which
\begin{equation}
M(z) w
=\mathcal{L}^{\mathrm{qinv}}
 R_{\mathcal{L};1}(z) w
+ z \nu(z)[Q_* R_{\mathcal{L};1}(z) w]
  \mathcal{L}^{\mathrm{qinv}} S(z)
- \frac{1}{2}  \nu(z)
[Q_* R_{\mathcal{L};1}(z) w] [\partial_z^2 \phi_z]_{z=0} .
\end{equation}
The desired properties now follow from the fact that
$I + zM(z)$ is invertible whenever $\abs{z}$ is sufficiently
small.
\end{proof}

In the following result
we explicitly identify the singular $O(z^{-2})$ and $O(z^{-1})$
terms in the expansion of $\mathcal{L}_z^{-1}$.
In order to express these in a convenient fashion,
we introduce the operator $\Gamma_*: L^2 \to \Real$
that acts as
\begin{equation}
\label{eq:prlm:def:gamma}
\Gamma_* h = - Q_* A_1 \mathcal{L}^{\mathrm{qinv}} h
  - \frac{1}{3} \frac{[\partial_z^3 \lambda_z]_{z=0}}
                     {[\partial_z^2 \lambda_z]_{z=0} } Q_* h.
\end{equation}

\begin{lem}
\label{lem:prlm:exp:l:z:inv}
Assume that (Hf), $(H\Phi)$, (HS1)-(HS3) and (HM) are satisfied
and pick a sufficiently small constant $\delta_z > 0$.
Then there exists an analytic function
\begin{equation}
\{ z \in \mathbb{C} : \abs{z} < \delta_z \}
\ni z \mapsto \mathcal{J}_z \in \mathcal{L}(L^2; H^1)
\end{equation}
so that we have
\begin{equation}
\label{eq:prlm:exp:l:z:inv}
\begin{array}{lcl}
 \mathcal{L}_z^{-1}
& = &
  2 [\partial_z^2 \lambda_z]_{z=0}^{-1} %\lambda''(0)^{-1}
\Big[
  z^{-2} (\Phi_*'  + z [\partial_z \phi_z]_{z=0}) Q_*
   + z^{-1} \Phi_*' \Gamma_*
\Big]
+ \mathcal{J}_z
\end{array}
\end{equation}
whenever $0 < \abs{z} < \delta_z$.
\end{lem}
\begin{proof}
Fix $z \neq 0$ together with $h \in L^2$ and
consider the functions
\begin{equation}
\begin{array}{lcl}
v_A & = &
 2 [\partial_z^2 \lambda_z]_{z=0}^{-1}
\Big[
  z^{-2} \big[\Phi_*' + z [\partial_z \phi_z]_{z=0}  \big] Q_* h
  + z^{-1} \Phi_*' \Gamma_* h
\Big]  ,
\\[0.2cm]
v_B & = &
\mathcal{L}^{\mathrm{qinv}} h   - z\mathcal{L}^{\mathrm{qinv}} A_1 \mathcal{L}^{\mathrm{qinv}} h ,
\\[0.2cm]
v_C & = &
2 [\partial_z^2 \lambda_z]_{z=0}^{-1}
\Big[
   \frac{1}{2} [\partial_z^2 \phi_z]_{z=0} Q_* f
      + [\partial_z \phi_z]_{z=0} \Gamma_* f
\Big]
\\[0.2cm]
& & \qquad
 + 2 z [\partial_z^2 \lambda_z]_{z=0}^{-1}
   \Big[
     \frac{1}{6} [\partial_z^3 \phi_z]_{z=0} Q_* h
      + \frac{1}{2} [\partial_z^2 \phi_z]_{z=0} \Gamma_* h
          %- \mathcal{L}^{\mathrm{qinv}}[ \frac{1}{2} A_1 \phi_0^{(2)}] Qf
     %- \mathcal{L}^{\mathrm{qinv}}[ A_1 \phi_0^{(1)} ]  \Gamma f
\Big]
\\[0.2cm]
& & \qquad
  - z \mathcal{L}^{\mathrm{qinv}} [\partial_z \phi_z]_{z=0} Q_* h .
\end{array}
\end{equation}
Performing the expansion
\begin{equation}
\mathcal{L}_z v_{\#}
= \mathcal{E}_{\#;0} + z \mathcal{E}_{\#;1} + z^2 \mathcal{R}_{\#;2}(z)
\end{equation}
and demanding that $\mathcal{R}_{\#;2}(z)$
is analytic in zero for $\# \in \{A, B, C \}$,
we can
use Lemma \ref{lem:prlms:ids:for:lambda:derivs}
to compute
\begin{equation}
\begin{array}{lcl}
\frac{1}{2} [\partial_z^2 \lambda_z]_{z=0}
   \mathcal{E}_{A;0} & = &
  \big(\frac{1}{2} A_2 \Phi_*' + A_1 [\partial_z \phi_z]_{z=0} \big) Q_* h
     + A_1 \Phi_*' \Gamma_* h ,
\\[0.2cm]
\frac{1}{2} [\partial_z^2 \lambda_z]_{z=0} \mathcal{E}_{A;1} & = &
  \big(\frac{1}{6} A_3 \Phi_*' + \frac{1}{2} A_2 [\partial_z \phi_z]_{z=0}
     \big) Q_* h
   + \frac{1}{2} A_2 \Phi_*' \Gamma_* h ,
\end{array}
\end{equation}
together with
\begin{equation}
\begin{array}{lcl}
\mathcal{E}_{B;0} & = &
  h - \Phi_*' Q_* h ,
\\[0.2cm]
\mathcal{E}_{B;1} & = &
  \Phi_*' Q_* A_1 \mathcal{L}^{\mathrm{qinv}} h
\end{array}
\end{equation}
and finally
\begin{equation}
\begin{array}{lcl}
\frac{1}{2} [\partial_z^2 \lambda_z]_{z=0} \mathcal{E}_{C;0} & = &
- \big(\frac{1}{2} A_2 \Phi_*' + A_1 [\partial_z \phi_z]_{z=0} \big) Q_* f
+ \frac{1}{2} [\partial_z^2 \lambda_z]_{z=0} \Phi_*' Q_* f
 - A_1 \Phi_*' \Gamma_* h ,
\\[0.2cm]
\frac{1}{2} [\partial_z^2 \lambda_z]_{z=0} \mathcal{E}_{C;1} & = &
\frac{1}{2} A_1 [\partial_z^2 \phi_z]_{z=0} Q_* h
+ A_1 [\partial_z \phi_z]_{z=0} \Gamma_* h
\\[0.2cm]
& & \qquad
    - \big(\frac{1}{6} A_3 \Phi_*' + \frac{1}{2} A_2 [\partial_z \phi_z]_{z=0}
        + \frac{1}{2} A_1 [\partial_z^2 \phi_z]_{z=0} \big) Q_* h
    + \frac{1}{6} [\partial_z^3 \lambda_z]_{z=0}  \Phi_*' Q_* h
       %+ \frac{1}{2} \lambda''(0) \phi_0^{(1)} Qf
  \\[0.2cm]
& & \qquad
     - \big(\frac{1}{2} A_2 \Phi_*' + A_1 [\partial_z \phi_z]_{z=0} \big) \Gamma_* h
    + \frac{1}{2} [\partial_z^2 \lambda_z]_{z=0}     \Phi_*' \Gamma_* h .
\end{array}
\end{equation}
Summing these expressions we see that
\begin{equation}
\begin{array}{lcl}
\mathcal{E}_{A;0} + \mathcal{E}_{B;0}
 + \mathcal{E}_{C;0}
&  = & h,
\\[0.2cm]
\mathcal{E}_{A;1} + \mathcal{E}_{B;1} + \mathcal{E}_{C;1}
& = &
\Phi_*' Q_* A_1 \mathcal{L}^{\mathrm{qinv}} h
+ \Phi_*' \Gamma_* h
+ \frac{1}{3} [\partial_z^2 \lambda_z]_{z=0}^{-1}
    [\partial_z^3 \lambda_z]_{z=0}  \Phi_*' Q_* h
\\[0.2cm]
& = &
0 ,
\end{array}
\end{equation}
where we used \sref{eq:prlm:def:gamma} to simplify the second expression.

In particular,
Lemma \ref{lem:prlm:inverse:b2}
allows us to write
\begin{equation}
\mathcal{L}_z^{-1} h =  v_A + v_B + v_C
- B_2(z) \big[ \mathcal{R}_{A;2}(z) + \mathcal{R}_{B;2}(z) + \mathcal{R}_{C;2}(z)
  \big],
\end{equation}
from which the desired expansion follows.
\end{proof}

\begin{proof}[Proof of Proposition \ref{prp:hom:delta:inv}]
For any  $\delta_z > 0$ and sufficiently small
$\eta_{\max} > 0$,
(HS1) implies
that $\mathcal{L}_z$ is invertible
for $\abs{\Re z} \le \eta_{\max}$
and $\delta_z \le \abs{\Im z} \le \pi$.
The result now follows from the
expansion obtained for $\mathcal{L}_z^{-1}$
in Lemma \ref{lem:prlm:exp:l:z:inv}.
\end{proof}

We now proceed towards establishing the
representation formula \sref{eq:prlm:main:repr:formula}.
To this end, we introduce the left-shift operator $S$
that acts on sequences as
\begin{equation}
\label{eq:prlm:intro:s}
[S V]_l = V_{l + 1}.
\end{equation}

\begin{lem}
\label{lem:prlm:lpl:on:shifts}
Consider any sequence $y \in BX_{\mu,\nu}(\mathcal{H})$
and pick an integer $k \ge 1$.
Then we have the identities
\begin{equation}
\label{eq:prlm:lpl:plus:shifts}
\begin{array}{lcl}
\mathcal{Z}_+[S^k y](z) & = &
e^{zk} \mathcal{Z}_+[y](z)
  -  \sum_{j=0}^{k-1}  e^{z(k-j)} v_j ,
\\[0.2cm]
\mathcal{Z}_+[S^{-k} y](z) & = &
  e^{-zk} \mathcal{Z}_+[y](z)
  +  \sum_{j=1}^{k}  e^{-z(k-j)} v_{-j}
\\[0.2cm]
\end{array}
\end{equation}
whenever $\Re z > \nu$,
together with
\begin{equation}
\label{eq:prlm:lpl:minus:shifts}
\begin{array}{lcl}
\mathcal{Z}_-[S^k y](z) & = &
  e^{zk} \mathcal{Z}_-[y](z)
  +  \sum_{j=0}^{k-1}  e^{z(k-j)} v_j ,
\\[0.2cm]
\mathcal{Z}_-[S^{-k} y](z) & = &
  e^{-zk} \mathcal{Z}_-[y](z)
  -  \sum_{j=1}^{k}  e^{-z(k-j)} v_{-j}
\end{array}
\end{equation}
whenever $\Re z < \mu$.
\end{lem}
\begin{proof}
Upon computing
\begin{equation}
\begin{array}{lcl}
\sum_{j \ge 0} e^{-z j} y_{j+k}
& = & \sum_{j' \ge k} e^{zk} e^{- z j'} y_{j'}
\\[0.2cm]
& = &
e^{z k} \sum_{j' \ge 0}  e^{ -z j'} y_{j'}
- e^{zk} \sum_{j'=0}^{k-1} e^{-z j'} y_{j'}
\end{array}
\end{equation}
together with
\begin{equation}
\begin{array}{lcl}
\sum_{j \ge 0} e^{-z j} y_{j-k}
&=& \sum_{j' \ge -k} e^{-z k} e^{- z j'} y_{j'}
\\[0.2cm]
&=&
e^{-z k} \sum_{j' \ge 0}  e^{- z j'} y_{j'}
+ e^{-z k} \sum_{j = 1}^{k}  e^{ z j} y_{-j},
\\[0.2cm]
\end{array}
\end{equation}
the identities \sref{eq:prlm:lpl:plus:shifts} readily follow.
In addition, we compute
\begin{equation}
\begin{array}{lcl}
\sum_{j \ge 1} e^{z j} y_{-j+k}
& = &
\sum_{j' \ge 1- k} e^{zk} e^{ z j'} y_{-j'}
\\[0.2cm]
& = &
 e^{zk}\sum_{j' \ge 1 }  e^{ z j'} y_{-j'}
 + e^{zk}\sum_{j=0}^{k-1}  e^{-zj} y_j
\end{array}
\end{equation}
together with
\begin{equation}
\begin{array}{lcl}
\sum_{j \ge 1} e^{z j} y_{-j-k}
& = & \sum_{j' \ge k+1} e^{-z k} e^{ z j'} y_{-j'}
\\[0.2cm]
& = & e^{-z k} \sum_{j' \ge 1}  e^{ z j'} y_{-j'}
- e^{-zk} \sum_{j'=1}^k e^{z j'} y_{-j'} ,
\end{array}
\end{equation}
from which \sref{eq:prlm:lpl:minus:shifts} follows.
\end{proof}

We now introduce for any $z \in \mathbb{C}$
and any $w: \{-\sigma_* + 1, \ldots, \sigma_* \}  \to H^1$
the notation
\begin{equation}
\begin{array}{lcl}
q_z [w]
& = &
\Big(
  \sum_{j=1}^{ \sigma_B- 1}
    w_{-j}  \sum_{k=j}^{\sigma_B - 1} e^{-z(k-j) } ,
    \sum_{j=0}^{\sigma_A-1}
       w_j \sum_{k=j+1}^{\sigma_A} e^{z (k - j) }
   ,
\\[0.2cm]
& & \qquad
  \sum_{j=0}^{\sigma_B-1}
       w_j \sum_{k=j+1}^{\sigma_B} e^{z (k - j) },
   \sum_{j=1}^{ \sigma_A- 1}
    w_{-j}  \sum_{k=j}^{\sigma_A - 1} e^{-z(k-j) } ,
    0
\Big).
\end{array}
\end{equation}

\begin{cor}
\label{cor:prlm:id:for:z:pm:s:diamond}
Consider any sequence $w \in BX_{\mu,\nu}(H^1)$
and pick an integer $k \ge 1$.
Then we have the identity
\begin{equation}
\begin{array}{lcl}
\mathcal{Z}_+\big[s^\diamond[w] \big](z)
& = & s^\diamond_z \mathcal{Z}_+[w](z)
 - q_z \mathrm{ev}_0 w
\end{array}
\end{equation}
whenever $\Re z > \nu$,
together with
\begin{equation}
\begin{array}{lcl}
\mathcal{Z}_-\big[s^\diamond[w]\big](z)
& = & s^\diamond_z \mathcal{Z}_-[w](z)
 + q_z \mathrm{ev}_0 w
\end{array}
\end{equation}
whenever $\Re z < \mu$.
\end{cor}
\begin{proof}
For the first compoment,
we compute
\begin{equation}
\begin{array}{lcl}
\sum_{j=0}^{\sigma_B - 1}
 \mathcal{Z}_+[S^{-j} w](z)
& = &
  \sum_{j=0}^{\sigma_B - 1}
    e^{-j z} \mathcal{Z}_+[w](z)
  + \sum_{j=1}^{\sigma_B - 1}
   \sum_{j'=1}^{j}
     e^{-z(j - j')} w_{-j'}
\\[0.2cm]
& = &
  \sum_{j=0}^{\sigma_B - 1}
    e^{-jz} \mathcal{Z}_+[w](z)
  + \sum_{j'=1}^{\sigma_B - 1}
    w_{-j'}
    \sum_{j=j'}^{\sigma_B -1}
      e^{-z (j-j')} .
\end{array}
\end{equation}
For the second component,
we note that
\begin{equation}
\begin{array}{lcl}
\sum_{j=1}^{\sigma_A}
 \mathcal{Z}_+[S^{j} w ](z)
& = &
  \sum_{j=1}^{\sigma_A}
    e^{j z} \mathcal{Z}_+[w](z)
  - \sum_{j=1}^{\sigma_A}
   \sum_{j'=0}^{j-1}
     e^{z(j - j')} w_{j'}
\\[0.2cm]
& = &
  \sum_{j=1}^{\sigma_A}
    e^{jz} \mathcal{Z}_+[w](z)
  - \sum_{j'=0}^{\sigma_A - 1}
    w_{j'}
    \sum_{j=j'+1}^{\sigma_A}
      e^{z (j-j')} .
\end{array}
\end{equation}
The desired expression
follows directly from
these computations,
noting that the
third and fourth
component
can be obtained by
flipping
$\sigma_A$ and $\sigma_B$
in the expressions
for the second respectively
first component.
\end{proof}

\begin{cor}
\label{cor:prlm:q:cnst:seq}
Consider any $w \in H^1$.
Then we have the identities
\begin{equation}
\begin{array}{lcl}
q_0[w \mathbf{1}]
& = & [\partial_z s^\diamond_z ]_{z=0} w ,
\\[0.2cm]
[\partial_z q_z]_{z=0} [w \mathbf{1} ]
& = & \frac{1}{2}\Big(
    [\partial_z s^\diamond_z ]_{z=0}
    + [\partial_z^2 s^\diamond_z ]_{z=0}
  \Big) w .
\end{array}
\end{equation}
\end{cor}
\begin{proof}
The expressions follow from the direct computation
\begin{equation}
\begin{array}{lcl}
q_0[w \mathbf{1}]
& = &
\Big(
\sum_{j=1}^{\sigma_B - 1}
  (\sigma_B - j ),
\sum_{j=0}^{\sigma_A - 1}
  (\sigma_A - j ),
  \sum_{j=0}^{\sigma_B - 1}
  (\sigma_B - j ),
  \sum_{j=1}^{\sigma_A - 1}
  (\sigma_A - j ), 0
\Big) w
\\[0.2cm]
& = &
\Big(
\sum_{j=1}^{\sigma_B - 1}
    j ,
\sum_{j=1}^{\sigma_A} j ,
\sum_{j=1}^{\sigma_B} j,
\sum_{j=1}^{\sigma_A - 1} j, 0
\Big) w
\\[0.2cm]
& = &
[\partial_z s^\diamond_z ]_{z=0} w ,
\end{array}
\end{equation}
%Note that this is indeed
%$[\partial_z s^\diamond_z ]_{z=0}$.
together with
%We also have
\begin{equation}
\begin{array}{lcl}
[\partial_z q_z]_{z=0}[w \mathbf{1} ]
& = &
\Big(
-\sum_{j=1}^{\sigma_B - 1}
  \sum_{k=j}^{\sigma_B - 1} (k-j),
\sum_{j=0}^{\sigma_A - 1}
  \sum_{k = j+1}^{\sigma_A }
    (k-j),
\\[0.2cm]
& & \qquad
\sum_{j=0}^{\sigma_B - 1}
  \sum_{k = j+1}^{\sigma_B }
    (k-j),
-\sum_{j=1}^{\sigma_A - 1}
  \sum_{k=j}^{\sigma_A - 1} (k-j), 0
\Big) w
\\[0.2cm]
& = &
\frac{1}{2}\Big(
  - \sum_{j=1}^{\sigma_B - 1}
    (\sigma_B - 1 -j)
    (\sigma_B - j),
    \sum_{j=0}^{\sigma_A - 1}
      \frac{1}{2}(\sigma_A - j)(\sigma_A + 1 - j)
  ,
\\[0.2cm]
& & \qquad
      \sum_{j=0}^{\sigma_B - 1}
      \frac{1}{2}(\sigma_B - j)(\sigma_B + 1 - j),
   - \sum_{j=1}^{\sigma_A - 1}
    (\sigma_A - 1 -j)
    (\sigma_A - j) , 0
\Big)  w
\\[0.2cm]
& = &
\frac{1}{2}
q_0[w \mathbf{1}]
 +
 \frac{1}{2}\Big(
  - \sum_{j=1}^{\sigma_B - 1}
    (\sigma_B - j)^2,
    \sum_{j=0}^{\sigma_A - 1}
      (\sigma_A - j)^2 ,
\\[0.2cm]
& & \qquad \qquad
   \sum_{j=0}^{\sigma_B - 1}
      (\sigma_B - j)^2 ,
    - \sum_{j=1}^{\sigma_A - 1}
    (\sigma_A - j)^2 , 0
\Big) w
\\[0.2cm]
& = &
  \frac{1}{2}\Big(
    [\partial_z s^\diamond_z ]_{z=0}
    + [\partial_z^2 s^\diamond_z ]_{z=0}
  \Big) w .
\end{array}
\end{equation}
\end{proof}

For any pair of sequences
\begin{equation}
(v,w): \{ - \sigma_* + 1 , \ldots, \sigma_* \} \to
H^1 \times H^1
\end{equation}
we introduce the notation
\begin{equation}
\label{eq:prlm:def:mc:q}
\mathcal{Q}(z)[v,w]
 = \big( (v_0 - w_0)e^z, Df (\tau \Phi_*) \overline{\tau} q_z w \big)^T.
\end{equation}
This allows us to take the two discrete Laplace transforms of the
main linear system \sref{eq:lin:repr:inhom:probl}.

\begin{cor}
\label{cor:prlm:id:for:lpl:trnfs}
Assume that (Hf) and $(H\Phi)$ are satisfied.
Pick a pair $\mu,\nu \in \Real$,
consider any $V \in BX_{\mu,\nu}(H^1 \times H^1)$
and write
\begin{equation}
H = \mathcal{D}(0)[V] - \mathcal{T}(0)[V] .
\end{equation}
Then we have the identity
\begin{equation}
\label{eq:cor:lpl:trnfs:reptr}
\begin{array}{lcl}
\Delta(z) \mathcal{Z}_+[V](z) & = & \mathcal{Q}(z) \mathrm{ev}_0 V
    + \mathcal{Z}_+[H](z)
\end{array}
\end{equation}
whenever $\Re s  > \nu$,
together with
\begin{equation}
\label{eq:cor:lpl:trnfs:reptr:min}
\begin{array}{lcl}
\Delta(z) \mathcal{Z}_-[V](z) & = & -\mathcal{Q}(z) \mathrm{ev}_0 V
    + \mathcal{Z}_-[H](z)
\end{array}
\end{equation}
whenever $\Re z  < \mu$.
\end{cor}
\begin{proof}
Writing $V = (v, w)$ and $H = (g,h)$, we
may use Lemma \ref{lem:prlm:lpl:on:shifts}
and Corollary \ref{cor:prlm:id:for:z:pm:s:diamond}
to compute
\begin{equation}
\begin{array}{lcl}
(e^z -1 ) \mathcal{Z}_+[v](z) - v_0 e^z
 & = & e^z (\mathcal{Z}_+[w](z) - w_0) + \mathcal{Z}_+[g](z) ,
\\[0.2cm]
Df( \tau \Phi_*) \overline{\tau} \big[
s^\diamond_z \mathcal{Z_+}[w](z) - q_z w
\big]
& = &
      - \mathcal{L}_* \mathcal{Z}_+[v](z) + \mathcal{Z}_+[h](z) ,
\end{array}
\end{equation}
which is equivalent to \sref{eq:cor:lpl:trnfs:reptr}.
The remaining identity \sref{eq:cor:lpl:trnfs:reptr:min}
follows in a similar fashion.
\end{proof}

For convenience, we introduce the linear operators
\begin{equation}
\begin{array}{lcl}
M_0 [v,w]
  &=&
Df(\tau\Phi_*) \overline{\tau}
\Big[
  s^\diamond_0 (v_0 - w_0)
        + q^\diamond_0   w
\Big] ,
\\[0.2cm]
M_1 [v,w] & = &
Df(\tau\Phi_*) \overline{\tau}
\Big[
  [\partial_z s^\diamond_z]_{z=0} (v_0 - w_0)
        + [\partial_z q^\diamond_z]_{z=0}   w
\Big] ,
\end{array}
\end{equation}
together with the projections
\begin{equation}
\begin{array}{lcl}
\pi_A
 & = & 2 [\partial_z^2 \lambda_z]_{z=0}^{-1}
    \big[   Q_*  M_1 + \Gamma_* M_0 \big] ,
\\[0.2cm]
\pi_B
 & = & 2 [\partial_z^2 \lambda_z]_{z=0}^{-1}  Q_* M_0 .
\end{array}
\end{equation}

\begin{lem}
\label{lem:prlm:id:for:residue}
Assume that (Hf), $(H\Phi)$, (HS1)-(HS3) and (HM) are satisfied.
Then we have the identity
\begin{equation}
\begin{array}{lcl}
\mathrm{Res}_{z = 0} \, e^{z l} \Delta(z)^{-1} \mathcal{Q}(z)
& = & [V_{\mathrm{hom}}^A ]_l \pi_A + [V^B_{\mathrm{hom}}]_l \pi_B .
\end{array}
\end{equation}
\end{lem}
\begin{proof}
Pick any $z \in \Wholes$ for which $\mathcal{L}_z$ is invertible.
Upon introducing the representation
\begin{equation}
e^{z l} \Delta(z)^{-1} \mathcal{Q}(z) [v,w]
= \mathcal{B}_l(z) \big( v_0 - w_0 , w \big)^T
\end{equation}
with
\begin{equation}
\mathcal{B}_l(z) =
\left(
  \begin{array}{cc}
    [\mathcal{B}_l(z)]_{1,1} & [\mathcal{B}_l(z)]_{1,2}
    \\[0.2cm]
    [\mathcal{B}_l(z)]_{2,1} & [\mathcal{B}_l(z)]_{2,2}
  \end{array}
\right),
\end{equation}
one can
use \sref{eq:prlm:def:delta:z:inv} and \sref{eq:prlm:def:mc:q}
to verify the expressions
\begin{equation}
\begin{array}{lcl}
{[}\mathcal{B}_l(z) ]_{1,1}
& = &
e^{z l} \mathcal{L}_z^{-1} Df(\tau \Phi_*)
  \overline{\tau} s^\diamond_z ,
\\[0.2cm]
{[}\mathcal{B}_l(z) ]_{1,2}
& = &
e^{z l} \mathcal{L}_z^{-1} Df(\tau \Phi_*)
  \overline{\tau} q^\diamond_z  ,
\\[0.2cm]
{[}\mathcal{B}_l(z) ]_{2,1}
& = &
- e^{zl}
+ e^{z l}  (1 - e^{-z}) \mathcal{L}_z^{-1} Df(\tau \Phi_*)
  \overline{\tau} s^\diamond_z ,
\\[0.2cm]
{[}\mathcal{B}_l(z) ]_{2,2}
& = &
e^{z l} (1 - e^{-z})\mathcal{L}_z^{-1} Df(\tau \Phi_*)
  \overline{\tau} q^\diamond_z .
\\[0.2cm]
\end{array}
\end{equation}
Using the Laurent expansion
\sref{eq:prlm:exp:l:z:inv},
we readily compute
\begin{equation}
\begin{array}{lcl}
\frac{1}{2}[\partial_z^2 \lambda_z]_{z=0}
   \mathrm{Res}_{z = 0} [\mathcal{B}_l(z) ]_{1,1}
  & = &  l \Phi_*' Q_* Df(\tau \Phi_*) \overline{\tau} s^\diamond_0
    + \Phi_*' Q_* Df(\tau \Phi_*) \overline{\tau}
        [\partial_z s^\diamond_z]_{z=0}
\\[0.2cm]
& & \qquad
        + [\partial_z \phi_z]_{z=0} Q_* Df(\tau \Phi_*) \overline{\tau} s^\diamond_0 +
         \Phi_*' \Gamma_* Df(\tau \Phi_*) \overline{\tau} s^\diamond_0 ,
\\[0.2cm]
\frac{1}{2}[\partial_z^2 \lambda_z]_{z=0} \mathrm{Res}_{z = 0}
[\mathcal{B}_l(z) ]_{2,1}
  & = &   \Phi_*' Q_* Df(\tau \Phi_*) \overline{\tau} s^\diamond_0  ,
\end{array}
\end{equation}
together with
\begin{equation}
\begin{array}{lcl}
\frac{1}{2}[\partial_z^2 \lambda_z]_{z=0} \mathrm{Res}_{z = 0}
[\mathcal{B}_l(z) ]_{1,2}
  & = &  l \Phi_*' Q_* Df(\tau \Phi_*) \overline{\tau} q^\diamond_0
    + \Phi_*' Q_* Df(\tau \Phi_*) \overline{\tau}
        [\partial_z q^\diamond_z]_{z=0}
\\[0.2cm]
& & \qquad
        + [\partial_z \phi_z]_{z=0}
              Q_* Df(\tau \Phi_*) \overline{\tau} q^\diamond_0 +
         \Phi_*' \Gamma_* Df(\tau \Phi_*) \overline{\tau} q^\diamond_0 ,
\\[0.2cm]
\frac{1}{2}[\partial_z^2 \lambda_z]_{z=0} \mathrm{Res}_{z = 0}
  [\mathcal{B}_l(z) ]_{2,2}
  & = &   \Phi_*' Q_* Df(\tau \Phi_*) \overline{\tau} q^\diamond_0 .
\end{array}
\end{equation}
The desired expressions now follow readily.
\end{proof}

\begin{lem}
\label{eq:prlm:pi:a:b:to:v:hom}
Assume that (Hf), $(H\Phi)$, (HS1)-(HS3) and (HM) are satisfied.
Then we have the identities
\begin{equation}
\pi_{A} \big[ \mathrm{ev}_0 V_{\mathrm{hom}}^A \big]
= \pi_{B} \big[ \mathrm{ev}_0 V_{\mathrm{hom}}^B \big] = 1,
\qquad \qquad
\pi_{A} \big[ \mathrm{ev}_0 V_{\mathrm{hom}}^B \big]
= \pi_{B} \big[ \mathrm{ev}_0 V_{\mathrm{hom}}^A \big] = 0 .
\end{equation}
\end{lem}
\begin{proof}
As preparations, we
use Lemma \ref{lem:prlms:ids:for:lambda:derivs}
to compute
\begin{equation}
\begin{array}{lcl}
Q_* A_1 \Phi_*'
 &= & 0,
\\[0.2cm]
Q_* A_1 [\partial_z \phi_z]_{z=0}
 & = &
   \frac{1}{2} [\partial_z^2 \lambda_z]_{z=0} - \frac{1}{2} Q_* A_2 \Phi_*' ,
\\[0.2cm]
Q_* A_1 [\partial_z^2 \phi_z]_{z=0}
 & = &
 \frac{1}{3 }[\partial_z^3 \lambda_z]_{z=0}
 - \frac{1}{3} Q_* A_3 \Phi_*'
 -  Q_* A_2 [\partial_z \phi_z]_{z=0} ,
\end{array}
\end{equation}
together with
\begin{equation}
\begin{array}{lcl}
\Gamma_* A_1 \Phi_*'
 & = & Q_* A_1 [\partial_z \phi_z]_{z=0}
\\[0.2cm]
 & = & \frac{1}{2}[\partial_z^2 \lambda_z]_{z=0} - \frac{1}{2} Q_* A_2 \Phi_*' ,
\\[0.2cm]
\Gamma_* A_2 \Phi_*'
& = &
   -\frac{1}{3} [\partial_z^2 \lambda_z]_{z=0}^{-1} [\partial_z^3 \lambda_z]_{z=0}
   Q_* A_2 \Phi_*'
   - Q_* A_1 \mathcal{L}^{\mathrm{qinv}} A_2 \Phi_*' ,
\\[0.2cm]
\Gamma_* A_1 \phi^{(1)}_0
 & = & - \frac{1}{3} [\partial_z^2 \lambda_z]_{z=0}^{-1} [\partial_z^3 \lambda_z]_{z=0}
           Q_* A_1 [\partial_z \phi_z]_{z=0}
       + \frac{1}{2} Q_* A_1 [\partial_z^2 \phi_z]_{z=0}
       + \frac{1}{2} Q_* A_1
            \mathcal{L}^{\mathrm{qinv}} A_2 \Phi_*'
\\[0.2cm]
% & = &
%   - \frac{1}{3}\frac{ [\partial_z^3 \lambda_z]_{z=0}}{[\partial_z^2 \lambda_z]_{z=0}}
%   + \frac{1}{6}\frac{ [\partial_z^3 \lambda_z]_{z=0}}{[\partial_z^2 \lambda_z]_{z=0}}
%           \Gamma_*_2 A_2 \Phi_*'
%   +\frac{1}{3} \frac{ [\partial_z^3 \lambda_z]_{z=0}}{[\partial_z^2 \lambda_z]_{z=0} }
%   - \frac{1}{6} \Gamma_*_2 A_3 \Phi_*'
%   - \frac{1}{2} \Gamma_*_2 A_2 [\partial_z \phi_z]_{z=0}
%   + \frac{1}{2} \Gamma_*_2 A_1
%            \mathcal{L}^{\mathrm{qinv}} A_2 \Phi_*'
%\\[0.2cm]
& = &
   + \frac{1}{6}[\partial_z^2 \lambda_z]_{z=0}^{-1} [\partial_z^3 \lambda_z]_{z=0}
           Q_* A_2 \Phi_*'
   - \frac{1}{6} Q_* A_3 \Phi_*'
   - \frac{1}{2} Q_* A_2 [\partial_z \phi_z]_{z=0}
\\[0.2cm]
& & \qquad 
   + \frac{1}{2} Q_* A_1
            \mathcal{L}^{\mathrm{qinv}} A_2 \Phi_*'  .
\end{array}
\end{equation}
Using Corollary
\ref{cor:prlm:q:cnst:seq}
together with \sref{eq:prlm:ids:ai:all}, %-\sref{eq:prlm:ids:ai:ii},
we also compute
\begin{equation}
\begin{array}{lcl}
M_0 \mathrm{ev}_0 V_{\mathrm{hom}}^A
 & = & A_1 \Phi_*' ,
\\[0.2cm]
M_1 \mathrm{ev}_0 V_{\mathrm{hom}}^A
 & = & \frac{1}{2} (A_1 + A_2) \Phi_*' ,
\end{array}
\end{equation}
together with
\begin{equation}
\begin{array}{lcl}
M_0 \mathrm{ev}_0 V_{\mathrm{hom}}^B
 & = &
   A_1 ( [\partial_z \phi_z]_{z=0} - \Phi_*' )
   + \frac{1}{2}(A_1 + A_2) \Phi_*'
\\[0.2cm]
& = &
  A_1 [\partial_z \phi_z]_{z=0} + \frac{1}{2}(A_2 - A_1) \Phi_*' ,
\\[0.2cm]
M_1 \mathrm{ev}_0 V_{\mathrm{hom}}^B
 & = &
   \frac{1}{2}(A_1 + A_2)
     ( [\partial_z \phi_z]_{z=0} - \Phi_*' )
   + \frac{1}{6}(A_3  + 3 A_2  + 2 A_1  ) \Phi_*'
\\[0.2cm]
 & = &
    \frac{1}{2} (A _1 + A_2) [\partial_z \phi_z]_{z=0}
    + \frac{1}{6} (A_3  - A_1) \Phi_*' .
\\[0.2cm]
\end{array}
\end{equation}
In particular, we find
\begin{equation}
\begin{array}{lclcl}
\frac{1}{2}[\partial_z^2 \lambda_z]_{z=0} \pi_A[V^A_{\mathrm{hom}}]
 & = &
  \frac{1}{2} Q_* A_2 \Phi_*'
   + \Gamma_* A_1 \Phi_*'
%\\[0.2cm]
 & = & \frac{1}{2}[\partial_z^2 \lambda_z]_{z=0} ,
\\[0.2cm]
\frac{1}{2}[\partial_z^2 \lambda_z]_{z=0} \pi_B[V^A_{\mathrm{hom}}]
& = & 0 ,
\end{array}
\end{equation}
together with
\begin{equation}
\begin{array}{lcl}
\frac{1}{2}[\partial_z^2 \lambda_z]_{z=0}
  \pi_A[V^B_{\mathrm{hom}}]
 & = &
   \frac{1}{2} Q_* (A_1 + A_2) [\partial_z \phi_z]_{z=0}
   + \frac{1}{6} Q_* A_3 \Phi_*'
   + \Gamma_* A_1 [\partial_z \phi_z]_{z=0}
   + \frac{1}{2} \Gamma_* (A_2 - A_1) \Phi_*'
\\[0.2cm]
& = &
  \frac{1}{4}[\partial_z^2 \lambda_z]_{z=0} - \frac{1}{4} Q_* A_2 \Phi_*'
    + \frac{1}{2} Q_* A_2 [\partial_z \phi_z]_{z=0}
    + \frac{1}{6} Q_* A_3 \Phi_*'
\\[0.2cm]
& & \qquad
  +  \frac{1}{6}\frac{ [\partial_z^3 \lambda_z]_{z=0}}{[\partial_z^2 \lambda_z]_{z=0}}
           Q_* A_2 \Phi_*'
   - \frac{1}{6} Q_* A_3 \Phi_*'
   - \frac{1}{2} Q_* A_2 [\partial_z \phi_z]_{z=0}
\\[0.2cm]
& & \qquad \qquad
   + \frac{1}{2} Q_* A_1
            \mathcal{L}^{\mathrm{qinv}} A_2 \Phi_*'
\\[0.2cm]
& & \qquad
  - \frac{1}{6}\frac{ [\partial_z^3 \lambda_z]_{z=0}}{[\partial_z^2 \lambda_z]_{z=0}} Q_* A_2 \Phi_*'
  - \frac{1}{2} Q_* A_1
            \mathcal{L}^{\mathrm{qinv}} A_2 \Phi_*'
  - \frac{1}{4}[\partial_z^2 \lambda_z]_{z=0} + \frac{1}{4} Q_* A_2 \Phi_*'
\\[0.2cm]
& = &  0 ,
\\[0.2cm]
\frac{1}{2}[\partial_z^2 \lambda_z]_{z=0}
  \pi_B[V^B_{\mathrm{hom}}]
& = & Q_* A_1 [\partial_z \phi_z]_{z=0} + \frac{1}{2} Q_* A_2 \Phi_*'
\\[0.2cm]
& = &
  \frac{1}{2}[\partial_z^2 \lambda_z]_{z=0} .
\end{array}
\end{equation}
\end{proof}

\begin{proof}[Proof of Proposition \ref{prp:prlm:repr:frm}]
Applying the inverse Laplace transform to the
identity \sref{eq:cor:lpl:trnfs:reptr},
the representation formula \sref{eq:prlm:main:repr:formula}
follows directly
from the Cauchy integral formula and Lemma's
\ref{lem:prlm:id:for:residue}-\ref{eq:prlm:pi:a:b:to:v:hom}.
\end{proof}

\section{Inhomogeneous linear system for constant $\theta$ }
\label{sec:linop:cc}

In this section we are interested in
the constant-coefficient linear system
\begin{equation}
\label{eq:lc:lin:sys:to:solve}
\mathcal{D}(\vartheta \mathbf{1})[V] =
[1 - P^{(1)}_{ \vartheta} ]\mathcal{T}(\vartheta \mathbf{1}) [V]
 + H,
\end{equation}
in which $\vartheta \in \Real$.
This equation can be seen as the linear part
of the $(v,w)$ system
in our main problem \sref{eq:set:red:sys}
for the special case of a constant sequence
$\theta_l = \vartheta$.

Our main result constructs a solution operator
for this system acting on sequences that are
allowed to grow at a small exponential rate.
In order to accomodate this,
we introduce
the notation
\begin{equation}
\begin{array}{lcl}
BS_{\eta}(\mathcal{H}) &=&
  BX_{-\eta, \eta}(\mathcal{H})
\\[0.2cm]
& = &
 \{ V: \mathbb{Z} \to \mathcal{H} \hbox{ for which }
   \norm{V}_{\eta} := \mathrm{sup} \,  e^{- \eta \abs{j} } \norm{V_j}_{\mathcal{H}}
 < \infty \}
\end{array}
\end{equation}
for any Hilbert space $\mathcal{H}$.
In addition, we recall the definitions \sref{eq:set:defns:q:i:ii}
for the projections $Q^{(1)}_{\vartheta}$ and $Q^{(2)}_{\vartheta}$.

\begin{prop}
\label{prp:inhom:cc:main}
Assume that (Hf), $(H\Phi)$, (HS1)-(HS3) and (HM)
are satisfied and recall the constant $r$ appearing in (Hf).
Pick a sufficiently small $\eta_{\max}> 0$
together with a sufficiently large $K > 0$.
Then for every $0 < \eta < \eta_{\max}$
there exists a $C^{r-1}$-smooth map
\begin{equation}
\mathcal{K}^{\mathrm{cc}}_\eta :
 \R \mapsto \mathcal{L}\big(  BS_\eta (H^1 \times L^2) ;
   BS_{\eta} (H^1\times H^1) \big)
\end{equation}
that satisfies the following properties.
\begin{enumerate}
\item[(i)] Pick any $0 < \eta < \eta_{\max}$ and $\vartheta \in \Real$.
For any $H \in BS_{\eta}( H^1 \times L^2)$,
the function $V = \mathcal{K}^{\mathrm{cc}}_{\eta}(\vartheta) H$
satisfies \sref{eq:lc:lin:sys:to:solve} and admits the orthogonality conditions
\begin{equation}
\label{eq:lin:cc:orth:cnds}
Q^{(1)}_{\vartheta} V_0 =  Q^{(2)}_{\vartheta} V_0 =0 .
\end{equation}
\item[(ii)] Pick any $0 < \eta < \eta_{\max}$ and
assume that $V \in BS_{\eta}(H^1 \times H^1)$ satisfies \sref{eq:lc:lin:sys:to:solve} with $H = 0$
for some $\vartheta \in \Real$.  Then there exists a pair $(a_1, a_2) \in \R^2$
for which we have
\begin{equation}
V_j = a_1 T_{\vartheta} (\Phi_*' , 0)
 +a_2 T_{\vartheta}( [\partial_z \phi_z]_{z=0}, \Phi_*' )  .
\end{equation}
\item[(iii)] For any $0 < \eta < \eta_{\max}$ and $\vartheta \in \Real$ we have the bound
\begin{equation}
\label{eq:lin:cc:bnd:k:cc}
\norm{\mathcal{K}^{\mathrm{cc}}_\eta(\vartheta)}_{\mathcal{L}\big(BS_{\eta}(H^1 \times L^2) ; BS_{\eta}(H^1 \times H^1) \big)}
   < K  \eta^{-3}. %\eta^{-2}.
\end{equation}
%
%\item[(iv)] Recall the constant $r$ appearing in XXX.
%For any $0 < \eta < \eta_{\max}$, the map
%\begin{equation}
%\vartheta  \mapsto \mathcal{K}^{lc}_\eta(\vartheta) \in \mathcal{L}\big(BS_{\eta}(L^2 \times L^2) ; BS_{\eta}(H^1 \times L^2) \big)
%\end{equation}
%is $C^{r - 1}$ smooth.
%
\item[(iv)]
Pick any $0 < \eta < \eta_{\max}$. Then for any pair $(\vartheta_1, \vartheta_2) \in \Real^2$,
we have the identity
\begin{equation}
\mathcal{K}_\eta^{\mathrm{cc}}(\vartheta_1)= T_{\vartheta_1 -\vartheta_2}
\mathcal{K}_\eta^{\mathrm{cc}}(\vartheta_2) T_{\vartheta_2 -\vartheta_1}.
\end{equation}
\item[(v)] Consider a pair $(\eta_1 , \eta_2) \in (0, \eta_{\max})^2$
together with a function
\begin{equation}
H \in BS_{\eta_1}(H^1 \times L^2) \cap BS_{\eta_2}(H^1 \times L^2).
\end{equation}
Then for any $\vartheta\in\R$ we have
\begin{equation}
\mathcal{K}_{\eta_1}^{\mathrm{cc}}(\vartheta)H=
 \mathcal{K}_{\eta_2}^{\mathrm{cc}}(\vartheta)H .
\end{equation}
\end{enumerate}
\end{prop}

Our strategy is to exploit the representation formula
derived in \S\ref{sec:prlm}
for the unprojected problem
\begin{equation}
\label{eq:lc:lin:sys:to:solve:unproj}
\mathcal{D}(0)[V] - \mathcal{T}(0) [V]
%\Lambda^{up} V
= H .
\end{equation}
In particular, we first use the Fourier symbols $\Delta(z)$
defined in \sref{eq:prlm:def:delta:z} to construct an inverse
in the sequence spaces
\begin{equation}
\ell^2_{\eta}(\mathcal{H})
= \{ V: \Wholes \to \mathcal{H} \hbox{ for which }
       \norm{V}_{\ell^2_\eta}^2 :=\sum_{l \in \Wholes} e^{-2 \eta l} \norm{V_l}_{\mathcal{H}}^2
          < \infty \},
\end{equation}
where again $\mathcal{H}$ is a Hilbert space.
This can subsequently be used to
obtain an inverse in the spaces
\begin{equation}
\ell^\infty_{\eta}(\mathcal{H})
= \{ V: \Wholes \to \mathcal{H} \hbox{ for which }
       \norm{V}_{\ell^\infty_\eta} :=
       \sup_{l \in \Wholes} e^{- \eta l} \norm{V_l}_{\mathcal{H}}
          < \infty \}
\end{equation}
by exploiting the fact that interactions
between lattice sites
decay exponentially with respect to the separation distance.

\begin{lem}
\label{eq:lin:cc:inv:l2}
Assume that (Hf), $(H\Phi)$, (HS1)-(HS3) and (HM)
are satisfied and
pick a sufficiently small $\eta_{\max} > 0$
together with a sufficiently large $K > 0$.
Then for every $\eta$ with $0 < \abs{\eta} < \eta_{\max}$,
there exists a bounded operator
\begin{equation}
\Lambda^{\mathrm{inv}}_{\eta} : \ell^2_{\eta}( H^1 \times L^2)
   \to \ell^2_{\eta}(H^1 \times H^1)
\end{equation}
that satisfies the following properties.
\begin{itemize}
\item[(i)]{
For any $H \in \ell^2_{\eta}(H^1 \times L^2)$, the sequence
$V = \Lambda^{\mathrm{inv}}_{\eta} H$
satisfies \sref{eq:lc:lin:sys:to:solve:unproj}.
%$\Lambda^{up} V = H$. %\sref{eq:lc:lin:sys:to:solve}.
}
\item[(ii)]{
We have the bound
\begin{equation}
\norm{   \Lambda^{\mathrm{inv}}_{\eta}  }_{
   \mathcal{L}\big(  \ell^2_{\eta}(H^1 \times L^2) ,
                  \ell^2_{\eta}(H^1 \times H^1) \big)
 } \le K \eta^{-2} .
\end{equation}
}
\item[(iii)]{
We have the explicit expression
\begin{equation}
\begin{array}{lcl}
\big[\Lambda^{\mathrm{inv}}_{\eta}  H \big]_l
 & = &
   \frac{1}{2 \pi i} \int_{\eta - i \pi}^{\eta + i \pi}
      e^{z l} \Delta(z)^{-1} \big[ \mathcal{Z}_+[H](z) + \mathcal{Z}_-[H](z) \big] \, d z.
\end{array}
\end{equation}
}
\end{itemize}
\end{lem}
\begin{proof}
This follows directly from Proposition \ref{prp:hom:delta:inv}
and standard properties of the Fourier transform; see
for example \cite[{\S}3]{HJHCM}.
\end{proof}

\begin{lem}
\label{lem:lin:cc:infty}
Assume that (Hf), $(H\Phi)$, (HS1)-(HS3) and (HM) are all satisfied.
Pick a sufficiently small $\eta_{\max} > 0$
together with a sufficiently large $K > 0$.
Then for every pair
$0 < \eta_1 < \eta_2 < \eta_{\max}$ and
any
\begin{equation}
H \in \ell^\infty_{\eta_1} (H^1 \times L^2)\cap \ell^2_{\eta_2}( H^1 \times L^2 )
\end{equation}
we have the inclusion
$ \Lambda_{\eta_2}^{\mathrm{inv}} H \in \ell^\infty_{\eta_1}(H^1 \times H^1)$,
together with the bound
\begin{equation}
\norm{\Lambda_{\eta_2}^{\mathrm{inv}} H}_{\ell_{\eta_1}^\infty (H^1\times H^1)}
\leq K \eta_1^{-3}
   \norm{H}_{\ell^\infty_{\eta_1}(H^1 \times L^2) } .
\end{equation}
In addition, for any
\begin{equation}
H \in \ell^\infty_{-\eta_1} (H^1 \times L^2)\cap \ell^2_{-\eta_2}( H^1 \times L^2 )
\end{equation}
we have the inclusion
$ \Lambda_{-\eta_2}^{\mathrm{inv}} H \in \ell^\infty_{-\eta_1}(H^1 \times H^1)$,
together with the bound
\begin{equation}
\norm{\Lambda_{-\eta_2}^{\mathrm{inv}} H}_{\ell_{-\eta_1}^\infty (H^1\times H^1)}
\leq K \eta_1^{-3} \norm{H}_{\ell^\infty_{-\eta_1}(H^1 \times L^2) } .
\end{equation}
\end{lem}
\begin{proof}
Following the approach in \cite[Lem. 5.8]{HJHBRGSG},
we introduce the sequences
\begin{equation}
H^{(k)}: \Wholes \to H^1 \times L^2,
\qquad
V^{(k)}: \Wholes \to H^1 \times H^1
\end{equation}
for $k \in \Wholes$ by writing
\begin{equation}
H^{(k)}_l = \delta_{k l} H_l ,
\qquad
V^{(k)} =
 \Lambda^{\mathrm{inv}}_{\eta_2} H^{(k)} .
\end{equation}
In view of the convergence
\begin{equation}
\sum_{k \in \Wholes} H^{(k)} = H \in \ell^2_{\eta_2}(H^1 \times L^2) ,
\end{equation}
the boundedness of $\Lambda^{\mathrm{inv}}_{\eta_2}$
implies that also
\begin{equation}
\sum_{k \in \Wholes} V^{(k)} = V \in \ell^2_{\eta_2}(H^1 \times H^1)
\end{equation}
and hence
\begin{equation}
\sum_{k \in \Wholes} V^{(k)}_l = V_l \in H^1 \times H^1
\end{equation}
for all $l \in \Wholes$.

We now pick two constants $\eta_\pm$
in such a way that
\begin{equation}
0 < \eta_- < \eta_1 < \eta_2 < \eta_+ < \eta_{\max} .
\end{equation}
By construction, we have
\begin{equation}
 H^{(k)}
 \in  \ell^2_{\eta_-}(H^1 \times L^2) \cap \ell^2_{\eta_+}(H^1 \times L^2)
    \cap \ell^2_{\eta_2}(H^1 \times L^2).
\end{equation}
Recalling the left-shift operator $S$
defined in \sref{eq:prlm:intro:s},
we note that
\begin{equation}
S^k \Lambda^{\mathrm{inv}}_{\eta_2}
 = \Lambda^{\mathrm{inv}}_{\eta_2} S^k
%S^k \Lambda^{\mathrm{inv}}_{\eta_\pm} = \Lambda^{\mathrm{inv}}_{\eta_\pm} S^k
\end{equation}
In particular,
we have
\begin{equation}
S^k V^{(k)}=
S^k \Lambda_{\eta_2}^{\mathrm{inv}} H^{(k)}
=
\Lambda_{\eta_2}^{\mathrm{inv}} S^k H^{(k)}
=  \Lambda_{\eta_\pm}^{\mathrm{inv}} S^k H^{(k)} .
\end{equation}
Here the last identity follows from
Proposition \ref{prp:prlm:repr:frm},
since the sequence
\begin{equation}
Y = \Lambda_{\eta_2}^{\mathrm{inv}} S^k H^{(k)}
 -  \Lambda_{\eta_\pm}^{\mathrm{inv}} S^k H^{(k)}
\end{equation}
satisfies
the inclusions
\begin{equation}
Y \in \ell_{\eta_2}^2(H^1 \times H^1) + \ell_{\eta_\pm}^2(H^1 \times H^1)
\subset BX_{\eta_-, \eta_+}(H^1 \times H^1) ,
\end{equation}
together with the homogeneous problem
\begin{equation}
\mathcal{D}(0) Y = \mathcal{T}(0)Y .
\end{equation}

We are now able to
use item (ii) of Lemma \ref{eq:lin:cc:inv:l2} to compute
\begin{equation}
\begin{array}{lcl}
 \norm{  e^{- \eta_1 l}V^{(k)}_l }_{H^1 \times H^1}
& =  &
    \norm{ e^{ - \eta_1 l} [S_k V^{(k)}]_{l-k} }_{H^1 \times H^1}
\\[0.2cm]
& =  &
    \norm{ e^{ - \eta_1 l}
   [\Lambda^{\mathrm{inv}}_{\eta_\pm} S_k H^{(k)}]_{l-k}
      }_{H^1 \times H^1}
\\[0.2cm]
& \le  & C_1 \eta_{\pm}^{-2}
  e^{ - \eta_1 l}
   e^{ -\eta_\pm (k-l) }\norm{S_k H^{(k)} }_{\ell^2_{\eta_\pm}(H^1 \times L^2) }
\\[0.2cm]
& =  &
   C_1 \eta_\pm^{-2}   e^{ - \eta_1 l}
   e^{ - \eta_\pm(k-l)} \norm{H_k }_{H^1 \times L^2 }
\\[0.2cm]
& =  &
   C_1 \eta_\pm^{-2} e^{ - (\eta_\pm - \eta_1)(k - l) }   e^{- \eta_1 k }  \norm{H_k }_{H^1 \times L^2 }
\\[0.2cm]
& \le &
   C_1 \eta_\pm^{-2} e^{ - (\eta_\pm - \eta_1)(k - l) }  \norm{H}_{\ell^\infty_{\eta_1}(H^1 \times L^2)}
\end{array}
\end{equation}
for some $C_1 > 0$.
Summing over $k$, we hence find
\begin{equation}
\begin{array}{lcl}
\norm{ e^{-\eta_1 l} V_l }_{H^1\times H^1} &\leq &
C_1 \norm{H}_{\ell^\infty_{\eta_1} ( H^1 \times L^2) }
\big[ \eta_+^{-2} \sum_{k\geq l} e^{ -(\eta_+ - \eta_1)(k - l) }
 +\sum_{k<l} \eta_-^{-2} e^{ - (\eta_1 - \eta_-) (l  - k) }\big]
\\[0.2cm]
 & \leq & C_1 \big[ \eta_+^{-2} (\eta_+ - \eta_1)^{-1}
   + \eta_-^{-2}(\eta_1 - \eta_-)^{-1}\big]
      \norm{H}_{\ell^\infty_{\eta_1} ( H^1 \times L^2) }.
\end{array}
\end{equation}
The result follows directly from
this bound, possibly after decreasing the
size of $\eta_{\max} > 0$.
\end{proof}

For any $H \in BS_{\eta}(H^1 \times L^2)$
we now introduce the splitting
\begin{equation}
H = H_{\ge 0} + H_{< 0}
\end{equation}
by writing
\begin{equation}
\big[ H_{\ge 0} \big]_{l} = \mathbf{1}_{l \ge 0} H_l,
\qquad
\big[ H_{< 0} \big]_{l} = \mathbf{1}_{l < 0} H_l.
\end{equation}
We subsequently write
\begin{equation}
\mathcal{K}^{\mathrm{up}}_{\eta;I} H
= \Lambda^{\mathrm{inv}}_{\eta + \epsilon} H_{\ge 0}
+ \Lambda^{\mathrm{inv}}_{-\eta - \epsilon} H_{< 0}
\end{equation}
for some small $\epsilon > 0$,
which by construction implies that
$V = \mathcal{K}^{\mathrm{up}}_{\eta;I} H$
satisfies the unprojected problem \sref{eq:lc:lin:sys:to:solve:unproj} with $\vartheta = 0$.
In addition, Lemma \ref{lem:lin:cc:infty}
implies that $V \in BS_{\eta}(H^1 \times H^1)$.
In order to allow for any $\vartheta \in \Real$,
we introduce the operator
\begin{equation}
\mathcal{K}^{\mathrm{up}}_{\eta;II}(\vartheta) H
= T_\vartheta \mathcal{K}^{\mathrm{up}}_{\eta;I} T_{-\vartheta} H.
\end{equation}
In view of
the orthogonality conditions \sref{eq:lin:cc:orth:cnds},
we finally write
\begin{equation}
\mathcal{K}^{\mathrm{up}}_{\eta}(\vartheta) H
=  \mathcal{K}^{\mathrm{up}}_{\eta;II}(\vartheta)H
 - T_{\vartheta} V_{\mathrm{hom}}^A Q^{(1)}_\vartheta
   \mathrm{pev}_0
     \mathcal{K}^{\mathrm{up}}_{\eta;II}(\vartheta) H
 - T_{\vartheta} V_{\mathrm{hom}}^B  Q^{(2)}_\vartheta
     \mathrm{pev}_0
       \mathcal{K}^{\mathrm{up}}_{\eta;II}(\vartheta) H
      .
\end{equation}

\begin{lem}
\label{lem:linsys:inhom:cc:k:up}
Assume that (Hf), $(H\Phi)$, (HS1)-(HS3) and (HM) are all satisfied.
Pick a sufficiently small $\eta_{\max}> 0$
together with a sufficiently large $K > 0$.
Then for any $0 < \eta < \eta_{\max}$,
any $\vartheta \in \Real$
and any $H \in BS_{\eta}(H^1 \times L^2)$,
 the function
$V = \mathcal{K}^{\mathrm{up}}_{\eta}(\vartheta) H$
satisfies the unprojected problem \sref{eq:lc:lin:sys:to:solve:unproj}
and admits the orthogonality conditions
\begin{equation}
Q^{(1)}_{\vartheta} V_0
= Q^{(2)}_{\vartheta} V_0
= 0 .
\end{equation}
In addition, properties
(iii) - (v) from Proposition \ref{prp:inhom:cc:main}
are satisfied after replacing
$\mathcal{K}_{\eta}^{\mathrm{cc}}$
by $\mathcal{K}_{\eta}^{\mathrm{up}}$.
\end{lem}
\begin{proof}
In view of the discussion above, the statements follow directly
from the fact that the set
of solutions to the homogeneous problem
\sref{eq:lin:repr:prob:hom} in $BS_{\eta}(H^1 \times H^1)$
is two-dimensional as a consequence of Proposition \ref{prp:prlm:repr:frm}.
A detailed discussion can be found in the proof
of \cite[Prop. 5.1]{HJHBRGSG}.
\end{proof}

We now set out to lift the results above
from the unprojected system \sref{eq:lc:lin:sys:to:solve:unproj}
to the full system \sref{eq:lc:lin:sys:to:solve}.
A key role is reserved for the summation operator
$\mathcal{J}$ that acts on a sequence $W$
as
\begin{equation}
\label{eq:lin:cc:def:j:sum}
\begin{array}{lcl}
\mathcal{J}[W]_l
& = &  \sum_{j=1}^l
   W_{j - 1}
  - \sum_{j=1}^{-l}
      W_{-j}
\\[0.2cm]
& = &
   \sum_{j=0}^{l-1}
   W_{j}
  - \sum_{j=1}^{-l}
       W_{-j} ,
\end{array}
\end{equation}
with the usual remark
that sums are set to zero
when the lower bound is strictly larger than the upper bound.

\begin{lem}
Pick a Hilbert space $\mathcal{H}$ together with a constant $\eta > 0 $.
Then for any $W \in BS_{\eta}(\mathcal{H})$,
we have $\mathcal{J}[W] \in BS_{\eta}(\mathcal{H}) $,
with
\begin{equation}
\label{eq:linsys:inhom:cc:id:for:j:w}
(S - I ) \mathcal{J}[W]
= W  .
\end{equation}
\end{lem}
\begin{proof}
These statements follow directly
by inspecting the definition \sref{eq:lin:cc:def:j:sum}.
\end{proof}

We now write
\begin{equation}
\mathcal{N}^{\mathrm{cc}}_{\eta}
= \{ V \in BS_{\eta}(H^1 \times H^1) :
  \mathcal{D}(0) [V] = (I - P^{(1)}_0) \mathcal{T}(0)[V] \}
\end{equation}
for the set of solutions to the homogenous
version of \sref{eq:lc:lin:sys:to:solve}.
By relating this set to its counterpart for \sref{eq:lin:repr:prob:hom}
we show that $\mathcal{N}^{\mathrm{cc}}_{\eta}$ is also two-dimensional
for small $\eta > 0$.

\begin{lem}
\label{lem:linsys:inhom:cc:char:ker}
Assume that (Hf), $(H\Phi)$, (HS1)-(HS3) and (HM) are all satisfied.
Then for all sufficiently small $\eta > 0$,
we have the identification
\begin{equation}
\label{eq:lin:cc:span:n:cc}
\mathcal{N}^{\mathrm{cc}}_{\eta}
= \mathrm{span} \{
  \big(\Phi_*' , 0\big),
  \big( [\partial_z \phi_z]_{z=0} , \Phi_*' \big)
  \} .
\end{equation}
\end{lem}
\begin{proof}
A direct computation shows that
\begin{equation}
\mathcal{D}(0) [\Phi_*' , 0] = \mathcal{T}(0)[ \Phi_*' , 0 ] =0 ,
\end{equation}
together with
\begin{equation}
\begin{array}{lcl}
\mathcal{T}(0)\Big[ [\partial_z \phi_z]_{z=0}, \Phi_*' \Big]
& = & (\Phi_*' , A_1 \Phi_*') ,
\\[0.2cm]
\mathcal{D}(0)\Big[ [\partial_z \phi_z]_{z=0} , \Phi_*' \Big]
& = & \big(0, Df (\tau \Phi_* ) \overline{\tau} s^\diamond[ \Phi_*' ] \mathbf{1} \big)
\\[0.2cm]
& = & \big(0, Df (\tau \Phi_* ) \overline{\tau} s^\diamond_0  \Phi_*'  \big)
\\[0.2cm]
& = & \big(0, A_1 \Phi_*'  \big) .
\end{array}
\end{equation}
In particular, we find
\begin{equation}
(I - P^{(1)}_0) \mathcal{T}(0)
  \big[ [\partial_z \phi_z]_{z =0} , \Phi_*'\big] = \big(0, A_1 \Phi_*' \big) ,
\end{equation}
which verifies that the right-hand-side of
\sref{eq:lin:cc:span:n:cc} is indeed
contained in $\mathcal{N}^{\mathrm{cc}}_{\eta}$.

Conversely, let us consider
a sequence $W \in \mathcal{N}^{\mathrm{cc}}_{\eta}$,
which implies that
\begin{equation}
\begin{array}{lcl}
\mathcal{D}(0) \Big[  P^{(1)}_0 \mathcal{J}\big[\mathcal{T}(0)[W]\big] \Big]
& = &  (S - I )   P^{(1)}_0 \mathcal{J}\big[\mathcal{T}(0)[W] \big]
\\[0.2cm]
& = &  P^{(1)}_0 (S - I )   \mathcal{J}\big[\mathcal{T}(0)[W] \big]
\\[0.2cm]
& = & P^{(1)}_0 \mathcal{T}(0)[W] .
\end{array}
\end{equation}
Upon writing
\begin{equation}
V = W + P^{(1)}_0 \mathcal{J}\big[\mathcal{T}(0)[W]\big] ,
\end{equation}
we note that
the identity $\mathcal{T}(0) P^{(1)}_0 = 0$
implies that
$\mathcal{T}(0)[V] = \mathcal{T}(0)[W]$.
We may hence compute
\begin{equation}
\begin{array}{lcl}
\mathcal{D}(0)[V] & = &
\mathcal{D}(0)[W] + \mathcal{D}(0) P^{(1)}_0 \mathcal{J}\big[\mathcal{T}(0)[W] \big]
\\[0.2cm]
& = & (I - P^{(1)}_0) \mathcal{T}(0)[W]
  + P^{(1)}_0[\mathcal{T}(0) W]
\\[0.2cm]
& = & (I - P^{(1)}_0) \mathcal{T}(0)[V]
  + P^{(1)}_0[\mathcal{T}(0) V]
\\[0.2cm]
& = & \mathcal{T}(0)[V] .
\end{array}
\end{equation}
Notice that the map $W \mapsto W + P^{(1)}_0 \mathcal{J}\big[\mathcal{T}(0)[W]\big]$
maps $BS_{\eta}(H^1 \times H^1)$ into itself.
It is also injective, which can be easily seen
by looking at the second component.
This implies that $\mathcal{N}^{\mathrm{cc}}_{\eta}$
is at most two-dimensional.
%,
%using XXX (char kernel previous section)
\end{proof}

\begin{proof}[Proof of Proposition \ref{prp:inhom:cc:main}]
Pick $H \in BS_{\eta}(H^1 \times L^2)$
and write
\begin{equation}
\label{eq:inhom:cc:id:for:k:cc}
V = [I - P^{(1)}_0] \mathcal{K}^{\mathrm{up}}(0) [I-P^{(1)}_0] H +
  P^{(1)}_0 \mathcal{J}[H].
\end{equation}
Using $\mathcal{T}(0) P_{0} = 0$
we may compute
\begin{equation}
\mathcal{T}(0) V =
\mathcal{T}(0) \mathcal{K}^{\mathrm{up}}_{\eta}(0) [I-P^{(1)}_0] H.
\end{equation}
Using the
commutation relation
$P^{(1)}_0 \mathcal{D}(0) = \mathcal{D}(0) P^{(1)}_0$
together with
\sref{eq:linsys:inhom:cc:id:for:j:w},
we see that
\begin{equation}
\mathcal{D}(0) P^{(1)}_0 \mathcal{J}[H]
= P^{(1)}_0 \mathcal{D}(0) \mathcal{J}[H]
= P^{(1)}_0 (S - I) \mathcal{J}[H]
= P^{(1)}_0 H.
\end{equation}
In particular, we may compute
\begin{equation}
\begin{array}{lcl}
\mathcal{D}(0)[V]
 & = &
   [I - P^{(1)}_0] \mathcal{T}(0) \mathcal{K}^{\mathrm{up}}_{\eta}(0) [I-P^{(1)}_0] H
     + [I-P^{(1)}_0] H + P^{(1)}_0 H
\\[0.2cm]
 & = &
   [I - P^{(1)}_0] \mathcal{T}(0) V + H.
\end{array}
\end{equation}
Upon writing
\begin{equation}
\mathcal{K}^{\mathrm{cc}}_{\eta}(\vartheta)
=  [I - P^{(1)}_{\vartheta}] \mathcal{K}^{\mathrm{up}}_{\eta}(\vartheta)
   [I-P^{(1)}_{\vartheta}] H
     + P^{(1)}_{\vartheta} \mathcal{J}[H],
\end{equation}
the desired properties
now follow readily from Lemma's \ref{lem:linsys:inhom:cc:k:up}
and \ref{lem:linsys:inhom:cc:char:ker}.
\end{proof}

\section{Slowly varying coefficients}
\label{sec:linop:sv}

In this section we study the properties
of the bounded linear operator
\begin{equation}
\Lambda(\theta): BS_{\eta}(H^1\times L^2) \mapsto BS_\eta(L^2 \times L^2)
\end{equation}
that for any sequence $\theta :\Z \mapsto \R$
acts as
\begin{equation}
[\Lambda(\theta)V]_l=
  \mathrm{pev}_l \mathcal{D}(\theta)[V]
    - (1-P^{(1)}_{\theta_l}) \mathrm{pev}_l \mathcal{T}(\theta)[V].
\end{equation}
We are specially interested in cases where
the sequence $\theta$ varies slowly with respect to $l \in \mathbb{Z}$,
which means
\begin{equation}
\label{eq:sv:theta:diff:small}
\norm{(S-I) \theta}_{\infty} < \delta_\theta
\end{equation}
for some small $\delta_\theta > 0$.

Our first main result states that the kernel of $\Lambda(\theta)$ is
again two-dimensional, provided that \sref{eq:sv:theta:diff:small} holds.
For technical reasons, we also extend the two basis functions
for the kernel to situations where \sref{eq:sv:theta:diff:small} fails
to hold.

\begin{prop}
\label{prp:sl:kern}
Assume that (Hf), $(H\Phi)$, (HS1)-(HS3) and (HM) are all satisfied.
Pick a sufficiently small constant $\delta_\theta > 0$
together with a sufficiently small $\eta > 0$.
Then for every $\theta: \Wholes \to \Real$
there exist two functions
\begin{equation}
V^A_{\mathrm{hom}}(\theta) \in \bigcap_{0 < \eta < \eta_{\max} } BS_{\eta}(H^1 \times H ^1),
\qquad
V^B_{\mathrm{hom}}(\theta) \in \bigcap_{0 < \eta < \eta_{\max} } BS_{\eta}(H^1 \times H ^1)
\end{equation}
that satisfy the following properties.
\begin{enumerate}
\item[(i)] For any $\theta :\Z\mapsto \R$ with $\norm{(S-I) \theta}_{\infty} < \delta_\theta$,
we have the identities
\begin{equation}
\Lambda(\theta) V^A_{\mathrm{hom}}(\theta) = \Lambda(\theta) V^B_{\mathrm{hom}}(\theta) = 0.
\end{equation}

\item[(ii)]
The normalization conditions
\begin{equation}
Q^{(1)}_{\theta_0} [V^A_{\mathrm{hom}}(\theta)]_0 = Q^{(2)}_{\theta_0} [ V^B_{\mathrm{hom}}(\theta) ]_0 = 1,
\qquad
Q^{(1)}_{\theta_0} [V^B_{\mathrm{hom}}(\theta)]_0 = Q^{(2)}_{\theta_0} [V^A_{\mathrm{hom}}(\theta) ]_0 = 0
\end{equation}
hold for all $\theta: \Wholes \to \Real$.

\item[(iii)]
Pick any $0 < \eta < \eta_{\max}$
and suppose that $\Lambda(\theta) V = 0$
for some $V \in BS_{\eta}(H^1 \times H^1)$ and
$\theta: \Wholes \to \Real$ for which $\norm{(S - I) \theta}_\infty < \delta_\theta$.
Then we have the identity
\begin{equation}
V = V^A_{\mathrm{hom}}(\theta) Q^{(1)}_{\theta_0} V_0
+ V^B_{\mathrm{hom}}(\theta) Q^{(2)}_{\theta_0} V_0.
\end{equation}
\end{enumerate}
\end{prop}

Our second main result constructs
operators $\mathcal{K}_{\eta}(\theta)$ that
can be seen as an inverse for $\Lambda(\theta)$
whenever \sref{eq:sv:theta:diff:small} holds. Naturally,
the kernel elements above obtained in the result above
need to be projected out, which is performed in
\sref{eq:sv:proj:at:zero}. Special care needs
to be taken when considering the smoothness with respect to $\theta$.
Indeed, the smoothness criteria below are based on the
%resemble the
arguments involving nested Banach spaces argument
that are traditionally used to establish the smoothness of center manifolds;
see for example \cite[{\S}IX.7]{VLDELAY}.  We remark that the notation
$\mathcal{L}^{(p)}$ stands for bounded $p$-linear maps.

\begin{prop}
\label{prp:sl:inverse}
Assume that (Hf), $(H\Phi)$, (HS1)-(HS3) and (HM) are all satisfied.
Recall the integer $r$ appearing in (Hf)
and pick two sufficiently
small constants
$0 < \eta_{min} < 2r \eta_{\min} < \eta_{\max}$.
Then for every $\eta_{\min} < \eta < \eta_{\max}$,
there exists a map
\begin{equation}
\mathcal{K}_\eta :
 \{ \theta: \Wholes \to \Real \}  \mapsto \mathcal{L}\Big(  BS_\eta (H^1 \times L^2) ;
   BS_{\eta} (H^1\times H^1) \Big)
\end{equation}
that satisfies the following properties.
\begin{enumerate}
\item[(i)] There exists a constant $\delta_{\theta} > 0$
so that for any $\eta_{\min} < \eta < \eta_{\max}$,
any $H \in BS_{\eta}(H^1 \times L^2)$
and any $\theta: \Wholes \to \Real$ for which
$\norm{(S - I) \theta}_{\infty} < \delta_\theta$,
the function $V = \mathcal{K}_{\eta}(\vartheta) H$
satisfies $\Lambda(\theta) V = H$.
\item[(ii)]{
Pick any $\eta_{\min} < \eta < \eta_{\max}$ and $\theta: \Wholes \to \Real$.
Then for any $H \in BS_{\eta}( H^1 \times L^2)$,
the function $V = \mathcal{K}_{\eta}(\theta) H$
satisfies the orthogonality conditions
\begin{equation}
\label{eq:sv:proj:at:zero}
Q^{(1)}_{\theta_0} V_0 =  Q^{(2)}_{\theta_0} V_0 =0 .
\end{equation}
}

\item[(iii)] For any $\eta_{\min} < \eta < \eta_{\max}$ and
$\theta : \Wholes \to \Real$ we have the bound
\begin{equation}
\norm{\mathcal{K}_\eta(\theta)}_{\mathcal{L}\big(BS_{\eta}(H^1 \times L^2) ; BS_{\eta}(H^1 \times H^1) \big)}
   < K .
\end{equation}

\item[(iv)]
Pick a triplet
\begin{equation}
(\eta_1, \eta_2, \eta_3) \in (\eta_{\min}, \eta_{\max})^3
\end{equation}
for which $\eta_1 + \eta_2 \le \eta_3$. Then for any
pair $(\theta^A, \theta^B) \in BS_{\eta_1}(\Real)^2$
and any $H \in BS_{\eta_2}(H^1 \times L^2)$
we have the estimate
\begin{equation}
\norm{\mathcal{K}_{\eta_2} (\theta^A)H -\mathcal{K}_{\eta_2} (\theta^B)H}_{BS_{\eta_3}(H^1 \times H^1)}
  \leq K \norm{\theta^A -\theta^B}_{BS_{\eta_1}(\Real)}
    \norm{H}_{BS_{\eta_2}(H^1 \times L^2) }  .
\end{equation}

\item[(v)] Consider a pair $(\eta_1 , \eta_2) \in (0, \eta_{\max})^2$
together with a function
\begin{equation}
H \in BS_{\eta_1}(H^1 \times L^2) \cap BS_{\eta_2}(H^1 \times L^2).
\end{equation}
Then for any $\theta: \mathbb{Z} \to \mathbb{R}$ we have
\begin{equation}
\mathcal{K}_{\eta_1}(\theta)H=
 \mathcal{K}_{\eta_2}(\theta)H .
\end{equation}

\item[(vi)]
Pick an integer $1 \le m \le r - 1$
together with a triplet
\begin{equation}
(\eta_1, \eta_2, \eta_3) \in (\eta_{\min}, \eta_{\max})^3
\end{equation}
for which $m (\eta_1 + \eta_2) \le \eta_3$.
Then the map
\begin{equation}
BS_{\eta_1}(\Real) \ni \theta \mapsto
\mathcal{K}(\theta)
\in
\mathcal{L}\big( BS_{\eta_2}(H^1 \times L^2) ; BS_{\eta_3}(H^1 \times H^1) \big)
\end{equation}
is $C^m$-smooth.  In addition, for any integer $1 \le p \leq m$,
the derivative $D^{p} \mathcal{K}$ can be seen as a map
\begin{equation}
D^{p} \mathcal{K} : BS_{\eta_1}(\Real)
 \mapsto
 \mathcal{L}^{(p)} \Big( BS_{\eta_1}(\Real)^p ;
   \mathcal{L}\big( BS_{\eta_2}(H^1 \times L^2) ; BS_{\eta}(H^1 \times H^1) \big) \Big)
\end{equation}
for every $\eta \ge p \eta_1 + \eta_2$. This map in continuous in the first
variable if $\eta > p \eta_1 +\eta_2$.
\end{enumerate}
\end{prop}

Our strategy is to exploit the inverses $\mathcal{K}^{\mathrm{cc}}_{\eta}(\vartheta)$
for the constant-coefficient problem \sref{eq:lc:lin:sys:to:solve}
to introduce an approximate inverse
\begin{equation}
\mathcal{K}_\eta^{\mathrm{apx}}(\theta): BS_{\eta}(H^1 \times L^2) \to BS_{\eta}(H^1 \times H^1)
\end{equation}
for $\Lambda(\theta)$ by writing
\begin{equation}
[\mathcal{K}_\eta^{\mathrm{apx}}(\theta)H]_j = \mathrm{pev}_j \mathcal{K}_\eta^{\mathrm{cc}} (\theta_{j})H .
\end{equation}
In order to turn this into an actual inverse,
we need to establish bounds for the remainder term
\begin{equation}
\begin{array}{lcl}
S^{\mathrm{rm}}_\eta(\theta) H
&= &\Lambda(\theta) \mathcal{K}_\eta^{\mathrm{apx}}(\theta)H - H .
\end{array}
\end{equation}

To this end,
we introduce the coordinate projection $\pi_2[v,w] = w$
together with the sequence
\begin{equation}
\begin{array}{lcl}
{[} \Delta^\diamond_s(\theta) H ]_l
 & = & [s^\diamond  \pi_2 \mathcal{K}^{\mathrm{apx}}_{\eta}(\theta) H]_l
 -  [s^\diamond \pi_2 \mathcal{K}^{\mathrm{cc}}_{\eta}(\theta_l) H]_l.
\end{array}
\end{equation}
A short computation shows that
\begin{equation}
\label{eq:sv:id:for:delta:s}
\begin{array}{lcl}
[\Delta^\diamond_s(\theta) H]_l
%& = &
%\Big(
% \big[ \mathcal{K}^\perp_{\theta_{l+1} } H
%   - \mathcal{K}^\perp_{\theta_l}  H\big]_{l+1}
%   + \ldots + \big[ \mathcal{K}^\perp_{\theta_{l+\sigma_B} } H
%   - \mathcal{K}^\perp_{\theta_l} H \big]_{l+\sigma_B}
% , \ldots \Big)
%\\[0.2cm]
& = &
\Big( - \sum_{j=0}^{\sigma_B - 1}
  \big[\mathcal{K}^{\mathrm{cc}}_{\eta}(\theta_{l-j} ) H
   - \mathcal{K}^{\mathrm{cc}}_{\eta}(\theta_l)  H\big]_{l-j},
 \sum_{j=1}^{\sigma_A}
   \big[\mathcal{K}^{\mathrm{cc}}_{\eta}(\theta_{l+j} ) H
   - \mathcal{K}^{\mathrm{cc}}_{\eta}(\theta_l)  H\big]_{l+j},
\\[0.2cm]
& & \qquad \qquad
 \sum_{j=1}^{\sigma_B}
   \big[\mathcal{K}^{\mathrm{cc}}_{\eta}(\theta_{l+j} ) H
   - \mathcal{K}^{\mathrm{cc}}_{\eta}(\theta_l)  H\big]_{l+j},
\\[0.2cm]
& & \qquad \qquad
 -  \sum_{j=0}^{\sigma_A - 1}
  \big[\mathcal{K}^{\mathrm{cc}}_{\eta}(\theta_{l-j} ) H
   - \mathcal{K}^{\mathrm{cc}}_{\eta}(\theta_l)  H\big]_{l-j},
  0
\Big) .
\end{array}
\end{equation}
In a similar spirit,
we introduce the sequences
\begin{equation}
\begin{array}{lcl}
{[} \Delta_{\mathcal{T}}(\theta) H ]_l
 & = & [\mathcal{T}(\theta)  \mathcal{K}^{\mathrm{apx}}_{\eta}(\theta) H]_l
 -  [\mathcal{T}(\mathbf{1} \theta_l) \mathcal{K}^{\mathrm{cc}}_{\eta}(\theta_l) H]_l ,
\\[0.2cm]
{[} \Delta_{\mathcal{D}}(\theta) H ]_l
 & = & [\mathcal{D}(\theta)  \mathcal{K}^{\mathrm{apx}}_{\eta}(\theta) H]_l
 -  [\mathcal{D}( \mathbf{1}\theta_l) \mathcal{K}^{\mathrm{cc}}_{\eta}(\theta_l) H]_l
\end{array}
\end{equation}
and compute
\begin{equation}
\label{eq:sv:id:for:delta:d:t}
\begin{array}{lcl}
[\Delta_{\mathcal{D}}(\theta) H]_l
& = &
\Big(
  \pi_1 [ \mathcal{K}^{\mathrm{cc}}_{\eta}(\theta_{l+1}) H - \mathcal{K}^{\mathrm{cc}}_{\eta}(\theta_l) H ]_{l+1},
   Df (T_{\theta_l} \tau \Phi_*) \overline{\tau} [\Delta^\diamond_s(\theta) H]_l
\Big),
\\[0.2cm]
[\Delta_{\mathcal{T}}(\theta) H]_l
& =&
\Big(
  \pi_2   [ \mathcal{K}^{\mathrm{cc}}_{\eta}(\theta_{l+1}) H - \mathcal{K}^{\mathrm{cc}}_{\eta}(\theta_l) H ]_{l+1},
  0
\Big).
\end{array}
\end{equation}

These computations allows us to obtain the identity
\begin{equation}
\label{eq:sv:id:for:s:rm}
\begin{array}{lcl}
[\mathcal{S}^{\mathrm{rm}}_{\eta}(\theta) ]_l
 & = &
\mathrm{pev}_l \Lambda(\theta) \mathcal{K}_\eta^{\mathrm{apx}}(\theta)H
- \mathrm{pev}_l \Lambda(\mathbf{1} \theta_l) \mathcal{K}^{\mathrm{cc}}_{\eta}( \theta_l) H
\\[0.2cm]
& = &
{[} \Delta_{\mathcal{D}}(\theta) H ]_l
 - [I - P^{(1)}_{\theta_l} ] [ \Delta_{\mathcal{T}}(\theta) H ]_l .
\end{array}
\end{equation}
In order to formulate appropriate bounds for this expression,
we introduce the notation
\begin{equation}
\begin{array}{lcl}
\mathrm{cev}_l \theta
 & = &  \mathrm{ev}_l \theta - \mathbf{1} \theta_l
\\[0.2cm]
 & = &  \big(\theta_{l - \sigma_* + 1} - \theta_l , \ldots, \theta_{l + \sigma_*} - \theta_l \big) .
\end{array}
\end{equation}

\begin{lem}
\label{lem:sv:bnds:point:s:rm}
Assume that (Hf), $(H\Phi)$, (HS1)-(HS3) and (HM) are all satisfied.
Pick a sufficently small constant $\eta_{\max} > 0$
together with a sufficiently large $K > 0$.
Then for any $0 < \eta < \eta_{\max}$ and any $H \in BS_{\eta}(H^1 \times L^2)$,
the following estimates hold.
\begin{itemize}
\item[(i)]{
For any sequence $\theta: \Wholes \to \Real$
we have the bound
\begin{equation}
\label{eq:sv:bnd:s:rm:abs}
\norm{ \mathrm{pev}_l S^{\mathrm{rm}}_{\eta}(\theta) }_{H^1 \times L^2}
  \le K \eta^{-3} e^{\eta \abs{ l }  }\abs{  \mathrm{cev}_l \theta } \norm{H}_{BS_{\eta}(H^1 \times L^2)} .
\end{equation}
}
\item[(ii)]{
For any pair of sequences
$(\theta^A, \theta^B): \Wholes \to \Real^2$
we have the bound
\begin{equation}
\label{eq:sv:bnd:s:rm:diff}
\norm{ \mathrm{pev}_l [
    S^{\mathrm{rm}}_{\eta}(\theta^A)
    - S^{\mathrm{rm}}_{\eta}(\theta^B)
 ]
 }_{H^1 \times L^2}
  \le K \eta^{-3} e^{\eta  \abs{l}  }\abs{  \mathrm{ev}_l (\theta^A - \theta^B) } \norm{H}_{BS_{\eta}(H^1 \times L^2)} .
\end{equation}
}
\end{itemize}
\end{lem}
\begin{proof}
As a consequence of the bound \sref{eq:lin:cc:bnd:k:cc} and the smoothness of the map
$\vartheta \mapsto \mathcal{K}^{\mathrm{cc}}_{\eta}(\vartheta)$,
there exists $C_1 > 0$
for which
\begin{equation}
\label{eq:sv:bnd:diff:k:cc}
\norm{\mathrm{pev}_l \mathcal{K}^{\mathrm{cc}}_{\eta}(\vartheta^A) H
  - \mathrm{pev}_l \mathcal{K}^{\mathrm{cc}}_{\eta}(\vartheta^B) H}_{H^1 \times H^1 }
\le C_1 \eta^{-3} \abs{\vartheta^A - \vartheta^B} e^{\eta \abs{l} }
  \norm{H}_{BS_{\eta}(H^1 \times L^2)}.
\end{equation}
In particular, we see that for all $\abs{j}\le \sigma_*$
we have
\begin{equation}
\norm{ \mathrm{pev}_{l + j}
\big[ \mathcal{K}^{\mathrm{cc}}_{\eta}(\theta_{l + j}) H
  - \mathcal{K}^{\mathrm{cc}}_{\eta}(\theta_l) H
 \big]
 }_{H^1 \times H^1} \le C_2 \eta^{-3} e^{\eta \abs{l} }
   \abs{ \mathrm{cev}_l \theta }
   \norm{H}_{BS_{\eta}(H^1 \times L^2) }.
\end{equation}
The bound \sref{eq:sv:bnd:s:rm:abs}
hence follows immediately
from the representations
\sref{eq:sv:id:for:delta:s},
\sref{eq:sv:id:for:delta:d:t}
and \sref{eq:sv:id:for:s:rm}.

In addition, \sref{eq:sv:bnd:diff:k:cc}
also implies that
\begin{equation}
\label{eq:sv:bnd:diff:k:apx}
\norm{\mathrm{ev}_l
[\mathcal{K}^{\mathrm{apx}}_{\eta}(\theta^A) H
- \mathcal{K}^{\mathrm{apx}}_{\eta}(\theta^B) H ]
}_{\mathbf{H}^1 \times \mathbf{H}^1}
\le
C_3 \eta^{-3}
\abs{\mathrm{ev}_l (\theta^A - \theta^B) }
e^{ \eta \abs{l} }
 \norm{H}_{BS_{\eta}(H^1 \times L^2)}.
\end{equation}
The second bound \sref{eq:sv:bnd:s:rm:diff}
readily follows from this.
\end{proof}

In order to restrict the size of the remainder term,
we need to introduce an appropriate cut-off function.
To this end, we pick an arbitrary $C^{\infty}$-smooth function $\chi: [0, \infty) \to \Real$
that has $\chi(\zeta) =1 $ for $0 \le \zeta \le 1$ and $\chi(\zeta) = 0$ for $\zeta \ge 2$.
For any $\delta > 0$, we subsequently write
$\chi_{\delta}$ for the function
\begin{equation}
\label{eq:defn:chi:delta}
\chi_{\delta}(\zeta) = \chi( \zeta / \delta).
\end{equation}
With this definition in hand,
we pick a constant $\tilde{\delta}_{\vartheta} > 0$
and introduce the
cut-off remainder term
\begin{equation}
[S^{\mathrm{rm};c}_{\eta}(\theta) H]_l =
\chi_{\tilde{\delta}_\theta} \big(\abs{\mathrm{cev}_l \theta} \big )
[S^{\mathrm{rm}}_{\eta}(\theta) H]_l.
\end{equation}
The pointwise estimates in Lemma \ref{lem:sv:bnds:point:s:rm}
immediately yield the following bounds on this
new remainder term.

\begin{cor}
Assume that (Hf), $(H\Phi)$, (HS1)-(HS3) and (HM) are all satisfied.
Pick a sufficently small constant $\eta_{\max} > 0$
together with a sufficiently large $K > 0$.
Then for
any $0 < \eta < \eta_{\max}$ and any $H \in BS_{\eta}(H^1 \times L^2)$
the following estimates are valid.
\begin{itemize}
\item[(i)]{
For  any sequence $\theta: \Wholes \to \Real$ we have the bound
\begin{equation}
\label{eq:sv:bnd:s:rm:abs:full}
\norm{  S^{\mathrm{rm};c}_{\eta}(\theta) }_{BS_{\eta}(H^1 \times L^2) H }
  \le K \tilde{\delta}_\theta \norm{H}_{BS_{\eta}(H^1 \times L^2) }.
\end{equation}
}
\item[(ii)]{
For any $\eta_1 > 0$
and any pair of sequences
$\theta^A, \theta^B \in BS_{\eta_1}(\Real)$,
we have the bound
\begin{equation}
\label{eq:sv:bnd:s:rm:diff:full}
\norm{
    S^{\mathrm{rm};c}_{\eta}(\theta^A)
    - S^{\mathrm{rm};c}_{\eta}(\theta^B)
 }_{BS_{\eta_1 + \eta} (H^1 \times L^2) }
  \le K \eta^{-3} \norm{  \mathrm{ev}_l (\theta^A - \theta^B) }_{BS_{\eta_1}(\Real)}
     \norm{H}_{BS_{\eta}(H^1 \times L^2)} .
\end{equation}
}
\end{itemize}
\end{cor}

After picking $\tilde{\delta}_{\theta} > 0$ to be sufficiently small,
the bound \sref{eq:sv:bnd:s:rm:diff:full}
allows us to define the full inverse
\begin{equation}
\mathcal{K}_{\eta}(\theta) = \mathcal{K}^{\mathrm{apx}}_\eta(\theta)[I+ S^{\mathrm{rm};c}_{\eta}(\theta)]^{-1}.
\end{equation}
We remark that the normalization conditions
\sref{eq:sv:proj:at:zero} hold as a direct consequence of
\sref{eq:lin:cc:orth:cnds} and the construction of $\mathcal{K}^{\mathrm{apx}}_{\eta}$.

\begin{proof}[Proof of Proposition \ref{prp:sl:inverse}]
Items (i), (ii), (iii) and (v) follow directly from the discussion above.
To obtain  (iv) and (vi) one can use
the bounds \sref{eq:sv:bnd:diff:k:apx} and \sref{eq:sv:bnd:s:rm:diff:full};
see e.g. the proof of \cite[Prop. 6.1]{HJHBRGSG}.
\end{proof}

Turning to identify the kernel of $\Lambda(\theta)$,
we introduce the operator
\begin{equation}
\mathcal{E}(\theta): BS_{\eta}(H^1 \times H^1) \to BS_{\eta}(H^1 \times L^2)
\end{equation}
that acts as
\begin{equation}
\mathcal{E}(\theta) [V]
= \Lambda( \theta)[V] - \Lambda( \theta_0 \mathbf{1} )[V]
\end{equation}
for any sequence $\theta: \Wholes \to \Real$.
We now write
\begin{equation}
\label{eq:sv:defns:v:a:b:hom}
\begin{array}{lcl}
V^{A}_{\mathrm{hom}}(\theta)
& = & T_{\theta_0} (\Phi_*', 0)
- \mathcal{K}_{\eta}(\theta) \mathcal{E}(\theta)
       T_{\theta_0}(\Phi_*', 0) \mathbf{1}  ,
\\[0.2cm]
V^{B}_{\mathrm{hom}}(\theta)
& = & T_{\theta_0} ([\partial_z \phi_z]_{z=0},\Phi_*')
- \mathcal{K}_{\eta}(\theta) \mathcal{E}(\theta)
       T_{\theta_0}([\partial_z \phi_z]_{z=0},\Phi_*') \mathbf{1} .
\end{array}
\end{equation}

\begin{proof}[Proof of Proposition \ref{prp:sl:kern}]
Item (i) follows from the fact
that
\begin{equation}
\mathcal{E}(\theta) T_{\theta_0}(\Phi_*', 0)  \mathbf{1}
= \Lambda(\theta) T_{\theta_0} (\Phi_*', 0) \mathbf{1} ,
\end{equation}
together with a similar identity for $V^B_{\mathrm{hom}}(\theta)$.
Item (ii) follows directly
from the normalization \sref{eq:sv:proj:at:zero}.
Finally,  (iii) can be established
by following the proof of \cite[Lem. 6.4]{HJHBRGSG}.
\end{proof}

\section{The center manifold}
\label{sec:cm}

Our goal here is to construct and analyze
a global center manifold
for the system \sref{eq:set:red:sys}
that captures all the solutions
where the pair $(v,w)$ remains small.
In particular, we set out to establish
Proposition \ref{prp:set:cm}.
While the main spirit of the ideas in \cite[{\S}7]{HJHBRGSG}
can be used to establish
the existence of the manifold,
we need to take special care to identify
the reduced equation that is satisfied
on the center space. The key issue
is that we wish to recover
a first order difference equation
from a differential-difference system
of order $2 \sigma_*$.

Let us choose two small constants
$\delta_c > 0$ and $\delta_v > 0$
and recall the constant $\delta_\theta > 0$
defined in Proposition \ref{prp:sl:inverse}.
%
%and write
%\begin{equation}
%I_c = (c_* - \delta_c, c_* + \delta_c) .
%\end{equation}
The main idea is that we look
for solutions to \sref{eq:set:red:sys} that can be written in the form
\begin{equation}
\label{eq:cm:ansatz:v}
\begin{array}{lcl}
\mathrm{ev}_{l} V
& = &
 \Psi^{\alpha}_l T_{ \Psi^\theta_l }
 (\Phi_*', 0) \mathbf{1}
 +
 \Psi^{\beta}_l T_{\Psi^\theta_l }
    \big( [\partial_z \phi_z]_{z=0},\Phi_*' \big)
    \mathbf{1}
%\\[0.2cm]
%& & \qquad
  + h( \Psi_l, c)  ,
\end{array}
\end{equation}
for a triplet of scalar functions
\begin{equation}
\Psi = ( \Psi^\theta, \Psi^\alpha, \Psi^\beta ): \Wholes \to \Real^3
\end{equation}
and a map
\begin{equation}
h: \Real \times (-\delta_v, \delta_v)^2 \times (c_* - \delta_c, c_* + \delta_c)
\to \mathbf{H}^1 \times \mathbf{H}^1
\end{equation}
that satisfies
\begin{equation}
Q_{\vartheta}^{(1)} \mathrm{pev}_0 h(\vartheta, \alpha, \beta)
=
Q_{\vartheta}^{(2)} \mathrm{pev}_0 h(\vartheta, \alpha, \beta)
= 0.
\end{equation}
%for all $\vartheta \in \Real$
%and $(\alpha,\beta) \in (-\delta_v, \delta_v)^2

In order to setup an appropriate fixed point argument
for the pair $(\Psi, h)$, we first
need to add cut-offs to the nonlinearities
defined in \sref{eq:set:defn:nl:s} and \sref{eq:set:defn:nl:r}.
In particular,
for any triplet
$\psi = (\psi_\theta, \psi_{\alpha}, \psi_{\beta}) \in \Real^3$,
we
recall the cut-off function \sref{eq:defn:chi:delta}
and introduce the notation
\begin{equation}
\mathcal{I}(\psi,h) \in
\mathbf{H}^1 \times \mathbf{H}^1
\end{equation}
for the functions
\begin{equation}
\mathcal{I}(\psi;h)
=
\chi_{\delta_v}( \abs{\psi_\alpha} + \abs{\psi_\beta} )
\Big[
\psi_\alpha %\mathrm{vev}_0
   T_{\psi_\theta}(\Phi_*',0) \mathbf{1}
  + \psi_\beta %\mathrm{vev}_0
  T_{ \psi_\theta}\big( [\partial_z \phi_z]_{z=0},\Phi_*' \big) \mathbf{1}
  + h( \psi  , c)
\Big].
\end{equation}
This allows us to introduce
the new nonlinearities
$\mathcal{R}^c$ and $\mathcal{S}^c$
that act as
\begin{equation}
\begin{array}{lcl}
\big[\mathcal{R}^c( \Psi ; h) \big]_l
& = &
  \chi_{\delta_\theta}(\abs{\Psi^\theta_{l+1} - \Psi^\theta_l} )
     \mathcal{R}\big( \mathrm{ev}_l \Psi^\theta,
       \mathcal{I}(\Psi_l;h ) \big) ,
\\[0.2cm]
\big[\mathcal{S}^c(\Psi; h) \big]_l
& = &
  \chi_{\delta_\theta}(\abs{\Psi^\theta_{l+1} - \Psi^\theta_l} )
     \mathcal{S}( \mathrm{ev}_l \Psi^\theta,
          \mathcal{I}(\Psi_l;h ) \big) .
\end{array}
\end{equation}

%[todo: deze constant meteen $\delta_\theta$ noemen???]
%Recall the constant $\delta >0 $
%defined in Proposition \ref{prp:sl:inverse},
%and assume that
Let us now assume that
we have
a solution to \sref{eq:set:red:sys} of the form
\sref{eq:cm:ansatz:v}
for some $\Psi \in BS_{\eta}(\Real^3)$
that satisfies
\begin{equation}
\label{eq:cm:apriori:assump:on:V:Psi}
\norm{(S-I) \Psi^\theta}_\infty < \delta_\theta,
\qquad
\norm{ \Psi^{\alpha}}_\infty + \norm{\Psi^{\beta}}_\infty
< \delta_v.
\end{equation}
Then upon writing
$\psi = \Psi(0)$, we must have
\begin{equation}
\begin{array}{lcl}
V & = & \psi_{\alpha} V^A_{\mathrm{hom}}(\Psi^\theta)
+ \psi_{\beta} V^B_{\mathrm{hom}}(\Psi^\theta)
%\\[0.2cm]
%& & \qquad
+ \mathcal{K}_{\eta} (\Psi^{\theta})
\Big[
(c - c_*)(0 , T_{\Psi^{\theta}} \Phi_*' )
+  \mathcal{R}^c( \Psi; h )
\Big] .
\end{array}
\end{equation}
Upon introducing
the sequence
\begin{equation}
(\alpha, \beta)_l =
  \big( Q^{(1)}_{\Psi^\theta_l} ,
      Q^{(2)}_{\Psi^\theta_l} \big) V_l,
\end{equation}
item (iii) of Proposition
\ref{prp:sl:kern}
implies that
\begin{equation}
\label{eq:w:lz:plus:l}
\begin{array}{lcl}
V_{l_0 + l}
 & = &
  \alpha_{l_0} \mathrm{pev}_l
    V^A_{\mathrm{hom}}(S_{l_0} \Psi^\theta)
  + \beta_{l_0} \mathrm{pev}_l
    V^B_{\mathrm{hom}}(S_{l_0} \Psi^\theta)
\\[0.2cm]
& & \qquad
  + \mathrm{pev}_l \mathcal{K} ( S_{l_0} \Psi^\theta )
       \Big[
         (c - c_*) (0, T_{S_{l_0} \Psi^\theta} \Phi_*')
       + \mathcal{R}^c( S_{l_0} \Psi ; h )
       \Big]
\\[0.2cm]
\end{array}
\end{equation}
for any pair $(l_0, l) \in \Wholes^2$.

We now introduce the notation
\begin{equation}
\begin{array}{lcl}
\mathcal{J}_{V;\alpha}(\Psi^\theta)
& = &
  -
  \mathcal{K}_{\eta}(\Psi^\theta) \mathcal{E}(\Psi^\theta)
    T_{\Psi^\theta_0} (\Phi_*', 0) \mathbf{1} ,
\\[0.2cm]
\mathcal{J}_{V;\beta}(\Psi^\theta)
& = &
  - \mathcal{K}_{\eta}(\Psi^\theta)
       \mathcal{E}(\Psi^\theta)
          T_{\Psi^\theta_0}
            \big( [\partial_z \phi_z]_{z=0}, \Phi_*' \big) \mathbf{1} ,
\\[0.2cm]
\end{array}
\end{equation}
together with
\begin{equation}
\begin{array}{lcl}
\mathcal{A}_V(\Psi^\theta)
& = &
\mathcal{K}_{\eta} ( \Psi^\theta )
      (0, T_{\Psi^\theta} \Phi_*') ,
\\[0.2cm]
\mathcal{M}_V(\Psi; h)
& = &
\mathcal{K}_{\eta} ( \Psi^\theta )
           \mathcal{R}^c( \Psi; h ).
\end{array}
\end{equation}
The representation \sref{eq:sv:defns:v:a:b:hom}
allows us to
rewrite
\sref{eq:w:lz:plus:l}
as
\begin{equation}
\label{eq:cm:v:lz:plus:l:final}
\begin{array}{lcl}
V_{l_0 + l}
 & = &  \alpha_{l_0} T_{\Psi^\theta_{l_0}} (\Phi_*', 0)
   + \beta_{l_0} T_{\Psi^\theta_{l_0}}\big( [\partial_z \phi_z]_{z=0} , \Phi_*'  \big)
   + \alpha_{l_0} \mathrm{pev}_l  \mathcal{J}_{V;\alpha}(S_{l_0} \Psi^\theta)
   + \beta_{l_0} \mathrm{pev}_l  \mathcal{J}_{V;\beta}(S_{l_0} \Psi^\theta)
\\[0.2cm]
& & \qquad
   +   (c-c_*) \mathrm{pev}_l \mathcal{A}_V(S_{l_0} \Psi^\theta)
     + \mathrm{pev}_l  \mathcal{M}_V(S_{l_0} \Psi; h).
\end{array}
\end{equation}

For any $y \in \Real^n$ we now introduce the
auxilliary cut-off function
\begin{equation}
\overline{\chi}_{\delta}(y)
= \chi_{\delta}(\abs{y}) y ,
\end{equation}
together with the notation
\begin{equation}
\begin{array}{lcl}
\mathcal{J}_{\alpha \beta}(\Psi^\theta)[\alpha, \beta]
& = &
  \overline{\chi}_{\delta_v}(\alpha)
  \overline{\chi}_{\delta_\theta}\big(
    \big[Q_{\Psi^\theta_1}- Q_{\Psi^\theta_0} \big]
    T_{\Psi^\theta_0}(\Phi_*' , 0)
   \big)
\\[0.2cm]
& & \qquad
  + \overline{\chi}_{\delta_v}(\beta)
  \overline{\chi}_{\delta_\theta}\big(
    \big[Q_{\Psi^\theta_1}- Q_{\Psi^\theta_0} \big]
    T_{\Psi^\theta_0}\big( [\partial_z \phi_z]_{z=0},\Phi_*' \big)
   \big)
\\[0.2cm]
& & \qquad
  + \overline{\chi}_{\delta_v}( \alpha )
       \overline{\chi}_{\delta_\theta}\big(
          Q_{\Psi^\theta_1} \mathrm{pev}_1
    \mathcal{J}_{V;\alpha}(\Psi^\theta)
       \big)
\\[0.2cm]
& & \qquad
  + \overline{\chi}_{\delta_v}( \beta )
       \overline{\chi}_{\delta_\theta}\big(
          Q_{\Psi^\theta_1} \mathrm{pev}_1
    \mathcal{J}_{V;\beta}(\Psi^\theta)
       \big)
\end{array}
\end{equation}
and also
\begin{equation}
\label{eq:cm:def:A:alpha:beta}
\begin{array}{lcl}
\mathcal{A}_{\alpha \beta}(\Psi^\theta)
& = & Q_{\Psi^\theta_1} \mathrm{pev}_1 \mathcal{A}_V(\Psi^\theta) ,
\\[0.2cm]
\mathcal{M}_{\alpha \beta}(\Psi;h )
& = &
Q_{\Psi^\theta_1 }
\mathrm{pev}_1 \mathcal{M}_V(\Psi; h).
\end{array}
\end{equation}

As a reminder, there exists a constant $K > 0$ so that
\begin{equation}
\norm{\mathcal{E}(\Psi^\theta)  T_{\Psi^\theta_0}(\Phi_*', 0)
     }_{BS_{\eta}(H^1 \times L^2) }
\le K \big[ \norm{(S - I) \Psi^\theta}_\infty L_0 + e^{-\eta L_0} \big] .
\end{equation}
holds for any $L_0 \ge 1$.
In particular, if we pick a sufficiently small $\tilde{\epsilon} > 0$
and strengthen the assumption
\sref{eq:cm:apriori:assump:on:V:Psi}
by demanding
\begin{equation}
\norm{(S-I) \Psi^\theta}_\infty + \norm{\alpha}_\infty
+ \norm{\beta}_\infty < \tilde{\epsilon},
\end{equation}
then in fact
\begin{equation}
\label{eq:cm:alpha:beta:eq}
\begin{array}{lcl}
(\alpha, \beta)_{l_0 + 1}
   & = &
   (\alpha, \beta)_{l_0 }
   +  \mathcal{J}_{\alpha \beta}(S_{l_0} \Psi^\theta )
     [\alpha_{l_0} , \beta_{l_0} ]
     + (c-c_*) \mathcal{A}_{\alpha \beta}(S_{l_0} \Psi^\theta)
     + \mathcal{M}_{\alpha \beta}(S_{l_0} \Psi; h)
\end{array}
\end{equation}
holds for every $l_0 \in \mathbb{Z}$.

Inspired by the identity \sref{eq:set:sys:for:theta:diff},
we introduce the expressions
\begin{equation}
\label{eq:cm:defns:j:a:w:theta}
\begin{array}{lcl}
\mathcal{J}_{\theta}(\Psi^\theta)[\alpha, \beta]
& = &
    + \overline{\chi}_{\delta_v}( \alpha )
       \overline{\chi}_{\delta_\theta}\big(
          Q^{(2)}_{\Psi^\theta_0} \mathrm{pev}_1
    \mathcal{J}_{V;\alpha}(\Psi^\theta)
       \big)
\\[0.2cm]
& & \qquad
  + \overline{\chi}_{\delta_v}( \beta )
       \overline{\chi}_{\delta_\theta}\big(
          Q^{(2)}_{\Psi^\theta_0} \mathrm{pev}_1
    \mathcal{J}_{V;\beta}(\Psi^\theta)
       \big) ,
\\[0.2cm]
\mathcal{A}_{\theta}(\Psi^\theta)
 & = &
   Q^{(2)}_{\Psi^\theta_0}
   \mathrm{pev}_1
          \mathcal{A}_V(\Psi^\theta),
\\[0.2cm]
\mathcal{M}_\theta(\Psi; h)
& = &
      Q^{(2)}_{\Psi^{\theta}_0}
   \mathrm{pev}_1
     \mathcal{M}_V( \Psi ; h)
   +
  \mathrm{pev}_0 \mathcal{S}^c(\Psi;h) .
\end{array}
\end{equation}
This allows us to augment the difference equation
\sref{eq:cm:alpha:beta:eq} for the pair $(\alpha,\beta)$
by introducing a new sequence $\theta: \Wholes \to \Real$
that must satisfy
\begin{equation}
\label{eq:cm:theta:eq}
\begin{array}{lcl}
\theta_{l_0 +1} - \theta_{l_0}
 & = &
   \beta_{l_0}
   + \mathcal{J}_{\theta}(S_{l_0} \Psi^\theta) [ \alpha_{l_0}, \beta_{l_0} ]
   + (c - c_*) \mathcal{A}_{\theta}(S_{l_0} \Psi^\theta )
   + \mathcal{M}_\theta(S_{l_0} \Psi; h) .
\end{array}
\end{equation}

In view of the original
Ansatz \sref{eq:cm:ansatz:v},
it is natural to look
for solutions
to \sref{eq:cm:alpha:beta:eq} and \sref{eq:cm:theta:eq}
that have $(\theta,\alpha, \beta ) = \Psi$.
This leads to the system
\begin{equation}
\label{eq:cm:fixpnt:prb}
\begin{array}{lcl}
\mathrm{pev}_{l_0} (S - I) (\Psi^{\alpha}, \Psi^{\beta})
 & = & \mathcal{J}_{\alpha \beta}(S_{l_0} \Psi^\theta )
   \mathrm{pev}_{l_0} [\Psi^\alpha, \Psi^\beta]
  + (c - c_*) \mathcal{A}_{\alpha \beta}(S_{l_0} \Psi^\theta)
  + \mathcal{M}_{\alpha \beta}(S_{l_0} \Psi; h) ,
\\[0.2cm]
\mathrm{pev}_{l_0} (S - I) \Psi^\theta
 & = & \beta_{l_0} + \mathcal{J}_{\theta}(S_{l_0} \Psi^\theta )
   \mathrm{pev}_{l_0} [\Psi^\alpha, \Psi^\beta]
  + (c - c_*) \mathcal{A}_{\theta}(S_{l_0} \Psi^\theta)
  + \mathcal{M}_{\theta}(S_{l_0} \Psi; h) .
\end{array}
\end{equation}
In addition, we can combine
\sref{eq:cm:ansatz:v}
and
\sref{eq:cm:v:lz:plus:l:final}
to arrive at
\begin{equation}
\label{eq:cm:fixpnt:prb:h}
\begin{array}{lcl}
h(\psi, c)
& = &  \psi_\alpha \mathrm{ev}_0 \mathcal{J}_{V;\alpha}(\Psi)
+ \psi_\beta \mathrm{ev}_0 \mathcal{J}_{V; \beta}(\Psi)
\\[0.2cm]
& & \qquad
+ (c - c_*) \mathrm{ev}_0 \mathcal{A}_W(\Psi)
+ \mathrm{ev}_0 \mathcal{M}_V(\Psi; h),
\end{array}
\end{equation}
which can be seen as a consistency
condition for the function $h$.
Our main result here states
that the global
center manifold for \sref{eq:set:red:sys}
can be constructed by solving the fixed point problem
generated by
\sref{eq:cm:fixpnt:prb}-\sref{eq:cm:fixpnt:prb:h}.

\begin{prop}
\label{prp:cm:act:cm:res}
Assume that (Hf), $(H\Phi)$, (HS1)-(HS3) and (HM) are all satisfied
and pick two sufficiently small constants
$0 < \eta_{\min} < \eta_{\max}$.
In addition, pick a sufficiently large constant $K > 0$
together with a sufficiently small constant
$\delta_v > 0$ and write
\begin{equation}
\delta_c = \delta_v^2, \qquad \delta_\theta =  \delta_v^{4/7}.
\end{equation}
Then there exist $C^{r-1}$-smooth maps
\begin{equation}
\begin{array}{lclcl}
\Psi_*  & : &
\Real^3 \times (c_* - \delta_c, c_* + \delta_c) & \to & BS_{\eta_{\min} }(\Real^3) ,
\\[0.2cm]
h_* & : & \Real \times (-\delta_v, \delta_v)^2 \times (c_* - \delta_c, c_* + \delta_c)
  & \to &
  \mathbf{H}^1 \times \mathbf{H}^1
\end{array}
\end{equation}
so that the following properties hold true.
\begin{itemize}
\item[(i)]{
  Pick any $\psi \in \Real^3$ and $c \in (c_* - \delta_c, c_* + \delta_c)$.
  Then the identity \sref{eq:cm:fixpnt:prb:h}
  holds upon writing
  \begin{equation}
  \Psi = \Psi_*(\psi,c ), \qquad h = h_* .
  \end{equation}
}
\item[(ii)]{
  Pick any $\psi \in \Real^3$ and $c \in (c_* - \delta_c, c_* + \delta_c)$.
  Then for any %sufficiently small $\eta > 0$,
   $\eta \in (\eta_{\min}, \eta_{max})$,
  the function $\Psi= \Psi_*(\psi, c)  $
  is the unique solution in $BC_{\eta}(\Real; \Real^3)$ to the problem
  \sref{eq:cm:fixpnt:prb}
  with $\mathrm{pev}_0 \Psi = \psi$ and $h = h^*$.
}
\item[(iii)]
Pick any $\psi \in \Real^3$ and $c \in (c_* - \delta_c, c_* + \delta_c)$
together with a pair $(l_0, l) \in \Wholes^2$.
Then we have the identity
\begin{equation}
\label{eq:cm:shift:for:psi:star}
\begin{array}{lcl}
\mathrm{pev}_{l + l_0} \Psi_*( \psi , c)
& = &  \mathrm{pev}_l \Psi_*\big(
       \mathrm{pev}_{l_0} \Psi_*(\psi, c) , c
     \big) .
\end{array}
\end{equation}
\item[(iv)]{
  We have $Q^{(1)} \mathrm{pev}_0 h_* = Q^{(2)} \mathrm{pev}_0 h_* = 0$
  and we have
  \begin{equation}
    \norm{\mathrm{pev}_0 \, h_*(\alpha, \beta, \theta)}_{H^1 \times H^1}
      \le K \big[ \abs{c - c_*} + \alpha^2 + \beta^2 \big] .
  \end{equation}
}
\item[(v)]
Pick any $\psi \in \Real^3$ and $c \in (c_* - \delta_c, c_* + \delta_c)$ and suppose that
\begin{equation}
\norm{ \Psi^\alpha_*( \psi, c) }_\infty + \norm{ \Psi^\beta_* (\psi, c) }_{\infty}
  < \delta_v .
\end{equation}
Then upon writing $\Psi = \Psi_*(\psi,c )$ together with
\begin{equation}
V_l = T_{\theta_l}(\Phi_*',0) \Psi^{\alpha}_l
  + T_{\theta_l}\big( [\partial_z \phi_z]_{z=0},\Phi_*' \big) \Psi^{\beta}_l
  + \mathrm{pev}_0 h_*( \Psi_l  , c),
\end{equation}
the pair $(\Psi^\theta, V)$
is a solution to the full system
\sref{eq:set:red:sys}.

\item[(vi)]
Consider two sequences
\begin{equation}
\theta: \Wholes \to \Real,
\qquad
V: \Wholes \to H^1 \times H^1
\end{equation}
that satisfy \sref{eq:set:red:sys} and admit the bound
\begin{equation}
\norm{\mathrm{pev}_l V}_{H^1 \times H^1} \le \delta_v
\end{equation}
for all $l \in \Wholes$.
Then
upon writing
\begin{equation}
\alpha_l =  Q^{(1)}_{\theta_l} \mathrm{pev}_l V,
\qquad
\beta_l = Q^{(2)}_{\theta_l} \mathrm{pev}_l V,
\end{equation}
we have
\begin{equation}
V_l = \alpha_l T_{\theta_l} (\Phi_*', 0)
+ \beta_l T_{\theta_l} \big( [\partial_z \phi_z]_{z=0} , \Phi_*' , 0 \big)
 + h_* ( \theta_l, \alpha_l, \beta_l , c)
\end{equation}
for all $l \in \Wholes$.
\end{itemize}
\end{prop}
\begin{proof}
In view of the bounds for $\mathcal{R}$
and $\mathcal{S}$ obtained
in Corollary \ref{cor:set:bnds:full:r:s}
and the properties of  $\mathcal{K}_{\eta}$ described in
Proposition %\ref{prp:sl:kern}-
\ref{prp:sl:inverse},
the arguments used
to establish
\cite[Lem. 7.2-7.5]{HJHBRGSG}
can also be applied to construct
solutions of the fixed point problem
\sref{eq:cm:fixpnt:prb}-\sref{eq:cm:fixpnt:prb:h}
and to show that these solutions admit the stated
properties.
\end{proof}

Upon demanding that
\begin{equation}
Q^{(1)}_{\theta_0} V_0 = 0
\end{equation}
any solution to \sref{eq:set:red:sys}
by construction has $Q^{(1)}_{\theta_l} V_l = 0$
for all $l \in \Wholes$.
In particular, imposing the initial condition $\alpha_0 = 0$
and using
\sref{eq:cm:shift:for:psi:star}
to eliminate
the nonlocal
terms appearing
in $\mathcal{J}_{\alpha \beta}$,
$\mathcal{A}_{\alpha \beta}$ and $\mathcal{M}_{\alpha \beta}$
and their counterparts for $\theta$,
the
problem \sref{eq:cm:fixpnt:prb} with $h = h^*$
can be reformulated in the form
\begin{equation}
\begin{array}{lcl}
\beta_{l+1} - \beta_l
 & = &
 f_\beta( \beta, c) ,
\\[0.2cm]
\theta_{l+1} - \theta_l
 & = &
   f_\theta( \beta, c).
\end{array}
\end{equation}
The fact that $f_\beta$ and $f_\theta$
do not depend on $\theta$
is a direct consequence of the shift-invariance
of the system.
We have the expansions
\begin{equation}
\begin{array}{lcl}
f_\beta(\beta, c) & = &
\nu_1 ( c - c_*) +
 \nu_2 \beta_l^2 +
 O\big( (c -c_*)^2 + ( c- c_*) \beta + \beta^3 \big) ,
\\[0.2cm]
f_{\theta}(\beta, c) & = &
\beta_l + \nu_3 ( c - c_*)
   + O \big( \beta_l^2 + (c - c_*)^2 \big) .
\end{array}
\end{equation}

We first set out to compute $\nu_1$ and $\nu_3$,
which by inspection of
the definitions \sref{eq:cm:def:A:alpha:beta}
and \sref{eq:cm:defns:j:a:w:theta}
for $\mathcal{A}_{\alpha \beta}$ and $\mathcal{A}_{\theta}$
can be seen to satisfy
%readily yields
%the identity
\begin{equation}
\nu_1 = \nu_3 = Q^{(2)}_0
  \mathrm{pev}_1 \mathcal{A}_V(0) .
\end{equation}
Using the identity
\sref{eq:inhom:cc:id:for:k:cc},
we may compute
\begin{equation}
\begin{array}{lcl}
\mathcal{A}_{V}( 0)
& = & \mathcal{K}_{\eta}(0)(0, \Phi_*')
\\[0.2cm]
& = & \mathcal{K}^{\mathrm{cc}}_{\eta}(0)(0, \Phi_*')
\\[0.2cm]
& = & [I - P^{(1)}_0] \mathcal{K}^{\mathrm{up}}_{\eta}(0) (0, \Phi_*') .
\end{array}
\end{equation}
This latter expression can be evaluated
by a direct computation involving an Ansatz that is polynomial in $l$.
This is performed in the next result, which directly implies that
\begin{equation}
\label{eq:cm:expr:for:nu:1:3}
\nu_1 = \nu_3 % Q^{(2)} \mathrm{pev}_1 \mathcal{A}_W(0)
= 2 [\partial_z^2 \lambda_z]_{z=0}^{-1} .
\end{equation}

\begin{lem}
Assume that (Hf), $(H\Phi)$, (HS1)-(HS3) and (HM) are all satisfied.
Then there  exists $v_* \in H^1$
for which we have
\begin{equation}
\mathrm{pev}_l \mathcal{K}^{\mathrm{up}}_{\eta}(0) (0, \Phi_*') \mathbf{1}
= 2 [\partial_z^2 \lambda_z]_{z=0}^{-1}
\Big[
  \frac{1}{2}  l^2 \big(\Phi_*', 0 \big)
+  l \big([\partial_z \phi_z]_{z=0} , \Phi_*'\big)
+ \Big(v_*, [\partial_z \phi_z]_{z=0}  - \frac{1}{2} \Phi_*' \Big)
+ \frac{1}{2} V_{\mathrm{hom}}^B
\Big].
\end{equation}
In particular, we have
\begin{equation}
Q^{(2)}_0 \mathrm{pev}_1 \mathcal{K}^{\mathrm{up}}_{\eta}(0) (0, \Phi_*') \mathbf{1}
=  2 [\partial_z^2 \lambda_z]_{z=0}^{-1}.
\end{equation}
\end{lem}
\begin{proof}
For convenience, we introduce the polynomial sequences
\begin{equation}
p^{(0)}_l = 1,
\qquad
p^{(1)}_l = l
\qquad
p^{(2)}_l = l^2 .
\end{equation}
For any $w \in H^1$,
we can compute
\begin{equation}
\mathrm{pev}_l s^\diamond[ p^{(0)} w  ]
= s^\diamond_0 w,
\end{equation}
together with
\begin{equation}
\begin{array}{lcl}
\mathrm{pev}_l s^\diamond[ p^{(1)} w]
& = &
  l \big( - \sigma_B , \sigma_A, \sigma_B , - \sigma_A) w
\\[0.2cm]
& & \qquad
 + \big(\sum_{j=0}^{\sigma_B - 1} l,
  \sum_{j=1}^{\sigma_A} l,
  \sum_{j=1}^{\sigma_B} l,
  \sum_{j=0}^{\sigma_A - 1} l
  \big) w
\\[0.2cm]
& = &
  l s^\diamond_0 w
 + [\partial_z s^\diamond_z]_{z=0} w.
\end{array}
\end{equation}
In particular, for any pair $(v,w) \in H^1 \times H^1$
we find
\begin{equation}
\begin{array}{lcl}
\mathrm{pev}_l \mathcal{D}(0) [ p^{(0)} (v, w) ]
& = & (0, A_1 w) ,
\\[0.2cm]
\end{array}
\end{equation}
together with
\begin{equation}
\begin{array}{lcl}
\mathrm{pev}_l \mathcal{D}(0) [ p^{(1)} (v, w) ]
& = & l \big(0, A_1 w)
+   \big(v,  \frac{1}{2}(A_1 + A_2)  w \big) .
\\[0.2cm]
\end{array}
\end{equation}
Finally, we have
\begin{equation}
\mathrm{pev}_l \mathcal{D}(0) [ p^{(2)} ( v, 0) ]
= (2l + 1) (v, 0).
\end{equation}

Upon writing
\begin{equation}
W_l = \frac{1}{2}  l^2 (\Phi_*', 0)
+  l \Big( [\partial_z \phi_z]_{z=0} , \Phi_*' \Big)
+ (v_*,   w_*) , %\phi_0^{(1)} - \frac{1}{2} \Phi_*')
\end{equation}
we may hence compute
\begin{equation}
\mathrm{pev}_l[\mathcal{D}(0) W]
= \frac{1}{2}  (2 l + 1) (\Phi_*', 0)
+  l (0, A_1 \Phi_*')
+  \big([\partial_z \phi_z]_{z=0} ,
     \frac{1}{2}(A_1 + A_2) \Phi_*' \big)
+  (0 , A_1 w_*),
\end{equation}
together with
\begin{equation}
\begin{array}{lcl}
\mathrm{pev}_l[\mathcal{T}(0) W]
 &= &
\Big( w_* + ( l+1) \Phi_*',
-  l \mathcal{L}_* [\partial_z \phi_z]_{z=0}
+ \mathcal{L}_* v_* \Big)
\\[0.2cm]
& = &
\Big( w_* + ( l+1) \Phi_*',
+  l A_1 \Phi_*'
- \mathcal{L}_* v_*  \Big).
\end{array}
\end{equation}
In particular, we find
\begin{equation}
\mathrm{pev}_l[\mathcal{D}(0) W - \mathcal{T}(0) W ]
= \Big( [\partial_z \phi_z]_{z=0}
-  \frac{1}{2} \Phi_*'
- w_*
,
\frac{1}{2}(A_1 + A_2) \Phi_*'
+ A_1 w_* + \mathcal{L} v_*
\Big) .
\end{equation}

Upon writing
\begin{equation}
w_* = [\partial_z \phi_z]_{z=0} - \frac{1}{2} \Phi_*'
\end{equation}
and picking a constant $\mu \in \Real$,
we hence see that
\begin{equation}
 \mathrm{pev}_l [ \mathcal{D}(0) W - \mathcal{T}(0) W ] = \mu (0, \Phi_*')
\end{equation}
holds if and only if
\begin{equation}
\mathcal{L} v_*
+  \frac{1}{2}  A_2 \Phi_*'
+  A_1 [\partial_z \phi_z]_{z=0}
= \mu \Phi_*'.
\end{equation}
This latter equation has a solution
$v_* \in H^1$  if and only if
\begin{equation}
\mu = \langle \psi_*,
\frac{1}{2} A_2 \Phi_*' + A_1[\partial_z \phi_z]_{z=0}
\rangle
= \frac{1}{2} [\partial_z^2 \lambda_z]_{z=0} .
\end{equation}
We hence see that
\begin{equation}
\mathcal{K}^{\mathrm{up}}_{\eta}(0) (0, \Phi_*') \mathbf{1}
= \frac{2}{[\partial_z^2 \lambda_z]_{z=0} } W
+ \nu_A V_{\mathrm{hom}}^A
+ \nu_B V_{\mathrm{hom}}^B
\end{equation}
for some pair $(\nu_A, \nu_B) \in \Real^2$
that ensures that
\begin{equation}
Q^{(1)}_0 \mathrm{pev}_0 \mathcal{K}^{\mathrm{up}}_{\eta}(0) (0, \Phi_*') \mathbf{1}
= Q^{(2)}_0 \mathrm{pev}_0\mathcal{K}^{\mathrm{up}}_{\eta}(0) (0, \Phi_*') \mathbf{1}
= 0 .
\end{equation}
Upon imposing the normalization
$Q_0 v_* = 0$,
the desired result follows by noting
that
\begin{equation}
Q^{(1)}_0 \mathrm{pev}_0 W =
Q_0 v_* = 0,
\qquad
Q^{(2)}_0 \mathrm{pev}_0 W
= Q_0 w_* = - \frac{1}{2}.
\end{equation}
\end{proof}

We now set out to compute the coefficient
$\nu_2$ by examining the zeroes
of the functions $f_\beta$ and $f_\theta$.
In particular, recalling
the direction-dependent waves $(c_\varphi, \Phi_\varphi)$
defined in Lemma \ref{lem:mr:angl:dep},
we introduce the sequences
\begin{equation}
\Xi^{(\varphi)}: \Wholes \to H^1
\end{equation}
that are given by
\begin{equation}
\Xi^{(\varphi)}_{ l}(\xi)
= \Phi_\varphi\Big( \cos\varphi[\xi + l \tan \varphi ] \Big)
\end{equation}
for any small $\abs{\varphi}$.
This allows us to rewrite the planar wave solutions
\sref{eq:mr:ansatz:u:nl:varphi} in the form
\begin{equation}
\begin{array}{lcl}
u_{nl}(t) & = &
 \Phi_{\varphi}( n \cos \varphi   + l \sin \varphi l + c_\varphi t )
\\[0.2cm]
& = &
 \Phi_{\varphi}(\cos \varphi [ n + \tan\varphi l + d_\varphi t] )
\\[0.2cm]
& = &
\Xi^{(\varphi)}_{l}(n + d_\varphi t).
\end{array}
\end{equation}

We now pick a constant $\theta^{(\varphi)}_0 $ in such a way that
\begin{equation}
\int_{-\infty}^\infty \psi_*(\xi) \big[ \Phi_{\varphi}( \xi\cos \varphi )
- \Phi_*(\xi + \theta^{(\varphi)}_0 ) \big] \, d \xi   = 0,
\end{equation}
which is possible for small $\abs{\varphi}$
on account of the normalization $\langle \psi_*, \Phi_*' \rangle = 1$.
In addition, we introduce
three sequences
\begin{equation}
(\theta^{(\varphi)}, v^{(\varphi)},w^{(\varphi)}) : \Wholes \to \Real \times H^1 \times H^1
\end{equation}
by writing
\begin{equation}
\theta^{(\varphi)}_l  =
    \theta^{(\varphi)}_0 + l \tan \varphi ,
\end{equation}
together with
\begin{equation}
\begin{array}{lcl}
v^{(\varphi)}_{l}
  & = & T_{\theta^{(\varphi)}_l} \Big[
                \Phi_{\varphi}( \cos \varphi[\cdot - \theta^{(\varphi)}_0] )
                - \Phi_*
        \Big] ,
\\[0.2cm]
w^{(\varphi)}_l
  & = & T_{ \theta^{(\varphi)}_{l} } \Big[
     \Phi_{\varphi}\Big( \cos \varphi [ \cdot - \theta^{(\varphi)}_0   ] \Big)
   - \Phi_{\varphi}\Big( \cos \varphi [ \cdot - \theta^{(\varphi)}_0
      -   \tan \varphi  ] \Big)
   \Big] .
\\[0.2cm]
\end{array}
\end{equation}
By construction, we have
\begin{equation}
Q_{\theta^{(\varphi)}_l} v^{(\varphi)}_l = 0 ,
\end{equation}
together with
\begin{equation}
\begin{array}{lcl}
\Xi^{(\phi)}_{l} & = & T_{\theta^{(\varphi)}_l}\Phi_* + v^{(\varphi)}_l ,
\\[0.2cm]
\Xi^{(\varphi)}_{l}
-\Xi^{(\varphi)}_{l-1}
& = & w^{(\varphi)}_l.
\end{array}
\end{equation}
Applying Lemma \ref{lem:set:formulations:eqv},
we hence see that
the triplet $(\theta^{(\varphi)}, v^{(\varphi)}, w^{(\varphi)})$
is a solution to the
problem \sref{eq:set:red:sys}
with $c = d_\varphi$.

Upon writing
\begin{equation}
\overline{\beta}_{\varphi}
= Q_0\Big[
     \Phi_{\varphi}\Big( \cos \varphi [ \cdot - \theta^{(\varphi)}_0   ] \Big)
   - \Phi_{\varphi}\Big( \cos \varphi [ \cdot - \theta^{(\varphi)}_0
     -   \tan \varphi  ] \Big)
   \Big] ,
\end{equation}
we see that
\begin{equation}
Q_{\theta^{(\varphi)}_l } w_l = \overline{\beta}_{\varphi}
\end{equation}
for all $l \in \Wholes$.
Whenever $\abs{\varphi}$
is sufficiently small,
item (vi) of Proposition \ref{prp:cm:act:cm:res}
now implies that
\begin{equation}
f_{\beta}(\overline{\beta}_{\varphi}, d_\varphi )
= 0 ,
\end{equation}
together with
\begin{equation}
f_{\theta}(\overline{\beta}_{\varphi}, d_\varphi )
= \tan \varphi.
\end{equation}
In view of the expansion
\begin{equation}
\overline{\beta}_{\varphi} = \varphi + O (\varphi^2) ,
\end{equation}
we find
\begin{equation}
\frac{1}{2} \nu_1 [\partial_\varphi^2 d_\varphi]_{\varphi=0}
+ \nu_2 = 0
\end{equation}
and hence
\begin{equation}
\label{eq:cm:expr:for:nu:2}
\nu_2 = -  [\partial_z^2 \lambda_z]_{z=0}^{-1}
    [\partial_\varphi^2 d_\varphi]_{\varphi=0} .
\end{equation}

\begin{proof}[Proof of Proposition \ref{prp:set:cm}]
We define the function $h$ via
\begin{equation}
h(\theta,\kappa,c) =h_*(\theta,0,\kappa,c).
\end{equation}
The results now follow directly
from Proposition \ref{prp:cm:act:cm:res}
together with the expressions
\sref{eq:cm:expr:for:nu:1:3} and \sref{eq:cm:expr:for:nu:2}
for $\nu_1$ and $\nu_2$.
\end{proof}

\bibliographystyle{klunumHJ}
%\bibliography{ref}

\end{document}